\documentclass{article}
\usepackage{amsmath, amssymb}
\usepackage{amsthm}
\usepackage{amsfonts}
\usepackage{ascmac}
\usepackage{geometry}
\usepackage{braket}
\usepackage{bm}
\usepackage{bbm}
\usepackage{color}
\usepackage{tikz-cd}
\usepackage{graphicx}  

\usetikzlibrary{fit, positioning}
\usepackage{enumitem}
\usepackage{url}
\usepackage{ulem} 
\usepackage[colorlinks=true, linkcolor=blue, citecolor=green]{hyperref}
\numberwithin{equation}{section}
\theoremstyle{definition}
\newtheorem{thm}{Theorem}[section]
\newtheorem{definition}[thm]{Definition}
\newtheorem{ex}[thm]{Example}
\newtheorem{lem}[thm]{Lemma}
\newtheorem{cor}[thm]{Corollary}
\newtheorem{prop}[thm]{Proposition}

\newtheorem{rem}[thm]{Remark}
\newtheorem{cons}[thm]{Construction}
\newtheorem{claim}{Claim}

\newcommand{\cone}{\mathring{c}}

\title{Classification of locally standard $T$-pseudomanifolds over topological stratified pseudomanifolds}
\date{}

\begin{document}

\author{
  Yuya Koike\thanks{Department of Pure and Applied Mathematics, Graduate School of Information Science and Technology, The University of Osaka, Suita, Osaka 565-0871, Japan. Email: \texttt{u175137d@ecs.osaka-u.ac.jp}} 
  \and 
  Shintar\^o Kuroki\thanks{Department of Applied Mathematics, Faculty of Science, Okayama University of Science, 1-1 Ridai-Cho Kita-Ku Okayama-shi, Okayama 700-0005, Japan. Email: \texttt{kuroki@ous.ac.jp}}
}

\maketitle

\begin{abstract}
We introduce the notion of a {\it locally standard $T$-pseudomanifold}, a class that generalizes both complete toric varieties and locally standard $T$-manifolds.
The main goal of this paper is to show that locally standard $T$-pseudomanifolds over topological stratified pseudomanifolds satisfying certain conditions are completely classified, up to (weakly) equivariant homeomorphism, by their {\it characteristic data}.
This result extends the classification of quasitoric manifolds by Davis–Januszkiewicz.
\end{abstract}

\tableofcontents

\section{Introduction}

The interplay between algebraic geometry and combinatorics is most elegantly manifested in the theory of toric varieties.
A {\it toric variety} is a normal (complete) algebraic variety equipped with an action of an algebraic torus $(\mathbb{C}^{\ast})^{n}$ having a dense orbit, and it is well-known that such varieties are classified by combinatorial objects, called {\it fans}.
In the realm of topology, the analogous correspondence was proved by Davis and Januszkiewicz \cite{DJ91}, who introduced {\it quasitoric manifolds}, i.e., smooth $2n$-dimensional manifolds with a locally standard $T^{n}$-action whose orbit space is a simple convex polytope.
Their seminal work showed that the equivariant homeomorphism types of quasitoric manifolds are classified by {\it characteristic pairs}, consisting of the orbit polytope and a characteristic function.

Since then, {\it toric topology} has expanded in several directions to encompass broader classes of spaces and orbit geometries.
For example, Masuda and Panov \cite{MP06} generalized the classification to locally standard torus manifolds over manifolds with corners.
Conversely, from the perspective of algebraic geometry, even singular objects like projective toric varieties can be classified up to $T^n$-equivariant homeomorphism using characteristic pairs on convex polytopes, see \cite{Jor98}.

However, a fundamental gap remains in the unified treatment of these objects.
While simple polytopes and manifolds with corners appear in the setting of manifolds with locally standard $T^{n}$-actions, general convex polytopes that arise from singular toric varieties do not necessarily possess the structure of a manifold with corners.
These orbit spaces are unified as the structure of a topological stratified pseudomanifold.

\subsection{Locally standard $T$-pseudomanifolds and the main theorem}

In this paper, we introduce and classify the notion of {\it locally standard $T$-pseudomanifolds}.
This new class of spaces provides a common framework that simultaneously generalizes complete toric varieties and locally standard $T$-manifolds.
By extending the {\it locally standard} condition to the setting of topological stratified pseudomanifolds, we capture a vast range of spaces with singularities that are not necessarily topological manifolds. 

The main contribution of this work is a complete classification of these spaces.
Specifically, for a locally standard $T$-pseudomanifold $X$ over a topological stratified pseudomanifold $Q$, we define the {\it characteristic data $(Q,\lambda, c)$}.
Here, $\lambda$ is a {\it characteristic functor} that assigns isotropy data to strata of all dimensions, generalizing the classical characteristic function, and $c\in H^{2}(Q;\mathbb{Z}^{m})$ is a cohomology class representing the Chern class of the $T$-principal bundle over the top strata. 
More precisely, the following theorem is the main theorem of this paper:

\begin{thm}[classification theorem (Theorem~\ref{classification theorem})]\label{classification theorem intro}
Let $X$, $X'$ be locally standard $T$-pseudomanifolds, and $Q$, $Q'$ be their orbit spaces, respectively. Assume that the set of top strata of $Q$ (resp. $Q'$) is homotopy equivalent to $Q$ (resp. $Q'$) itself. Then, the following two statements are equivalent:
\begin{enumerate} \item The characteristic data $(Q,\lambda,c)$ of $X$ and the characteristic data $(Q',\lambda',c')$ of $X'$ are {\it (weakly) isomorphic}; \item $X$ and $X'$ are (weakly) $T$-equivariantly homeomorphic.
\end{enumerate}
\end{thm}

Our main theorem asserts that under a natural homotopy condition on the orbit space, the (weakly) equivariant homeomorphism type of a locally standard $T$-pseudomanifold is uniquely determined by its characteristic data.
This result not only unifies the classification of quasitoric manifolds, locally standard torus manifolds, and toric varieties but also extends the reach of toric topology to more general spaces with singularities.

\subsection{Structure of the paper}
The structure of this paper is as follows.
In Section~\ref{Notation}, we introduce the notation and assumptions used throughout the paper.
In particular, we define a {\it filtration by orbit dimension} on spaces with a torus action,
and impose three conditions on this filtration.
In Section~\ref{section T-pseudomanifold}, we introduce the {\it locally standard $T$-pseudomanifold} and show that it is invariant under weakly $T$-equivariant homeomorphisms (see Proposition~\ref{inv of T-pseudomanifold}).
In Section~\ref{section 4}, we show that the orbit space of a locally standard $T$-pseudomanifold
admits a {\it topological stratified pseudomanifold} structure.
It also prepares for the definition of an invariant called the {\it characteristic data}.

In Section~\ref{section 5}, we introduce the notion of a {\it characteristic functor}, which generalizes the {\it characteristic function} introduced by Davis and Januszkiewicz.
Furthermore, by abstracting the characteristic data of locally standard $T$-pseudomanifolds described in Section~\ref{section 4}, we define characteristic data on topological stratified pseudomanifolds.
In Section~\ref{section weak isomorphism}, we introduce the notion of (weak) isomorphism between them.
We also prove that the characteristic data are invariants under (weakly) equivariant homeomorphisms
of locally standard $T$-pseudomanifolds (see Proposition~\ref{prop in 4}).
In Section~\ref{easy examples}, we present basic examples.

In Section~\ref{section canonical model}, we construct a {\it canonical model} from given characteristic data and describe its basic properties.
The construction of the canonical model is very natural as a topological model of spaces with a group action (see also \cite{DJ91}, \cite{MP06}, \cite{BH99}, \cite{SS21}, \cite{Yos11}, \cite{KK25}, among others).
However, it is a nontrivial problem to show that the canonical model admits the structure of a locally standard $T$-pseudomanifold; this will be established in Section~\ref{section 12}.
In Section~\ref{section 9}, we prove that if two characteristic data are (weakly) isomorphic, then the corresponding canonical models are also (weakly) equivariantly homeomorphic (see Theorem~\ref{thm of equivariant homeo of canonical model}).

In Section~\ref{section: model space}, we introduce the notion of a {\it model space}, which is needed for the proof of the main theorem.
The construction of this space is inspired by \cite{Ayz18}.
We then study its basic properties, and in particular, show that it is equivariantly homeomorphic to the canonical model (see Lemma~\ref{lem a}).
In Section~\ref{section 11}, we state and prove the classification theorem, which is the main result of this paper.
The proof proceeds via the model space:
we first show that any locally standard $T$-pseudomanifold is equivariantly homeomorphic to the corresponding model space (see Lemma~\ref{lem b}),
and then combine this with Lemma~\ref{lem a}.
This implies that locally standard $T$-pseudomanifolds are classified by canonical models (see Lemma~\ref{main thm}).
Furthermore, by the uniqueness of canonical models (see Theorem~\ref{thm of equivariant homeo of canonical model}),
it follows that they correspond to characteristic data.

In Section~\ref{section 12}, we show that the canonical model constructed from characteristic data admits the structure of a locally standard $T$-pseudomanifold (see Theorem~\ref{canonical model is a T-pseudomanifold}).
Combining this property with the invariance of locally standard $T$-pseudomanifolds under weakly equivariant homeomorphisms (see Proposition~\ref{inv of T-pseudomanifold}),
it follows that any space whose equivariant homeomorphism type is realized by a canonical model admits the structure of a locally standard $T$-pseudomanifold.
Using this fact, we conclude the following results.

\begin{thm}
\label{class of T-spaces}
The class of locally standard $T$-pseudomanifolds includes the following:

\begin{itemize}
    \item complete toric varieties (see Section~\ref{sec: applications});
    \begin{itemize}
        \item in particular, projective toric varieties;
    \end{itemize}
    \item compact locally standard $T$-manifolds (see Section~\ref{sec: applications});
    \begin{itemize}
        \item in particular, locally standard torus manifolds;
        \item more specifically, quasitoric manifolds;
    \end{itemize}
    \item quasitoric orbifolds (see Subsection~\ref{subsec: characteristic function on simple polytope});
    \begin{itemize}
        \item in particular, quasitoric manifolds;
    \end{itemize}
    \item moment-angle manifolds (see Remark~\ref{rem: moment-angle manifold});
    \item locally $k$-standard $T$-manifolds (see Remark~\ref{rem: locally k-standard T-manifold}).
\end{itemize}
\end{thm}

This theorem demonstrates that the introduction of locally standard $T$-pseudomanifolds enables a unified treatment of both complete toric varieties and compact locally standard $T$-manifolds.

\begin{rem}
In the proof of this theorem, one must verify that the orbit spaces of these spaces admit the structure of a topological stratified pseudomanifold.
For instance, although it may be a routine exercise for experts to show that a convex polytope has such a structure, we were unable to find a reference containing a complete proof. Therefore, we include a proof in this paper (see Proposition~\ref{polytope is pseudomanifold}).
\end{rem}

In Sections~\ref{section: unit sphere} and~\ref{sec: applications}, we show that the {\it unit sphere with the standard torus action}, {\it complete toric varieties}, and {\it compact locally standard $T$-manifolds} admit the structure of a locally standard $T$-pseudomanifold.
Finally, in Section~\ref{last ex}, we provide an example that is neither a complete toric variety
nor a locally standard $T$-manifold.
This example is the Thom space of a complex line bundle over a Hirzebruch surface,
which has a pyramid as its orbit space.

In Appendix~\ref{app: topology}, we prepare a topological lemma needed for the proof of the main theorem.
In Appendix~\ref{BB}, we review basic notions of {\it topological stratified pseudomanifolds}.
In Appendix~\ref{sec: polytope}, we provide a detailed proof that convex polytopes admit such a structure.
We also explain how to construct a {\it characteristic functor}
from a {\it characteristic function} on a simple convex polytope.

\vskip\baselineskip

Regarding Theorem~\ref{class of T-spaces}, the correspondence between the classes of spaces with a $T$-action and their orbit spaces is summarized in Table~\ref{space and orbit space}.

\begin{table}[htbp]
  \centering
\begin{tabular}{|c|c|}
  \hline
  Topological space with $T$-action  & $T$-orbit space \\
  \hline
  Complete toric variety & Closed ball \\
  Projective toric variety & Convex polytope \\
  Locally standard $T$-manifold & Manifold with corners \\
    Quasitoric orbifold & Simple polytope \\
  Quasitoric manifold & Simple polytope \\
 Moment-angle manifold & Simple polytope \\
 Locally $k$-standard $T$-manifold & Simple polytope \\
 Locally standard $T$-pseudomanifold & Topological stratified pseudomanifold \\
  \hline
\end{tabular}
  \caption{The classes of spaces with a $T$-action and their orbit spaces}
  \label{space and orbit space}
\end{table}

\subsection{Acknowledgements}  

The second author was supported by JSPS KAKENHI Grant Number 21K03262. This work was supported by the Research Institute for Mathematical Sciences, an International Joint Usage/Research Center located at Kyoto University.

\section{Notation and assumption}\label{Notation}

Let $X$ be a second-countable, compact Hausdorff space. Note that $X$ is not necessarily connected.
Recall that a {\it manifold stratified space} is defined as follows:
\begin{definition}[manifold stratified space {\cite[Chapter 2]{Fri20}}]\label{def of manifold stratified space}
A {\it manifold stratified space} is a stratified space equipped with a filtration
\[
\mathfrak{X}: X=X_m \supsetneq X_{m-1} \supset \cdots \supset X_{l}\supsetneq X_{l-1} = \emptyset, \quad (m \ge l \ge 0)
\]
such that $X_i \setminus X_{i-1}$ (for $m \ge i \ge l$) is an $i$-dimensional topological manifold (possibly empty).
\end{definition}
When $X_i \setminus X_{i-1} \neq \emptyset$, each connected component of $X_i \setminus X_{i-1}$ is called a {\it stratum} of dimension $i$.
In particular, a stratum of $X_m \setminus X_{m-1}$ is called a {\it top stratum}.
The space $X_i$ is called the $i$-{\it skeleton}.
The index $l$ indicates the lowest dimension of the non-empty skeletons.
We call $(m-l)$ the {\it length} of $\mathfrak{X}$ and write $m=l+n$ with $n \ge 0$ (i.e., $n$ denotes the length).
If $X_i \setminus X_{i-1}=\emptyset$, we omit $X_i$ from the filtration. In this abbreviated notation, the indices of the skeletons correspond to their dimensions.
A manifold stratified space $X$ equipped with a specified filtration $\mathfrak{X}$ is also denoted by $(X, \mathfrak{X})$.

Throughout this paper, we use the following three notations:
\begin{itemize}
\item the circle group $U(1):=\{z\in\mathbb{C}\ |\ |z|=1\}$ that acts on $\mathbb{C}$ by scalar multiplication;
\item $\mathbb{C}^{\times}:=\mathbb{C}\setminus \{0\}$; 
\item $T^{m}\cong U(1)^{m}$ is an $m$-dimensional torus. We often denote it by $T$ when the dimension of $T$ is clear from the context.
\end{itemize}

Suppose $T^m$ acts continuously on $X$.
This action induces a {\it filtration by orbit dimension}, defined as follows.
Given integers $l$, $n$ with $l\geq 0$ and $0\leq n\leq m$, we define the following subsets of $X$: for an integer $i$ ($0\le i\le n$),
\begin{align}\label{dimension set}
X_{l+2i+1+(m-n)}=X_{l+2i+(m-n)}:=\{x\in X\ |\ \dim T(x)\le i+(m-n)\},
\end{align}
where $T(x)$ is the $T^m$-orbit of $x\in X$.
The $T^m$-action on $X$ induces the following filtration: 
\begin{align}\label{filtration by orbit dimension}
\mathfrak{X}: X=X_{l+m+n}\supset X_{l+2(n-1)+(m-n)}\supset \cdots \supset X_{l+2i+(m-n)}\supset \cdots \supset X_{l+(m-n)}\supset \emptyset.
\end{align}
We call this filtration $\mathfrak{X}$ {\it the filtration by orbit dimension}.
\begin{rem}\label{assumption of subtorus}
In this paper, we assume that the isotropy subgroup $T_{x}$ is a subtorus (i.e., a connected closed subgroup of $T^{m}$) for all $x\in X$.
\end{rem}

For each $0 \le i \le n$, since $X_{l+2i+(m-n)}$ is $T$-invariant, the orbit projection $\pi: X \to Q:=X/T$ induces the following filtration $\mathfrak{X}/T$ of the orbit space:
\begin{align}\label{orbit filtration}
\mathfrak{X}/T: Q= Q_{l+n} \supset Q_{l+n-1} \supset \cdots \supset Q_{l+i}
\supset \cdots \supset Q_{l} \supset \emptyset,
\end{align}
where $Q_{l+i} := X_{l+2i+1+(m-n)}/T=X_{l+2i+(m-n)}/T $ for $0 \leq i \leq n$.
We call $(X/T, \mathfrak{X}/T)$ the {\it filtered orbit space}, often denoted by $(Q, \mathfrak{Q})$. The given integers $l$ and $n$ determine the lowest index of skeletons of $Q$ and the length of $\mathfrak{Q}$, respectively.

In this paper, we assume the following three conditions:
\begin{enumerate}[label=(Cond.\arabic*), ref=Cond.\arabic*]
\item \label{cond-1}
$(X, \mathfrak{X})$ is a manifold stratified space (see Definition \ref{def of manifold stratified space});
\item \label{cond-2}
$X_{l-2+(m-n)}=X_{l-1+(m-n)}:=\{x\in X \mid  \dim T(x)\le (m-n)-1 \} = \emptyset$, i.e., $i=-1$;
\item \label{cond-3}
we assume $X_{l+m+n} \supsetneq X_{l+m+n-2}$ (i.e., $X$ has free orbits, hence the action is effective).
\end{enumerate}

\begin{rem}
The assumption that $X$ is second-countable is needed for the proof of Lemma \ref{lem b}, which is a technical lemma required to prove Theorem \ref{classification theorem intro}.
\end{rem}

\section{Locally standard $T$-pseudomanifold}\label{section T-pseudomanifold}
In this section, we  first give the definition of a {\it locally standard $T$-pseudomanifold}, which is the main object of study in this paper.
Subsequently, we show that the property of being a locally standard $T$-pseudomanifold is preserved under weakly $T$-equivariant homeomorphisms (see Proposition \ref{inv of T-pseudomanifold}).

Before defining locally standard $T$-pseudomanifolds, we introduce the concept of an {\it open cone}, which plays a crucial role in the local structure around singular points.

In this paper, for $Y\subset X$, the notation $X/Y$ denotes the topological space obtained by collapsing $Y$ to a point in $X$, where the topology on $X/Y$ is defined by the quotient topology induced by the natural projection $X\to X/Y$.

\begin{definition}[open cone {{\cite[Definition 2.1.1]{Max19}}}]\label{def of open cone}
Let $L$ be a compact Hausdorff space. An {\it open cone on $L$} is defined by
\[
\mathring{c}(L):=L \times [0, 1) / (L \times \{0\}).
\] 
Namely, $\mathring{c}(L)$ is the quotient space of $L\times [0,1)$ obtained by collapsing the subspace $L\times \{0\}$ to a single point. We assume that $\mathring{c}(\emptyset)$ is a single point. We call the point $[v,0]\in \cone(L)$ a {\it cone vertex}, where $v$ is any element in $L$.
\end{definition}

\subsection{Definition of locally standard $T$-pseudomanifold}
Let $X$ be a second-countable, compact Hausdorff space with a continuous $T^{m}$-action, and $\mathfrak{X}$ denote the filtration by orbit dimension (see \eqref{filtration by orbit dimension}).
We define a locally standard $T$-pseudomanifold as follows.

\begin{definition}\label{def of T-pseudomanifold}
An {\it $(l+m+n)$-dimensional locally standard $T$-pseudomanifold} $(X, \mathfrak{X})$ is defined by induction on $(l+n)$ if it satisfies the following:
\begin{itemize}
\item for $n = 0$ and $l = 0$, $X$ is a disjoint union of $m$-dimensional tori equipped with the multiplicative $T^{m}$-action (note that the compactness of $X$ implies that this is a finite disjoint union of $T^{m}$);

\item for $n\ge 0$ and $l \ge 0$, with $l+n\neq 0$, the following two conditions hold:
\begin{enumerate}
    \item The top strata $X_{l+m+n} \setminus X_{l+m+n-2}$ is dense in $X$;

\item 
 For any $x \in X_{l+2i+(m-n)} \setminus X_{l+2(i-1)+(m-n)}$, there exists a triple $(U_x, L_x, \varphi_x)$ such that
\begin{enumerate}
  \item $U_x \subset X$ is a $T$-invariant open neighborhood of $x$;

  \item $L_x$ is a ($2n-2i-1$)-dimensional compact locally standard $T_x$-pseudomanifold, possibly empty, equipped with an action of $T_x \cong T^{n-i}$, where $T_x$ is the isotropy subgroup of $x$ ($L_x$ is called a {\it link} of $x$);

 \item  $\varphi_x : U_x \to O^l \times \big( \Omega \times U(1)^{m-n} \big) \times \mathring{c}({L}_x)$
is a weakly equivariant homeomorphism,
where $O^l \subset \mathbb{R}^l$ is a contractible open subset and $\Omega \subset (\mathbb{C}^{\times})^i$ is a $U(1)^i$-invariant open subset.
The torus
$T^m$ acts on $ O^l \times \big( \Omega \times U(1)^{m-n} \big) \times \mathring{c}({L}_x)$
via an isomorphism
\[ T^m \cong T^m/{T_x} \times T_x \cong U(1)^{i+(m-n)} \times T^{n-i}, \]
where $U(1)^{i+(m-n)}$ acts on $\Omega \times U(1)^{m-n}$
by the standard multiplication (i.e., the free action), and the $T_{x}$-action on $\mathring{c}(L_{x})=L_{x}\times [0,1)/L_{x}\times \{0\}$ is induced by the $T_{x}$-action on the $L_{x}$-factor described in (b) and is trivial on the $[0,1)$-factor.
\end{enumerate}
\end{enumerate}
Here, by the definition of an open cone, $\cone(L_x)$ has a natural filtration by
\begin{align}\label{通常のcone filtration}
(\mathring{c}(L_{x}))_{j+1}:=\mathring{c}((L_{x})_{j}), \quad (-1\le j\le 2n-2i-1).
\end{align}
\end{itemize}
\end{definition}

\begin{rem}\label{restriction to free part}
The weakly equivariant homeomorphism $\varphi_x$ preserves orbit dimensions.
Hence, we have the restriction as a weakly equivariant homeomorphism
\[
\varphi_x |_{U_x \cap X_{l+2i+(m-n)} \setminus X_{l+2(i-1)+(m-n)}}:
U_x \cap X_{l+2i+(m-n)} \setminus X_{l+2(i-1)+(m-n)} \to O^l \times \big( \Omega \times U(1)^{m-n} \big) \times \mathring{c}(\emptyset).
\]
\end{rem}

\begin{rem}
The notion of a locally standard $T$-pseudomanifold $(X, \mathfrak{X})$ is a special case of the $G$-pseudomanifold discussed in \cite{Pop00}, where $G$ is a compact Lie group.
Then, by \cite[Theorem 2.10]{Pop00}, $(X, \mathfrak{X})$ is a {\it topological stratified pseudomanifold} (see Appendix \ref{BB} and also \cite{Fri20} for details).
\end{rem}

\begin{rem}\label{rem: locally k-standard T-manifold}
  The notion of a locally standard $T$-pseudomanifold is a generalization of that of a {\it locally $k$-standard $T$-manifold} introduced in \cite{SS21}.
In fact, in the above definition, if $X$ is a smooth manifold with $l=0$ and the link of each point is a sphere, then $X$ is a locally $k$-standard $T$-manifold.
\end{rem}

\subsection{Invariance of locally standard $T$-pseudomanifold under weakly equivariant homeomorphism}
We show that the property of being a locally standard $T$-pseudomanifold is preserved under weakly $T$-equivariant homeomorphisms. That is, if a space is weakly $T$-equivariantly homeomorphic to a locally standard $T$-pseudomanifold, then it is also a locally standard $T$-pseudomanifold.

\begin{prop}\label{inv of T-pseudomanifold}
Let $X$ and $X'$ be second-countable, compact Hausdorff spaces with a continuous $T^{m}$-action. Let $\mathfrak{X}$ and $\mathfrak{X}'$ be the filtrations by orbit dimension \eqref{filtration by orbit dimension}, respectively. Suppose $(X, \mathfrak{X})$ is an $(l+m+n)$-dimensional locally standard $T$-pseudomanifold. If $X$ and $X'$ are weakly $T^m$-equivariantly homeomorphic, then $(X', \mathfrak{X}')$ is also a locally standard $T$-pseudomanifold.
\end{prop}

\begin{proof}
Let $f:X \to X'$ be a weakly $T$-equivariant homeomorphism. Since $f$ preserves orbits of dimension $(i+(m-n))$, we have
\[
f(X_{l+2i+(m-n)})=X'_{l+2i+(m-n)}=\{ x \in X' \mid \dim T(x) \le i+(m-n) \}, \quad (0 \leq i \leq n).
\]
If $n=0$ and $l=0$, then $X\cong T^{m}\sqcup \cdots \sqcup T^{m}$, and hence $X' \cong T^m \sqcup \cdots \sqcup T^m$ as well.
Therefore, we may assume that $n\ge 0$ and $l \ge 0$ with $l+n \neq 0$. 
We will verify Definition \ref{def of T-pseudomanifold}-1 and Definition \ref{def of T-pseudomanifold}-2.

Suppose $X'_{l+2i+(m-n)} \setminus X'_{l+2(i-1)+(m-n)} \neq \emptyset$. The weakly $T$-equivariant homeomorphism $f$ maps connected components of $X_{l+2i+(m-n)} \setminus X_{l+2(i-1)+(m-n)}$ to connected components of $X'_{l+2i+(m-n)} \setminus X'_{l+2(i-1)+(m-n)}$. Therefore, each connected component of $X'_{l+2i+(m-n)} \setminus X'_{l+2(i-1)+(m-n)}$ can be written as $f(S)$, where $S \subset X_{l+2i+(m-n)} \setminus X_{l+2(i-1)+(m-n)}$ is a stratum. 

Since $X$ is a locally standard $T$-pseudomanifold, by Definition \ref{def of T-pseudomanifold}-1, $X_{l+m+n} \setminus X_{l+m+n-2}$ is dense. Because $f$ is a homeomorphism, it is in particular surjective. Therefore,
\[
f(X_{l+m+n} \setminus X_{l+m+n-2})=f(X_{l+m+n}) \setminus f(X_{l+m+n-2}) = X'_{l+m+n} \setminus X'_{l+m+n-2}
\]
is also dense. This shows that $(X', \mathfrak{X}')$ satisfies Definition \ref{def of T-pseudomanifold}-1.

We next show that $(X', \mathfrak{X}')$ satisfies Definition \ref{def of T-pseudomanifold}-2. In what follows, we use the notation from Definition \ref{def of T-pseudomanifold}. Take any point $f(x) \in X'_{l+2i+(m-n)} \setminus X'_{l+2(i-1)+(m-n)}$ (for some $x \in X_{l+2i+(m-n)} \setminus X_{l+2(i-1)+(m-n)}$).
Since $X$ is a locally standard $T$-pseudomanifold, we can take $(U_x, L_x, \varphi_x)$ satisfying Definition \ref{def of T-pseudomanifold}-2. Define
\[
(U_{f(x)}, L_{f(x)}, \varphi_{f(x)}):=(f(U_x), L_x, \varphi_x \circ f^{-1}).
\]
Since $f$ is a weakly $T$-equivariant homeomorphism, $U_{f(x)}$ is a $T$-invariant open neighborhood of $f(x)$.
The weak $T$-equivariance of $f$ and commutativity of $T$ imply that $T_x$ coincides with $T_{f(x)}$, up to an automorphism of $T$.
Therefore, we may regard $L_{x}$ as a locally standard $T_{f(x)}$-pseudomanifold.
Then, the composition
\[
\varphi_x \circ f^{-1} : U_{f(x)} \to O^l \times \big( \Omega \times U(1)^{m-n} \big) \times \mathring{c}({L}_x)
\]
is weakly equivariant homeomorphism, where the $T^m$-action on  $O^l \times \big( \Omega \times U(1)^{m-n} \big) \times \mathring{c}({L}_x)$ via an isomorphism
\[
T^m \cong T^m/{T_x} \times T_x \cong T^m/{T_{f(x)}} \times T_{f(x)} \cong U(1)^{i+(m-n)} \times T^{n-i}.
\]
This shows that $(X', \mathfrak{X}')$ satisfies Definition~\ref{def of T-pseudomanifold}-2.
\end{proof}

\section{Orbit space of locally standard $T$-pseudomanifold}\label{section 4}
Let $(X,\mathfrak{X})$ be an $(l+m+n)$-dimensional locally standard $T$-pseudomanifold, and let $\pi:X \to Q$ denote the orbit projection.
Then the filtered orbit space $(Q, \mathfrak{Q})$ (see \eqref{orbit filtration}) is a topological stratified pseudomanifold (see Proposition \ref{structure of orbit space} and Definition \ref{def of pseudomanifold} for the precise definition of a topological stratified pseudomanifold).
Furthermore, we show that the orbit projection allows us to define the {\it characteristic data} $(Q, \lambda, c)$ of $X$, which consists of the triple:
\begin{itemize}
\item
the topological stratified pseudomanifold $Q = X/T$ such that the top strata $Q\setminus Q_{l+n-1}$ is homotopy equivalent to $Q$ (see Proposition \ref{structure of orbit space} and \eqref{condition of top strata});

\item
a correspondence $\lambda$ assigning to each stratum its isotropy group, which is a subtorus (see Lemma \ref{lem of trivial bundle});

\item
a {\it Chern class} $c \in H^2(Q; \mathbb{Z}^m)$ of the locally standard $T^m$-pseudomanifold $X$ (see Definition \ref{def of Chern class}).
\end{itemize}

We first show that the orbit space of a locally standard $T$-pseudomanifold is a topological stratified pseudomanifold.

\begin{prop}\label{structure of orbit space}
Let $(X, \mathfrak{X})$ be a locally standard $T$-pseudomanifold. Then its filtered orbit space $(Q, \mathfrak{Q})$ is a compact, Hausdorff and second-countable topological stratified pseudomanifold.
\end{prop}

\begin{proof}
We first verify that $Q=X/T$ is a compact, Hausdorff and second-countable space.  
Since $X$ is compact and Hausdorff, and the action of the compact group $T$ on $X$ is continuous, the orbit space $Q$ is also compact and Hausdorff.
For second-countability, note that the orbit projection $\pi: X \to Q$ is continuous, open and surjective, and $X$ is second-countable.
Therefore, $Q$ is also second-countable.

We now verify that $(Q, \mathfrak{Q})$ is a topological stratified pseudomanifold.
By \cite[Theorem 3.4]{Pop00}, the orbit space of a $G$-pseudomanifold is a topological stratified pseudomanifold (where $G$ is a compact Lie group).
Since a locally standard $T$-pseudomanifold is a special case of a $G$-pseudomanifold with $G = T$, it follows that $(Q, \mathfrak{Q})$ is a topological stratified pseudomanifold.
\end{proof}

The following lemma shows that the isotropy subgroup $T_x$ of a point $x \in X$ depends only on the stratum containing $\pi(x)$.

\begin{lem}\label{lem of trivial bundle}
Let $X$ be a locally standard $T$-pseudomanifold, and $\pi: X \to Q$ be the orbit projection.
For each stratum $S \subset Q_{l+i} \setminus Q_{l+i-1}$ ($0 \leq i \leq n$), all points in $\pi^{-1}(S)$ have the same isotropy subgroup.
\end{lem}
\begin{proof}
By \eqref{orbit filtration}, we have $\pi^{-1}({S}) \subset X_{l+2i+(m-n)}\setminus X_{l+2(i-1)+(m-n)}$.
For any $x \in \pi^{-1}(S)$, by Definition \ref{def of T-pseudomanifold}-2 (b), we have $T_x \cong T^{n-i}$.
Take a $T$-invariant open neighborhood $U_x$ of $x$ that is connected.
Then we have $U_x \cap \pi^{-1}(S) = U_x \cap X_{l+2i+(m-n)} \setminus X_{l+2(i-1)+(m-n)}$.
By Remark \ref{restriction to free part}, we have
\[
U_x \cap X_{l+2i+(m-n)} \setminus X_{l+2(i-1)+(m-n)} \cong O^l \times \left( \Omega \times U(1)^{m-n} \right) \times \cone(\emptyset).
\]
By Definition \ref{def of T-pseudomanifold}-2~(c), $T/T_x$ acts freely on $O^l \times \left( \Omega \times U(1)^{m-n} \right) \times \cone(\emptyset)$, which implies that $T_x$ fixes $U_x \cap \pi^{-1}(S)$.
For any $x' \in \pi^{-1}(S)$, consider a $T$-invariant open neighborhood $U_{x'}$ of $x'$ such that $U_{x'} \subset U_x$.
Since $T_x$ fixes $U_{x'}$, we have $T_x \subset T_{x'}$.
By Definition \ref{def of T-pseudomanifold}-2~(b), $T_{x'} \cong T^{n-i}$.
Since both $T_x$ and $T_{x'}$ are subtori of $T$ with the same dimension, we conclude that $T_{x'}=T_x$.
Therefore, for any $x, y \in \pi^{-1}(S)$, if $U_x \cap U_y \neq \emptyset$, then $T_x=T_y$.
Since $\pi^{-1}(S)$ is connected, by considering a covering by $T$-invariant open neighborhoods, we conclude that all points in $\pi^{-1}(S)$ have the same isotropy subgroup.
\end{proof}

\begin{rem}\label{T_S}
By this lemma, for any $x \in X$, the isotropy group $T_x$, which is a subtorus, depends only on the stratum $S$ containing $\pi(x)$. Therefore, $T_x$ is often denoted by $T_{S}$ or $T_p$ for $p \in S$, depending on the context.
\end{rem}

We next describe {\it the Chern class of the locally standard $T$-pseudomanifold $X$}.
We prepare the following lemma.
\begin{lem}\label{lem: top strata is paracompact}
The set of top strata $Q\setminus Q_{l+n-1}$ is second-countable, paracompact and Hausdorff.
\end{lem}
\begin{proof}
By \cite[Lemma 2.3.13 and Lemma 2.3.15]{Fri20}, $Q\setminus Q_{l+n-1}$ is locally compact and Hausdorff. Moreover, since $Q$ is second-countable (see Proposition~\ref{structure of orbit space}), its subspace $Q\setminus Q_{l+n-1}$ is also second-countable.
Since every second-countable locally compact Hausdorff space is paracompact, it follows that $Q\setminus Q_{l+n-1}$ is second-countable, paracompact and Hausdorff.
\end{proof}

By Definition~\ref{def of T-pseudomanifold}-2~(b) and Remark~\ref{assumption of subtorus}, for any $x \in \pi^{-1}(Q\setminus Q_{l+n-1})$, the isotropy group $T_x = \{ 1 \}$ is the identity group.
Therefore, the map
\[
\pi|_{\pi^{-1}(Q\setminus Q_{l+n-1})} : \pi^{-1}(Q\setminus Q_{l+n-1}) \to Q\setminus Q_{l+n-1}
\]
is a $T^m$-principal bundle over $Q\setminus Q_{l+n-1}$.
The isomorphism classes of $T^m$-principal bundles over $Q\setminus Q_{l+n-1}$ are classified by the homotopy set $[Q\setminus Q_{l+n-1}, BT^m]$ (see \cite[Theorem 3.5.1]{RS17}), where $BT^m \cong (\mathbb{C}P^{\infty})^{m}$ is the classifying space of $T^m$.
We have
\[
[Q\setminus Q_{l+n-1}, BT^m] = [Q\setminus Q_{l+n-1}, K(\mathbb{Z}, 2)^m],
\]
where $K(\mathbb{Z}, 2)$ is the Eilenberg-MacLane space. By \cite[p.32]{Mor75}, since $\mathbb{Z}$ is an abelian group, we obtain
\[
[Q\setminus Q_{l+n-1}, K(\mathbb{Z}, 2)^m] \cong \check{H}^2(Q\setminus Q_{l+n-1}; \mathbb{Z}^m),
\]
where $\check{H}^2(Q\setminus Q_{l+n-1};\mathbb{Z}^m)$ denotes the \v{C}ech cohomology. Since $Q\setminus Q_{l+n-1}$ is paracompact Hausdorff (see Lemma~\ref{lem: top strata is paracompact}), by \cite[Chap.~6]{Spa66} we have
\[
\check{H}^2(Q\setminus Q_{l+n-1}; \mathbb{Z}^m) \cong H^2(Q\setminus Q_{l+n-1};\mathbb{Z}^m),
\]
where $H^2(Q\setminus Q_{l+n-1};\mathbb{Z}^m)$ denotes the singular cohomology.
Therefore, the isomorphism classes of $T^m$-principal bundles over $Q\setminus Q_{l+n-1}$ are classified by a cohomology class in $H^2(Q\setminus Q_{l+n-1};\mathbb{Z}^m)$.
We call this cohomology class {\it the Chern class of the principal bundle $\pi^{-1}(Q\setminus Q_{l+n-1})$}.

In this paper, we assume the following condition:
\begin{align}\label{condition of top strata}
\text{the top strata }Q\setminus Q_{l+n-1} \text{ is homotopy equivalent to } Q \text{ itself}.
\end{align}
\begin{rem}
This setting covers important examples of topological stratified pseudomanifolds, such as convex polytopes and the orbit space of a locally standard $T$-manifold.
\end{rem}
Under this assumption, the inclusion
\[
i : Q\setminus Q_{l+n-1} \to Q
\]
is a homotopy equivalence. This implies that the induced map
\begin{align}\label{iso induced by homotopy equivalence}
i^{\ast}: H^2(Q; \mathbb{Z}^m) \to H^2(Q\setminus Q_{l+n-1}; \mathbb{Z}^m)
\end{align}
is an isomorphism.
We are now ready to define the Chern class of the locally standard $T$-pseudomanifold $X$.

\begin{definition}\label{def of Chern class}
Let $X$ be a locally standard $T$-pseudomanifold, and $\pi:X \to Q$ be the orbit projection. Assume that the top strata $Q \setminus Q_{l+n-1}$ is homotopy equivalent to $Q$ itself.
The cohomology class $c \in H^2(Q; \mathbb{Z}^m)$, whose pullback to $H^2(Q\setminus Q_{l+n-1}; \mathbb{Z}^m)$ is the Chern class of the principal bundle $\pi^{-1}(Q\setminus Q_{l+n-1})$, is called the {\it Chern class of the locally standard $T$-pseudomanifold $X$}.
\end{definition}

\begin{rem}\label{trivial cohomology}
Note that the Chern class of a locally standard $T$-pseudomanifold can be defined if the map in \eqref{iso induced by homotopy equivalence} is an isomorphism.
In particular, if $H^2(Q; \mathbb{Z}^m)=H^2(Q\setminus Q_{l+n-1}; \mathbb{Z}^m)=0$, then the Chern class is trivial.
\end{rem}

\section{Characteristic functor}\label{section 5}
Let $X$ be a locally standard $T$-pseudomanifold, and let $\pi : X \to Q$ be the orbit projection.

By Lemma \ref{lem of trivial bundle}, we obtain a correspondence $\lambda$ that assigns to each stratum its isotropy group, which is a subtorus.
We explain that $\lambda$ defines a functor, called a {\it characteristic functor of $X$}. In order to define a characteristic functor of $X$, we begin by defining a category of strata of $Q$ (where an $i$-dimensional stratum is a connected component of $Q_i \setminus Q_{i-1}$) as follows.

\begin{definition}[poset category]\label{def of poset category}
Let $Q$ be a topological stratified pseudomanifold (see Definition~\ref{def of pseudomanifold}).
The set of strata of $Q$ forms a poset under the inclusion $S_1 \subset \overline{S_2}$. We denote the poset category of this poset by $\mathcal{S}(Q)$, i.e., an object is a stratum, and there is a morphism $S_{1}\to S_{2}$ if and only if $S_{1} \subset \overline{S}_{2}$.
\end{definition}

\begin{ex}\label{ex2}
The face category of the triangle $P$, which is shown in the figure below, consists of the three vertices $v_1,v_2,v_3$, the three edges $e_1, e_2, e_3$, and the $2$-dimensional face $P$ itself. The morphisms are given by inclusions among these faces. For instance, $v_1 \to e_2$ and $e_1 \to P$. The right diagram shows the poset category $\mathcal{S}(P)$ of $P$, where the notations $\mathring{e}_{i}$ for $i=1,2,3$ and $\mathring{P}$ denote their relative interiors.

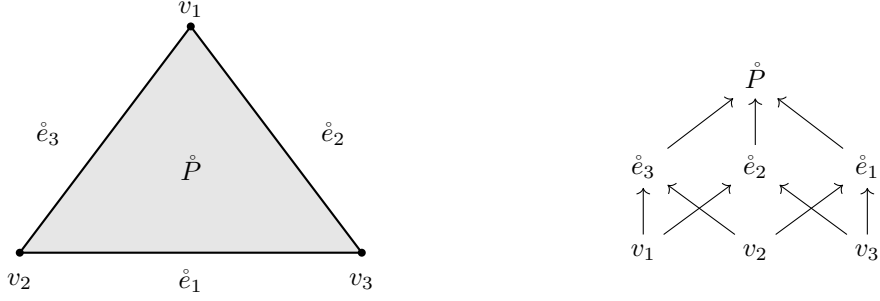
\begin{minipage}{0.5\textwidth}
\centering
\begin{tikzpicture}[scale=0.75]
    \fill[gray!20] (0,0) -- (6,0) -- (3,4) -- cycle; 
    \draw[thick] (0,0) -- (6,0) -- (3,4) -- cycle; 

    \fill (0,0) circle (2pt);
    \fill (6,0) circle (2pt);
    \fill (3,4) circle (2pt);
    \node at (0,-0.5){${v}_2$};
    \node at (6,-0.5) {${v}_3$};
   \node at (3,-0.5) {$\mathring{e}_1$};
    \node at (3,4.3) {${v}_1$};
    \node at (5.5, 2.1) {$\mathring{e}_2$};
    \node at (0.5, 2.1){$\mathring{e}_3$};
    \node at (3, 1.5) {$\mathring{P}$};
\end{tikzpicture}
\end{minipage}
\begin{minipage}{0.5\textwidth}
\centering
\[
\begin{tikzcd}[ampersand replacement=\&]
\& \mathring{P} \& \\
\mathring{e}_3 \ar{ur} \& \mathring{e}_2 \ar{u} \& \mathring{e}_1 \ar{ul} \\
{v}_1 \ar{u} \ar{ur} \& {v}_2 \ar{lu} \ar{ur} \& {v}_3 \ar{ul} \ar{u}
\end{tikzcd}
\]
\end{minipage}
\end{ex}

Let $\mathcal{T}$ be the category of torus subgroups of $T^m$, whose morphism $T_1 \to T_2$ is defined by the inclusion $T_1 \subset T_2$, where a torus subgroup of $T^m$ (also called a subtorus) means a connected closed subgroup of $T^m$. 

We prove the following proposition to define the characteristic functor of $X$.
Here we use the notation introduced in Remark \ref{T_S}.

\begin{prop}\label{prop for functor}
Let $X$ be a locally standard $T$-pseudomanifold, and let $\pi : X \to Q$ be the orbit projection.
For two strata $S_1, S_2 \in \mathcal{S}(Q)$, if $S_1 \subset \overline{S_2}$, then $T_{S_1} \supset T_{S_2}$.
\end{prop}
\begin{proof}
Since $S_1 \subset \overline{S_2}$, we have $\pi^{-1}(S_1) \subset \pi^{-1}(\overline{S_2})$.
Since $T_{S_2}$ fixes $\pi^{-1}(\overline{S_2})$, $\pi^{-1}(S_1)$ is also fixed by $T_{S_2}$.
This means $T_{S_2} \subset T_{S_1}$.
\end{proof}

By this proposition, the following definition is well-defined.

\begin{definition}[characteristic functor of a locally standard $T$-pseudomanifold]\label{def of char functor}
We define {\it characteristic functor of $X$} by
\[
\begin{array}{rccc}\lambda:&\mathcal{S}(Q)^{\mathrm{op}}&\to&\mathcal{T}\\
  & \rotatebox{90}{$\in$}&               & \rotatebox{90}{$\in$} \\
&S &\mapsto&T_{S}\end{array}
\]
Here, $\mathcal{S}(Q)^{\mathrm{op}}$ denotes the opposite category of $\mathcal{S}(Q)$.
\end{definition}

\begin{rem}
This definition generalizes the characteristic function introduced in \cite{DJ91}.
In Construction \ref{from a characteristic function}, we show that a characteristic function induces a characteristic functor.
\end{rem}

\begin{prop}\label{associated dimension}
Let $\lambda : \mathcal{S}(Q)^{\mathrm{op}} \to \mathcal{T}$
be the characteristic functor of $X$.
For $0 \le i \le n$, the following condition is satisfied:
\begin{itemize}
\item
for every $(l+n-i)$-dimensional stratum $S$ (i.e., $\mathrm{codim} S=i$),
$\lambda(S) \subset T^m$ is an $i$-dimensional subtorus.
\end{itemize}
\end{prop}
\begin{proof}
Let $S \subset Q_{l+n-i} \setminus Q_{l+n-i-1}$ be an $(l+n-i)$-dimensional stratum.
The preimage $\pi^{-1}(S)$ is a subset of $X_{l+2(n-i)+(m-n)} \setminus X_{l+2(n-i)+(m-n)-1}$.
By the definition of $X_{l+2(n-i)+(m-n)}$ and $X_{l+2(n-i)+(m-n)-1}$, we have
\[
X_{l+2(n-i)+(m-n)} \setminus X_{l+2(n-i)+(m-n)-1}
=\{ x \in X \mid \dim T(x) = (n-i)+(m-n) \}
\]
Because of $T(x) \cong T^{m}/T_x$, the dimension of $T_x=T_{S}=\lambda(S)$ is $m-\left( (n-i)+(m-n) \right)=i$.
\end{proof}

Motivated from the above properties, we next introduce the notion of characteristic data.

\begin{definition}[characteristic data]\label{def of char data}
A {\it characteristic data} is a triple $(Q, \lambda, c)$ satisfying the following conditions:
\begin{itemize}
\item
$Q=(Q, \mathfrak{Q})$ is an $(l+n)$-dimensional, second-countable, compact Hausdorff topological stratified pseudomanifold whose top strata $Q\setminus Q_{l+n-1}$ is homotopy equivalent to $Q$;

\item 
$\lambda : \mathcal{S}(Q)^{\mathrm{op}} \to \mathcal{T}$ is a functor satisfying the condition of Proposition \ref{associated dimension}:
for each $(l+n-i)$-dimensional stratum $S$, we have $\lambda(S) \subset T^m$ is an $i$-dimensional subtorus, where $m\ge n$;

\item
$c \in H^2(Q; \mathbb{Z}^m)$ is a cohomology class.
\end{itemize}
\end{definition}

\begin{rem}\label{not orbit space}
Note that in the above definition, $Q$ is not necessarily the orbit space of a locally standard $T$-pseudomanifold.
\end{rem}

\section{Weak isomorphism between characteristic datum}\label{section weak isomorphism}
Let $(Q, \lambda, c)$ and $(Q' , \lambda', c')$ be characteristic datum.
We define a {\it weak isomorphism between characteristic datum}.
We begin by recalling the notions of {\it stratified map} and {\it stratified homeomorphism} (see \cite[Chapter 2.9]{Fri20} for details).

\begin{definition}[stratified map]\label{def of stratified map}
Let $Q$, $Q'$ be manifold stratified spaces, and $f : Q \to Q'$ be a continuous map. We say that $f$ is a {\it stratified map} if for each stratum $S \subset Q$, there exists a unique stratum $S' \subset Q'$ such that $f(S) \subset S'$.
\end{definition}

\begin{definition}[stratified homeomorphism]\label{def of stratified homeo}
Let $Q$ and $Q'$ be manifold stratified spaces. A stratified map $f:Q\to Q'$ is said to be a {\it stratified homeomorphism} if it satisfies the following conditions:
\begin{enumerate}
\item $f$ is a homeomorphism;
\item $f^{-1}:Q'\to Q$ is also a stratified map (and a homeomorphism).
\end{enumerate}
\end{definition}

\begin{rem}
In \cite{Fri20}, a stratified homeomorphism is defined as a map between {\it filtered spaces} (see \cite[Chapter 2.2]{Fri20}) that satisfies the additional condition:
\begin{itemize}
\item for each stratum $S \subset Q$, we have $\mathrm{codim} S = \mathrm{codim} f(S)$.
\end{itemize}
However, in the setting of manifold stratified spaces, each stratum $S$ is a topological manifold (see Definition \ref{def of manifold stratified space}). Since $f$ is a homeomorphism, codimensions are preserved. Therefore, this condition is automatically satisfied.

\end{rem}

To define a {\it weak isomorphism between characteristic datum}, we first prove the following proposition. 

\begin{prop}\label{prop: stratified homeo induces a poset iso}
A stratified homeomorphism $f: Q \to Q'$ induces a poset isomorphism
\[
\mathcal{S}(f): \mathcal{S}(Q)^{\mathrm{op}} \to \mathcal{S}(Q')^{\mathrm{op}}.
\]
\end{prop}
\begin{proof}
Let $S \subset Q$ be a stratum. Since $f$ is a stratified homeomorphism, $f(S) (=:S')$ is a stratum of $Q'$.
By the definition of stratified homeomorphism, the inverse $f^{-1}$ is also a stratified homeomorphism. Therefore, the map 
\[
\begin{array}{rccc}
\mathcal{S}(f) :&\mathcal{S}(Q)^{\mathrm{op}} &\to &\mathcal{S}(Q')^{\mathrm{op}} \\
 & \rotatebox{90}{$\in$}&               & \rotatebox{90}{$\in$} \\
&S &\mapsto&S'. \end{array}
\]
is a bijection. 
It remains to show that $\mathcal{S}(f)$ is a poset isomorphism (see also Definition~\ref{def of poset category}); that is, for strata
$S, R \in \mathcal{S}(Q)$, we have 
\begin{align*}
S \subset \overline{R}
\iff
S' \subset \overline{R'},
\end{align*}
where $S'=f(S)$ and $R'=f(R)$.
Suppose that $S \subset \overline{R}$.
Since $f$ is continuous, it follows that
\[
S'=f(S) \subset f(\overline{R})\subset \overline{f(R)}=\overline{R'}.
\]
The converse follows in the same way, using the continuity of $f^{-1}$.
\end{proof}

This proposition leads us to define the following notion.

\begin{definition}\label{def of isomorphism of char pairs}
Let $f: Q \to Q'$ be a stratified homeomorphism.
We say that $f$ is a {\it weak isomorphism between characteristic datum} $(Q, \lambda, c)$ and $(Q', \lambda', c')$ if $c=f^{\ast}(c')$, where $f^{\ast} : H^2(Q'; \mathbb{Z}^m) \to H^2(Q; \mathbb{Z}^m)$ is the induced isomorphisim from $f$, and there exists an automorphism $\psi: T^m \to T^m$ such that the following diagram commutes (i.e., for each stratum $S \in \mathcal{S}(Q)^{\mathrm{op}}$, $\Psi \circ \lambda(S)= \lambda' \circ \mathcal{S}(f)(S)$):
\begin{align}\label{comm diagram: weak isomorphism}
\begin{tikzcd}[ampersand replacement=\&]
\mathcal{S}(Q)^{\mathrm{op}} \ar{rr}{\mathcal{S}(f)}[']{\cong}  \ar{d}[']{\lambda} \&
   \&   
      \mathcal{S}(Q')^{\mathrm{op}} \ar{d}{\lambda'} \\
    \mathcal{T}  \ar{rr}{\Psi}[']{\cong}
  \& \&
    \mathcal{T}
\end{tikzcd}
\end{align}
where $\Psi$ is the isomorphism functor determined by $\Psi(T_S)=\psi(T_S)$ for $T_S \in  \mathcal{T}$.
If there exists such a weak isomorphism between $(Q, \lambda, c)$ and $(Q',\lambda', c')$, then we say that $(Q, \lambda, c)$ is {\it weakly isomorphic} to $(Q',\lambda', c')$.
When the automorphism $\psi$ is the identity, we simply say {\it isomorphism} instead of weak isomorphism, and {\it isomorphic} instead of weakly isomorphic.
\end{definition}

Note that if $f: (Q,\lambda,c) \to (Q', \lambda',c')$ is a weak isomorphism, then its inverse $f^{-1}:(Q', \lambda',c') \to (Q,\lambda,c)$ is also a weak isomorphism.

\subsection{Uniqueness of the characteristic data}
Let $X$ and $X'$ be locally standard $T$-pseudomanifolds, and let $\mathfrak{X}$ and $\mathfrak{X'}$ denote their respective filtrations by orbit dimension.
In this subsection, we show that if $X$ and $X'$ are (weakly) equivariantly homeomorphic, then their characteristic datum $(X/T, \lambda, c)$ and $(X'/T, \lambda', c')$ are (weakly) isomorphic (see Proposition \ref{prop in 4}).
To prove Proposition \ref{prop in 4}, we first prepare two lemmas.

\begin{lem}\label{equivariant stratum}
Let $f: X\to X'$ be a (weakly) $T$-equivariant homeomorphism between locally standard $T$-pseudomanifolds.
Then, for any stratum $S$ of $X_{l+2i+(m-n)}\setminus X_{l+2(i-1)+(m-n)}$, the image $S':=f(S)$ is a stratum of $X'_{l+2i+(m-n)}\setminus X'_{l+2(i-1)+(m-n)}$.
Moreover, $S$ and $S'$ are (weakly) $T$-equivariantly homeomorphic.
\end{lem}
\begin{proof}
Recall \eqref{dimension set}
\[
X_{l+2i+(m-n)}=\{x\in X\ |\ \dim T(x)\le i+(m-n)\}.
\]
Each stratum (i.e., a connected component) $S \subset X_{l+2i+(m-n)}\setminus X_{l+2(i-1)+(m-n)}$ consists of orbits of the same dimension, and $T$ acts on $S$ such that the isotropy group is a subtorus (see Remark~\ref{assumption of subtorus}).
A (weakly) $T$-equivariant homeomorphism preserves orbits of each dimension.
Therefore, the image satisfies $f(S) \subset X'_{l+2i+(m-n)}\setminus X'_{l+2(i-1)+(m-n)}$.
Since the homeomorphism $f$ preserves connected components, it follows that $S':=f(S)$ is a stratum of $X'_{l+2i+(m-n)}\setminus X'_{l+2(i-1)+(m-n)}$. 
Since $f$ is a (weakly) $T$-equivariant homeomorphism, $S$ is (weakly) $T$-equivariant homeomorphic to $S'$.
\end{proof}

\begin{lem}\label{iso of filtered orbit spaces}
If two
locally standard $T$-pseudomanifolds $X$ and $X'$ are weakly $T$-equivariantly homeomorphic, then the filtered orbit spaces $(X/T, \mathfrak{X}/T)$ and $(X'/T , \mathfrak{X}'/T)$ are stratified homeomorphic.\end{lem}
\begin{proof}
A weakly $T$-equivariant homeomorphism $f: X \to X'$ induces a homeomorphism
\[
\widehat{f}: X/T \to X'/T.
\]
Then, we have the following commutative diagram.
\begin{align*}
\begin{tikzcd}[ampersand replacement=\&]
X \rar{f}[']{\cong} \dar{/T} \&
  X' \dar{/T}
\\
X/T \rar{\widehat{f}}[']{\cong} \&
  X'/T
\end{tikzcd}
\end{align*}
Each stratum of $X/T$ is of the form $S/T$, where $S$ is a stratum of $X$. 
By Lemma \ref{equivariant stratum}, $S':=f(S)$ is a stratum of $X'$. 
Then, by commutativity of above diagram, we have
\begin{align}\label{stratum of the form S/T}
\widehat{f}(S/T)=S'/T.
\end{align}
Therefore, both $\widehat{f}$ and $\widehat{f}^{-1}$ are stratified maps.
Furthermore, by the latter part of Lemma~\ref{equivariant stratum}, we have $\dim S/T =\dim S'/T$. Hence, $\widehat{f}$ is a stratified homeomorphism.
\end{proof}

Using Lemma \ref{equivariant stratum} and \ref{iso of filtered orbit spaces}, we have the following proposition.

\begin{prop}\label{prop in 4}
Let $X$ and $X'$ be locally standard $T$-pseudomanifolds whose orbit spaces $Q$ and $Q'$ have the set of top strata that is homotopy equivalent to the entire space.
If $X$ and $X'$ are (weakly) $T$-equivariantly homeomorphic, then their characteristic datum $(Q, \lambda, c)$ and $(Q', \lambda', c')$ are (weakly) isomorphic.
\end{prop}
\begin{proof}
Let $f:X \to X'$ be a ($\psi$-weakly) $T$-equivariant homeomorphism.
By Lemma \ref{iso of filtered orbit spaces}, $f$ induces a stratified homeomorphism $\widehat{f} : Q \to Q'$.
Note that by Proposition~\ref{prop: stratified homeo induces a poset iso}, $\widehat{f}$ induces a poset isomorphism $\mathcal{S}(\widehat{f}): \mathcal{S}(Q)^{\mathrm{op}} \to \mathcal{S}(Q')^{\mathrm{op}}$.
Let $c \in H^2(Q;\mathbb{Z}^m)$ and $c' \in H^2(Q';\mathbb{Z}^m)$ be the Chern classes. 
Since $f$ is a ($\psi$-weakly) $T$-equivariant homeomorphism, $f$ induces a bundle isomorphism between free orbits of $X$ and $X'$. Therefore, we obtain
\[
\widehat{f}^{\ast}(c')=c,
\]
where $\widehat{f}^{\ast} : H^2(Q';\mathbb{Z}^m) \to H^2(Q;\mathbb{Z}^m)$ is the isomorphism induced by $f$.
We now verify the commutativity of the diagram \eqref{comm diagram: weak isomorphism}. For each stratum $S/T \in \mathcal{S}(Q)^{\mathrm{op}}$, where $S \subset X$ is a stratum, we have $\mathcal{S}(\widehat{f})(S/T)=f(S)/T \in \mathcal{S}(Q')^{\mathrm{op}}$ (see \eqref{stratum of the form S/T}). By Lemma \ref{equivariant stratum}, $S$ and $f(S)$ are ($\psi$-weakly) $T$-equivariantly homeomorphic.
Moreover, $S$ and $f(S)$ consist of $T$-orbits of the same dimension.
Since $T$ is commutative, they have the same isotropy subgroups (up to the automorphism $\psi$).
It follows that
\[
\psi(\lambda(S/T))=\lambda'(f(S)/T),
\]
where $\lambda$ and $\lambda'$ are the characteristic functors of $X$ and $X'$, respectively (see Definition \ref{def of char functor}).
Therefore, the diagram \eqref{comm diagram: weak isomorphism} commutes, and the proposition follows.
\end{proof}

\section{Example}\label{easy examples}

We give three examples of locally standard $T$-pseudomanifolds and their characteristic datum. 

\begin{ex}\label{ex1}
Let $S^{3}$ be the unit sphere in $\mathbb{C}^{2}$.
We consider the Thom space $X$ of the complex line bundle $S^3 \times_{S^1} \mathbb{C}$ over $\mathbb{C}P^1  \left( =S^3/S^1 \right)$, i.e.,
$X=X_4:=(S^3 \times_{S^1} D^2)/(S^3 \times_{S^1} S^1)$,
where $D^2$ and $S^1$ are the unit disk and the unit circle in $\mathbb{C}$, respectively.
Here,
the $S^1$-action on $S^3 \times \mathbb{C}$ is defined as follows:
for $\left( (x,y), z \right) \in S^3 \times \mathbb{C}$ and $t \in S^1$, 
\[
t \cdot \left( (x,y), z \right) := \left( (tx,ty), t^k z \right),
\]
where $S^3 \subset \mathbb{C}^2$ and $k \in \mathbb{Z}$.
We define a $T^2$-action on the Thom space $(S^3 \times_{S^1} D^2)/(S^3 \times_{S^1} S^1)$ by
\[
(t_1, t_2) \cdot [(x,y),z] = [(x,t_1 y), t_2 z]
\]
for $(t_1, t_2) \in T^2$ and
$[(x,y),z] \in (S^3 \times_{S^1} D^2)/(S^3 \times_{S^1} S^1)$.
The fixed points are $[(1,0),0], [(0,1),0]$, and $[(x,y),w]$ where $w \in S^1 = \partial D^2$. We denote the set of these fixed points by $X_0$. The set $X_2$ consists of the closures of the following three strata (i.e., the connected components of $X_2 \setminus X_0$):
\begin{align*}
S_1=\{  [(x,y),0] \mid (x,y) \in S^3 \} \setminus X_0, \,
S_2=\{ [(0,1),z] \mid z \in D^2 \} \setminus X_0 ,\,
S_3= \{ [(1,0),z] \mid z \in D^2 \} \setminus X_0.
\end{align*}

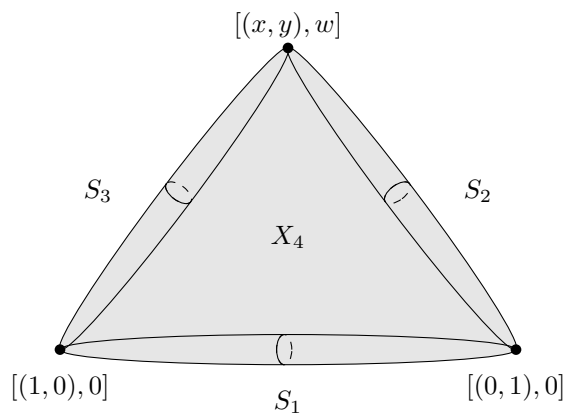
\[
\begin{tikzpicture}
    \fill[gray!20] (0,0) -- (6,0) -- (3,4) -- cycle; 
    \draw[thick] (0,0) -- (6,0) -- (3,4) -- cycle; 

\fill[gray!20] (6,0) arc[start angle=0, end angle=360, x radius=3cm, y radius=0.2cm];
\fill[gray!20][rotate around={-53.13:(6,0)}] (6,0) arc[start angle=0, end angle=360, x radius=2.5cm, y radius=0.2cm];
\fill[gray!20][rotate around={53.13:(3,4)}] (3,4) arc[start angle=0, end angle=360, x radius=2.5cm, y radius=0.2cm];

\draw (6,0) arc[start angle=0, end angle=360, x radius=3cm, y radius=0.2cm];
\draw[rotate around={-53.13:(6,0)}] (6,0) arc[start angle=0, end angle=360, x radius=2.5cm, y radius=0.2cm];
\draw[rotate around={53.13:(3,4)}] (3,4) arc[start angle=0, end angle=360, x radius=2.5cm, y radius=0.2cm];

\draw[dashed](3.05,0) arc[start angle=0, end angle=360, x radius=0.1cm, y radius=0.2cm];
\draw(2.95, 0.2) arc[start angle=90, end angle=270, x radius=0.1cm, y radius=0.2cm];

\draw[dashed][rotate around={-53.13:(4.5,2)}] (4.5,2) arc[start angle=0, end angle=360, x radius=0.1cm, y radius=0.2cm];
\draw[rotate around={-53.13:(4.6,2.2)}] (4.6,2.2) arc[start angle=90, end angle=270, x radius=0.1cm, y radius=0.2cm];

\draw[dashed][rotate around={53.13:(1.4,2.2)}] (1.4,2.2) arc[start angle=90, end angle=450, x radius=0.1cm, y radius=0.2cm];
\draw[rotate around={53.13:(1.4,2.2)}] (1.4,2.2) arc[start angle=90, end angle=270, x radius=0.1cm, y radius=0.2cm];

    \fill (0,0) circle (2pt);
    \fill (6,0) circle (2pt);
    \fill (3,4) circle (2pt);

    \node at (0,-0.5) {$[ (1,0),0 ]$};
    \node at (6,-0.5) {$[(0,1),0]$};
   \node at (3,-0.7) {$S_1$};
    \node at (3,4.3) {$[(x,y),w]$};
    \node at (5.5, 2.1) {$S_2 $};
    \node at (0.5, 2.1) {$S_3 $};
    \node at (3, 1.5) {$X_4 $};
\end{tikzpicture}
\]
The space $X$ is a $4$-dimensional locally standard $T$-pseudomanifold equipped with the following filtration:
\[
\mathfrak{X} : X = X_4 \supsetneq X_2 \supset X_0 \supsetneq \emptyset.
\]
Namely, this is the case when $m=2$, $n=2$ and $l=0$.

We next verify the isotropy subgroup corresponding to each stratum, which is a subtorus. Each fixed point corresponds to $T^2$, and $X_4$ corresponds to the identity subgroup $\{ 1\} \subset T^2$. The other strata correspond to the following subgroups:
\begin{align*}
S_1=\{  [(x,y),0] \mid (x,y) \in S^3 \} 
&\longleftrightarrow 
T\braket{0,1}=\{ (1, t) \in T^2 \mid t \in S^1 \}, 
\\
S_2= \{ [(0,1),z] \mid z \in D^2 \}
&\longleftrightarrow  
T\braket{1,k}=\{ (t,t^k) \in T^2 \mid t \in S^1 \} ,      
   \\
S_3= \{ [(1,0),z] \mid z \in D^2 \}
&\longleftrightarrow
T\braket{1,0}=\{ (t,1) \in T^2 \mid t \in S^1 \} ,
\end{align*}
where the symbol $T\braket{a,b}$ represents the circle subgroup defined by the weight $(a,b) \in \mathfrak{t}_{\mathbb{Z}} \cong \mathbb{Z}^2$.
Because $(1,0),\,(0,1),\, (1,k)$ form a basis of weight lattice $\mathfrak{t}_{\mathbb{Z}}$,
we can define the characteristic functor 
$\lambda : \mathcal{S}(P)^{\mathrm{op}} \to \mathcal{T}$
on the triangle $P$ as shown in the following figure:

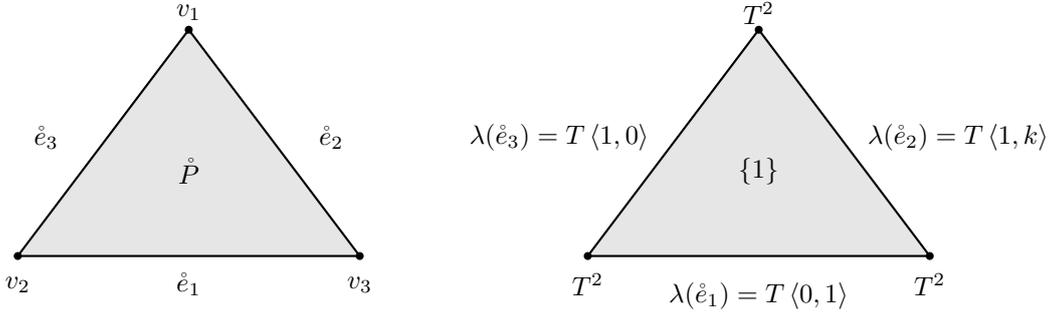
\[
\begin{tikzpicture}[scale=0.75]
\begin{scope}
    \fill[gray!20] (0,0) -- (6,0) -- (3,4) -- cycle; 
    \draw[thick] (0,0) -- (6,0) -- (3,4) -- cycle; 

    \fill (0,0) circle (2pt);
    \fill (6,0) circle (2pt);
    \fill (3,4) circle (2pt);
    \node at (0,-0.5){$v_2$};
    \node at (6,-0.5) {$v_3$};
   \node at (3,-0.5) {$\mathring{e}_1$};
    \node at (3,4.3) {$v_1$};
    \node at (5.5, 2.1) {$\mathring{e}_2$};
    \node at (0.5, 2.1){$\mathring{e}_3$};
    \node at (3, 1.5) {$\mathring{P}$};
\end{scope}

\begin{scope}[xshift=10cm]
    \fill[gray!20] (0,0) -- (6,0) -- (3,4) -- cycle; 
    \draw[thick] (0,0) -- (6,0) -- (3,4) -- cycle; 

    \fill (0,0) circle (2pt);
    \fill (6,0) circle (2pt);
    \fill (3,4) circle (2pt);

    \node at (0,-0.5) {$T^2$};
    \node at (6,-0.5) {$T^2$};
   \node at (3,-0.7) {$\lambda(\mathring{e}_1)=T\braket{0,1}$};
    \node at (3,4.3) {$T^2$};
    \node at (6.5, 2.1) {$\lambda(\mathring{e}_2)=T\braket{1,k} $};
    \node at (-0.5, 2.1) {$\lambda(\mathring{e}_3)=T\braket{1,0} $};
    \node at (3, 1.5) {$\{ 1 \}$};
\end{scope}
\end{tikzpicture}
\]

We next compute the link of each point.
The following links are easily computed:
\begin{itemize}
\item For $x \in X_4 \setminus X_2$, we put $L_x= \emptyset$;

\item For $x_1 \in S_1\setminus X_0, x_2 \in S_2\setminus X_0$ and $x_3 \in S_3\setminus X_0$, we put
$L_{x_1}=L_{x_2}=L_{x_3}=S^1$;
\end{itemize}
It remains to compute the link of each fixed point.
In this example, the link of each fixed point can be computed from the link of each vertex in the orbit space (in this case, the edges $E_1, E_2, E_3$ in the figure below) and the {\it hyper characteristic function} on it.
See the Remark~\ref{rem: hyper ch function} for the hyper characteristic function.
In this example, for instance, the hyper characteristic function on $E_1$ is the characteristic functor defined by $(1,0)$ and $(1,k)$ (see the figure below).
\[
\begin{tikzpicture}[scale=0.75]
\begin{scope}
    \coordinate (v1) at (3,4);
    \coordinate (v2) at (0,0);
    \coordinate (v3) at (6,0);

    \def\cut{0.3}

    \path (v2) -- (v3) coordinate[pos=\cut](a1);
    \path (v3) -- (v2) coordinate[pos=\cut](a2);
    \path (v3) -- (v1) coordinate[pos=\cut](b1);
    \path (v1) -- (v3) coordinate[pos=\cut](b2);
    \path (v1) -- (v2) coordinate[pos=\cut](c1);
    \path (v2) -- (v1) coordinate[pos=\cut](c2);

    \draw[dashed] (v1)--(v2)--(v3)--cycle;

    \draw[thick, red] (a2)--(b1);
    \draw[thick, red] (b2)--(c1);
    \draw[thick, red] (a1)--(c2);

    \fill (v1) circle (2pt);
    \fill (v2) circle (2pt);
    \fill (v3) circle (2pt);

    \fill[red] (a1) circle (2pt);
    \fill[red] (a2) circle (2pt);
    \fill[red] (b1) circle (2pt);
    \fill[red] (b2) circle (2pt);
    \fill[red] (c1) circle (2pt);
    \fill[red] (c2) circle (2pt);

    \node at (0,-0.5){$v_2$};
    \node at (6,-0.5) {$v_3$};
    \node at (3,4.3) {$v_1$};

    \node[below] at (a1) {$T\braket{0,1}$};
    \node[below] at (a2) {$T\braket{0,1}$};
    \node[right] at (b1) {$T\braket{1,k}$};
    \node[right] at (b2) {$T\braket{1,k}$};
    \node[left] at (c1) {$T\braket{1,0}$};
    \node[left] at (c2) {$T\braket{1,0}$};

    \node[xshift=0pt, yshift=16pt] at (a1) {$\{1\}$};
    \node[xshift=2pt, yshift=16pt] at (a2) {$\{1\}$};
    \node[xshift=0pt, yshift=-40pt] at (3,4.3) {$\{1\}$};

    \node[red] [xshift=-3pt, yshift=-16pt] at (b1) {$E_3$};
    \node[red] [xshift=4pt, yshift=-16pt] at (c2) {$E_2$};
    \node[red] [xshift=0pt, yshift=-26pt] at (3,4.3) {$E_1$};
\end{scope}
\end{tikzpicture}
\]

Using \cite[Proposition 2.3]{SS16}, the link of each fixed point is computed as follows:
\begin{itemize}

\item $L_{[ (1,0),0 ]}=L_{[(0,1),0]}=S^3$;

\item
For $k=0$, $L_{[(x,y),w]}=\mathbb{C}P^1 \times S^1$; \\ For $k=\pm1$, $L_{[(x,y),w]}=S^3$;
\\
For $k\neq 0, \pm1$, $L_{[(x,y),w]}=S^3 / (\mathbb{Z}^2/ \mathbb{Z}\braket{\bm{e}_1, \bm{e}_1+ k \bm{e}_2} )$.
\end{itemize}

\end{ex}

\begin{rem}\label{rem: hyper ch function}
In {{\cite[Section 2]{SS16}}}, a {\it hyper characteristic function} is studied (this is a special case of the {\it unimodular labeling} in \cite{KK25}).
This is a function that assigns a weight lattice in $\mathbb{Z}^{n+1}$ to each facet of an $n$-simplex.
Sarkar and Suh classify the spaces constructed from such hyper characteristic functions (see {{\cite[Proposition 2.3]{SS16}}}).
In Section \ref{section 12}, we show that a {\it characteristic functor} can be used to construct a link.
In this example, the characteristic functor used to construct a link may be regarded as  a hyper characteristic function.
Therefore, in this case, the link can be computed using {{\cite[Proposition 2.3]{SS16}}}.
However, note that this tool for computing a link only applies when $m=n$ and the orbit space is stratified homeomorphic to an $n$-simplex.
In contrast, by using the results in \cite{Wie22} or \cite{KK25}, we can also compute the link for $m>n$.
\end{rem}

\begin{ex}
Let $m \ge n \ge 1$. We next consider the Thom space $X$ of the trivial $\mathbb{R}$-vector bundle $T^{m+n-1} \times \mathbb{R}$ over $T^{m+n-1}$, i.e., $X=X_{m+n}=(T^{m+n-1} \times [-1,1])/(T^{m+n-1} \times \{ -1,1 \})$.
We define a $T^{m}$-action on $(T^{m+n-1} \times [-1,1])/(T^{m+n-1} \times \{ -1,1 \})$ by 
\[
(t_1, \ldots , t_m) \cdot [s_1, \ldots , s_m, s_{m+1}, \ldots ,s_{m+n-1} , r]=[t_1 s_1, \ldots , t_m s_m, s_{m+1}, \ldots ,s_{m+n-1} , r]
\]
for $(t_1, \ldots , t_m) \in T^m$ and $[s_1, \ldots , s_m, s_{m+1}, \ldots ,s_{m+n-1}, r] \in (T^{m+n-1} \times [-1,1])/(T^{m+n-1} \times \{ -1,1 \})$.
Consider the filtration $\mathfrak{X}$ by the dimension of the orbits. Since every orbit of dimension less than $m$ is a fixed point of the form $[s_1, \ldots , s_{m+n-1} , v]$, where $v \in \{ -1,1\}$, we have
\[
X_{m+n-1}= \cdots =X_0=\{ [s_1, \ldots , s_{m+n-1} , v] \}.
\] 
Hence, the filtration (see the left figure below) is given by
\[
\mathfrak{X} : X_{m+n} \supsetneq X_0 \supsetneq \emptyset.
\]
It is easy to check that there are exactly two isotropy groups, i.e., the isotropy group of $X_{0}$ is $T^{m}$ and that of $X_{m+n}\setminus X_{0}$ is $\{1\}$.
The filtered orbit space of $X$ (see the right figure below) is given by
\[
\mathfrak{X}/T : (T^{n-1} \times [-1,1])/(T^{n-1} \times \{ -1,1 \}) \supsetneq \{ pt \} \supsetneq \emptyset.
\]
Note that the lowest dimension of strata is $l=0$ and the length of the filtration is $n$ (i.e., $\mathfrak{X}/T$ has $n$-dimension).

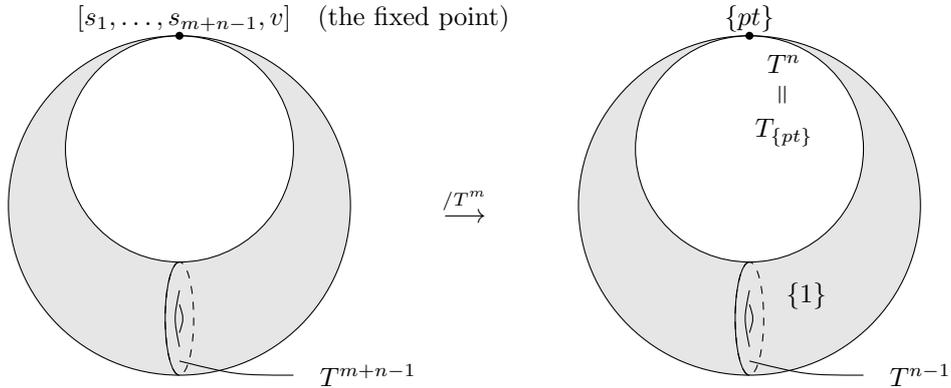
\[
\begin{tikzpicture}[scale=0.75]
\begin{scope}
\fill[gray!20] (0,0) arc[start angle=90, end angle=450, x radius=3cm, y radius=3cm];
\fill[white] (0,0) arc[start angle=90, end angle=450, x radius=2cm, y radius=2cm];
    
    \draw (0,0) arc[start angle=90, end angle=450, x radius=3cm, y radius=3cm];
    \draw (0,0) arc[start angle=90, end angle=450, x radius=2cm, y radius=2cm];

    \draw[dashed] (0,-4) arc[start angle=90, end angle=450, x radius=0.25cm, y radius=1cm];
    \draw (0,-4) arc[start angle=90, end angle=270, x radius=0.25cm, y radius=1cm];

    \fill (0,0) circle (2pt);

     \draw (0,-4.5) .. controls (-0.1,-5)  .. (0,-5.5);
     \draw (0,-4.75) .. controls (0.08,-5)  .. (0,-5.25);
     \draw (0,-5.75) .. controls (1,-6)  .. (2,-6);

   \node at (2.3,-6) [right] {$T^{m+n-1}$};
   \node at (2,0.3) {$[s_1, \ldots , s_{m+n-1}, v]$\quad (the fixed point)};

\node at (5,-3) {$\overset{/T^m}{\longrightarrow}$};
\end{scope}

\begin{scope}[xshift=10cm]
\fill[gray!20] (0,0) arc[start angle=90, end angle=450, x radius=3cm, y radius=3cm];
\fill[white] (0,0) arc[start angle=90, end angle=450, x radius=2cm, y radius=2cm];
    
    \draw (0,0) arc[start angle=90, end angle=450, x radius=3cm, y radius=3cm];
    \draw (0,0) arc[start angle=90, end angle=450, x radius=2cm, y radius=2cm];

    \draw[dashed] (0,-4) arc[start angle=90, end angle=450, x radius=0.25cm, y radius=1cm];
    \draw (0,-4) arc[start angle=90, end angle=270, x radius=0.25cm, y radius=1cm];

    \fill (0,0) circle (2pt);

     \draw (0,-4.5) .. controls (-0.1,-5)  .. (0,-5.5);
     \draw (0,-4.75) .. controls (0.08,-5)  .. (0,-5.25);
     \draw (0,-5.75) .. controls (1,-6)  .. (2,-6);

   \node at (2.3,-6) [right] {$T^{n-1}$};
   \node at (0,0.3) {$\{ pt \}$};

\node at (-0.3,-0.1) [below right] {$\begin{array}{c}
T^n\\
\rotatebox{90}{=} \\
T_{\{ pt \}}
\end{array}$};
\node at (1,-5) [above] {$\{ 1 \}$};
\end{scope}

\end{tikzpicture}
\]
If $X$ is a locally standard $T$-pseudomanifold over $(X/T, \mathfrak{X}/T)$, then by Proposition \ref{associated dimension}, the dimension of the torus subgroup corresponding to $\{ pt \}$ (note that $\mathrm{codim}(pt) = n$) must be $n$.
This only happens for $m=n$.
Then we may take the links as follows:
\begin{itemize}
\item
For $x \in X_{2n} \setminus X_0$, we put $L_x = \emptyset$;

\item
$L_{[s_1, \ldots , s_{2n-1}, v]}=T^{n-1} \sqcup T^{n-1}$.
\end{itemize}
\end{ex}

\begin{rem}
This example can be generalized as follows: consider the product $S^1 \times M^{2n}$, where $M^{2n}$ is a locally standard $T^n$-manifold (see Definition~\ref{def of locally standard T-manifold}). Then, collapsing a fiber $\{ pt \} \times M^{2n}$ to a point yields a space that generalizes the example.
\end{rem}

\begin{rem}
This example shows that the set of top strata is not homotopy equivalent to the entire space (see also \eqref{condition of top strata}).
\end{rem}

\begin{ex}
We consider $S^3 = \{ (z,w) \in \mathbb{C}^2  \mid  |z|^2 +|w|^2=1 \}$ equipped with the $T^1$-action defined by
\[
t \cdot (z, w) := (z, tw)
\]
for $t \in T^1$ and $(z, w) \in S^3$.
The fixed point set is 
\[
\{ (z, w) \in S^3 \mid |z|=1, w=0 \}
=\{ (z,0) \in S^3  \mid  |z|=1 \} \cong S^1.
\]
Then $S^3$ is a $3$-dimensional locally standard $T$-pseudomanifold equipped with the following filtration:
\[
\mathfrak{X} : S^3 \supsetneq S^1 \supsetneq \emptyset.
\]
Namely, this is the case when $m=1$, $n=1$ and $l=1$.
We may take the links as follows:
\begin{itemize}
\item
For $x \in S^3 \setminus S^1$, we put $L_x=\emptyset$;

\item
For $x \in S^1$, we have $L_x \cong S^1$.
\end{itemize}

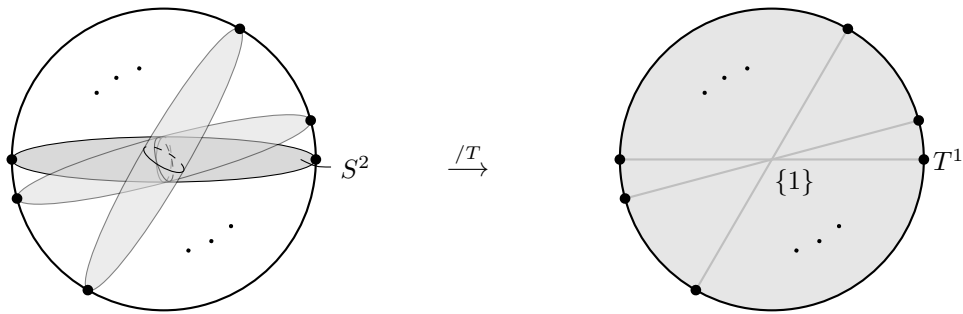
\begin{figure}[htbp]
\[
\begin{tikzpicture}
\begin{scope}

  \def\r{2}

  \draw[thick] (0:\r) arc[start angle=0, end angle=360, radius=\r];

 \begin{scope}
  \draw[fill=gray!30] (0, 0) ellipse [x radius=2, y radius=0.3];

  \draw[dashed] (0,-0.3) arc[start angle=-90, end angle=90, x radius=0.1, y radius=0.3];

  \draw (0,0.3) arc[start angle=90, end angle=270, x radius=0.1, y radius=0.3];

     \draw (1.8,0) .. controls (2,-0.1)  .. (2.2,-0.1);
  \end{scope}
\begin{scope}[rotate=15]
    \draw[fill=gray!30, opacity=0.5] (0, 0) ellipse [x radius=2, y radius=0.3];

    \draw[dashed] (0,-0.3) arc[start angle=-90, end angle=90, x radius=0.1, y radius=0.3];

    \draw (0,0.3) arc[start angle=90, end angle=270, x radius=0.1, y radius=0.3];
  \end{scope}
 \begin{scope}[rotate=60]
    \draw[fill=gray!30, opacity=0.5] (0, 0) ellipse [x radius=2, y radius=0.3];

    \draw[dashed] (0,-0.3) arc[start angle=-90, end angle=90, x radius=0.1, y radius=0.3];

    \draw (0,0.3) arc[start angle=90, end angle=270, x radius=0.1, y radius=0.3];
\end{scope}

\begin{scope}[rotate=105]
    \fill (1.25,0) circle (0.8pt);
\end{scope}
\begin{scope}[rotate=120]
    \fill (1.25,0) circle (0.8pt);
\end{scope}
\begin{scope}[rotate=135]
    \fill (1.25,0) circle (0.8pt);
\end{scope}

\begin{scope}[rotate=285]
    \fill (1.25,0) circle (0.8pt);
\end{scope}
\begin{scope}[rotate=300]
    \fill (1.25,0) circle (0.8pt);
\end{scope}
\begin{scope}[rotate=315]
    \fill (1.25,0) circle (0.8pt);
\end{scope}

\begin{scope}
    \fill (2,0) circle (2pt);
    \fill (-2,0) circle (2pt);
\end{scope}
\begin{scope}[rotate=15]
    \fill (2,0) circle (2pt);
    \fill (-2,0) circle (2pt);
\end{scope}
\begin{scope}[rotate=60]
    \fill (2,0) circle (2pt);
    \fill (-2,0) circle (2pt);
\end{scope}
\node at (2.2,-0.1) [right] {$S^2$};

\node at (4,0) {$\overset{/T}{\longrightarrow}$};
\end{scope}

\begin{scope}[xshift=8cm, yshift=0cm]
\fill[gray!20] (2,0) arc[start angle=0, end angle=360, x radius=2cm, y radius=2cm];
   \draw[thick] (2,0) arc[start angle=0, end angle=360, x radius=2cm, y radius=2cm];

\node at (2,0) [right] {$T^1$};
\node at (0.3, -0.3) {$\{ 1 \}$};

\begin{scope}
\draw[thick, gray!50] (-2,0) -- (2,0); 
\end{scope}
\begin{scope}[rotate=15]
\draw[thick, gray!50] (-2,0) -- (2,0); 
\end{scope}
\begin{scope}[rotate=60]
\draw[thick, gray!50] (-2,0) -- (2,0); 
\end{scope}

\begin{scope}[rotate=105]
    \fill (1.25,0) circle (0.8pt);
\end{scope}
\begin{scope}[rotate=120]
    \fill (1.25,0) circle (0.8pt);
\end{scope}
\begin{scope}[rotate=135]
    \fill (1.25,0) circle (0.8pt);
\end{scope}

\begin{scope}[rotate=285]
    \fill (1.25,0) circle (0.8pt);
\end{scope}
\begin{scope}[rotate=300]
    \fill (1.25,0) circle (0.8pt);
\end{scope}
\begin{scope}[rotate=315]
    \fill (1.25,0) circle (0.8pt);
\end{scope}
\begin{scope}
    \fill (2,0) circle (2pt);
    \fill (-2,0) circle (2pt);
\end{scope}
\begin{scope}[rotate=15]
    \fill (2,0) circle (2pt);
    \fill (-2,0) circle (2pt);
\end{scope}
\begin{scope}[rotate=60]
    \fill (2,0) circle (2pt);
    \fill (-2,0) circle (2pt);
\end{scope}
\end{scope}

\end{tikzpicture}
\]
\vspace{1pt}
The boundary $S^1$ of the disk represents $|z|=1$; this corresponds to the fixed points.
The interior points of the disk (where $w\neq 0$) correspond to the free $S^1$-orbits.
\end{figure}
We now verify the isotropy subgroup corresponding to each stratum.
The top stratum $S^3 \setminus S^1$ corresponds to the identity subgroup $\{1 \} \subset T^1$, and $S^1$ corresponds to $T^1$.
Thus, we obtain the characteristic functor on $S^3/T \cong D^2$ as shown in the right-hand side of the figure above.
\end{ex}

\section{Canonical model}\label{section canonical model}
In this section, we construct a topological space $X(Q, \lambda, c)$ from a characteristic data $(Q, \lambda, c)$. In Theorem \ref{canonical model is a T-pseudomanifold}, we will prove that $X(Q, \lambda, c)$ is a locally standard $T$-pseudomanifold. Therefore, we may call $X(Q, \lambda, c)$ a {\it canonical model} of a locally standard $T$-pseudomanifold.
If $Q$ is a simple polytope (note a simple polytope is face acyclic, $c$ is trivial) and $\lambda$ is induced from facets and satisfies the {\it unimodularity condition} (called a {\it characteristic function}), $X(Q, \lambda, 0)$ is the canonical model of a quasitoric manifold over $Q$ in {{\cite[p.423, 1.5. The basic construction]{DJ91}}} or {{\cite[Construction 7.3.5]{BP12}}}.

The following definition is based on \cite{KK25}.

\begin{definition}\label{def of canonical model}
Let $(Q, \lambda, c)$ be a characteristic data (see Definition \ref{def of char data}).
The space $X(Q, \lambda, c)$ is defined as the quotient topological space
\[X(Q, \lambda, c):={P_c}/{\sim}, \]
where $\xi: P_c \to Q$ is a $T^m$-principal bundle over $Q$ whose Chern class is $c \in H^2(Q; \mathbb{Z}^m)$, and the equivalence relation $\sim$ is defined as follows.
Two points $x$ and $y$ in $P_c$ are equivalent (denoted by $x\sim y$ or $x\sim_{\lambda} y$ if we emphasize the characteristic functor $\lambda$) if they satisfy $\xi(x)=\xi(y)=:p$ and $x$ and $y$ lie in the same $\lambda(S)$-orbit when $p \in {S}$ for some stratum $S$.
Moreover, $X(Q, \lambda, c)$ has the canonical $T^m$-action induced by the action on $P_c$.
\end{definition}

\begin{rem}\label{T_S T_p}
Let $\pi : X(Q, \lambda, c) \to Q$ be the orbit projection.
For a point $x \in X(Q, \lambda, c)$, suppose that $\pi(x)=: p \in S$ for some stratum $S \subset Q$. Then, the isotropy group $T_x$ of $x$, which is a subtorus, is $\lambda(S)$.
In view of the analogy with Remark \ref{T_S}, we may also denote $T_x$ by $T_p$ or $T_S$.
\end{rem}

\begin{rem}\label{face acyclic case}
When the second cohomology of $Q$ vanishes, every $T^m$-principal bundle over $Q$ is trivial.
Therefore, we have
\[
X(Q, \lambda, c)
=
X(Q, \lambda, 0)
=
Q \times T^m / {\sim}
\]
where $(p,t) \sim (q,s)$ if $p=q$ and $t^{-1}s \in \lambda(S)$ whenever $p \in S$ for some stratum $S$.
\end{rem}

\begin{rem}
This construction is very natural as a topological model for spaces with a group action.
For instance, when $c=0$, it corresponds to the topological model $D(Y, \varphi)$ for spaces associated with {\it simple complexes of groups} in geometric group theory (see \cite{BH99} and \cite{DLS19}).
Furthermore, this idea underlies the topological models of quasitoric manifolds introduced in \cite{DJ91}, locally standard torus manifolds in \cite{MP06}, moment-angle manifolds, and locally $k$-standard manifolds in \cite{SS21}.
For the case $c \neq 0$, this construction has also been employed in \cite{Yos11} and \cite{KK25}.
However, it is not immediate that these topological models actually preserve the structures of the corresponding original spaces.
In this paper, Section~\ref{section 12} shows that the canonical model admits the structure of a locally standard $T$-pseudomanifold (see Theorem~\ref{canonical model is a T-pseudomanifold}).
\end{rem}

For the filtration
\[
\mathfrak{Q}: Q=Q_{l+n} \supset Q_{l+n-1} \supset \cdots \supset Q_{l} \supsetneq \emptyset,
\]
we can restrict the poset category $\mathcal{S}(Q)$ to the subcategory $\mathcal{S}(Q_{l+i})$ consisting of the strata $S$ of $Q$ such that $S \subset Q_{l+i}$ ($0 \leq i \leq n$), i.e., connected components of $Q_{l+i} \setminus Q_{l+i-1}$. 
The characteristic functor $\lambda: \mathcal{S}(Q)^{\mathrm{op}} \to \mathcal{T}$  restricted to $\mathcal{S}(Q_{l+i})$, denoted by
\[
\lambda_{l+i} := \lambda |_{\mathcal{S}(Q_{l+i})^{\mathrm{op}}} : \mathcal{S}(Q_{l+i})^{\mathrm{op}}  \to \mathcal{T}.
\]
Thus, we can also define the topological space with $T$-action
\[
X(Q_{l+i}, \lambda_{l+i}, c):= (P_c)|_{Q_{l+i}}/{\sim_{\lambda_{l+i}}}:=(P_c)|_{Q_{l+i}}/{\sim},
\]
where $(P_c)|_{Q_{l+i}}$ denotes the restriction of the bundle $P_c$ to $Q_{l+i}$.
It is easy to check that the inclusion $(P_c)|_{Q_{l+i}} \hookrightarrow P_c$ induces a $T$-embedding $X(Q_{l+i}, \lambda_{l+i}, c) \subset X(Q, \lambda, c)$. Therefore, $X(Q, \lambda, c)$ has the following filtration:
\begin{align}\label{filtration of canonical model}
\mathfrak{X} : X(Q, \lambda, c) 
 \supset 
X(Q_{l+n-1}, \lambda_{l+n-1}, c) 
\supset \cdots \supset
X(Q_{l+i}, \lambda_{l+i}, c) 
\supset \cdots \supset
X(Q_l, \lambda_l, c) 
\supsetneq \emptyset ,
\end{align}
where the dimension of $X(Q_{l+i}, \lambda_{l+i}, c)$ is $(l+2i+(m-n))$ for $0 \leq i \leq n$.

\subsection{Basic properties of canonical model}
Let $(Q, \lambda, c)$ be a characteristic data.
In this subsection, using the facts that will be proved in Section \ref{section: model space}, we show that the canonical model is a compact, Hausdorff, and second-countable space.
Let $\pi : X(Q, \lambda, c) \to Q$ be the orbit projection.
\begin{lem}\label{compactness of canonical model}
The canonical model $X(Q, \lambda, c)$ is compact.
\end{lem}
\begin{proof}
Since $Q$ is compact, $T$-principal bundles over $Q$ are also compact.
The quotient map $\rho: P_c \to X(Q, \lambda, c)$ is continuous and surjective.
Since $P_c$ is compact, the image of the quotient map $X(Q, \lambda, c)$ is also compact.
\end{proof}

\begin{lem}\label{Hausdorff of canonical model}
The canonical model $X(Q, \lambda, c)$ is Hausdorff.
\end{lem}
\begin{proof}
The canonical model $X(Q, \lambda, c)$ is equivariantly homeomorphic to the {\it model space} $Y(Q, \lambda, c)$ which will be defined later (see Construction \ref{def of model space}). Since the model space $Y(Q, \lambda, c)$ is Hausdorff (see Lemma \ref{lem of Y is Hausdorff}), the canonical model $X(Q, \lambda, c)$ is also Hausdorff.
\end{proof}

To show that the canonical model $X(Q, \lambda, c)$ is second-countable, we use the following fact.

\begin{lem}[{{\cite[Section 31-Exercises 7 (d)]{Mun14}}}]\label{lem of Munkres 31.7}
Let $X$ and $Y$ be topological spaces, and let $p: X \to Y$ be a closed continuous surjective map such that $p^{-1}(y)$ is compact for each $y \in Y$. If $X$ is second-countable space, then $Y$ is also second-countable space.
\end{lem}

We conclude this subsection by proving that the canonical model is a second-countable space.

\begin{lem}\label{2nd of canonical model}
The canonical model $X(Q, \lambda, c)$ is a second-countable space.
\end{lem}
\begin{proof}
Since $P_c$ is compact and $X(Q, \lambda, c)$ is Hausdorff,
the quotient map $\rho : P_c \to X(Q, \lambda, c)$ is a closed map. 
For any $x \in X(Q, \lambda, c)$, the fiber $\rho^{-1}(x)$ is homeomorphic to $\lambda(S)$, where $\pi(x) \in S \subset Q$.  
Since $\lambda(S)$ is a compact torus, each fiber $\rho^{-1}(x)$ is compact.  
By assumption, $Q$ is second-countable (see Definition~\ref{def of char data}), and hence $P_c$ is also second-countable.  
Therefore, by Lemma~\ref{lem of Munkres 31.7}, the quotient space $X(Q, \lambda, c)$ is second-countable.
\end{proof}

\section{Uniqueness of the canonical model $X(Q, \lambda, c)$}\label{section 9}

In this section, we show that the canonical models of two (weakly) isomorphic characteristic datum are (weakly) $T$-equivariantly homeomorphic.
Namely, we will prove the following theorem.

\begin{thm} \label{thm of equivariant homeo of canonical model}
Let $f: (Q, \lambda, c) \to (Q', \lambda', c')$ be a (weakly) isomorphism between characteristic datum. Then $f$ lifts a (weakly) $T$-equivariant homeomorphism $\widetilde{f}: X(Q, \lambda, c) \to X(Q', \lambda', c')$.
\end{thm}
\begin{proof}
Let $ f:(Q, \lambda, c) \to (Q', \lambda', c')$ be a (weak) isomorphism between characteristic datum, and let $\psi: T^m \to T^m$ be the automorphism associated with $f$ (see Definition~\ref{def of isomorphism of char pairs}).
Let $\xi: P_c \to Q$ and $\xi': P_{c'} \to Q'$ be $T^m$-principal bundles whose Chern classes are $c$ and $c'$, respectively. 
Consider the pull-back of $\xi': P_{c'} \to Q'$ along the homeomorphism $f: Q \to Q'$:
\[
\begin{tikzcd}[ampersand replacement=\&]
f^* P_{c'} \arrow[r] \arrow[d] \& P_{c'} \arrow[d, "\xi'"] \\
Q \arrow[r,"f"] \& Q'
\end{tikzcd}
\]
The pull-back bundle $f^* P_{c'} \to Q$ is a $T^m$-principal bundle whose 
Chern class is $f^*(c')$. By the definition of a (weak) isomorphism, 
we have $f^*(c') = c$, so that $P_c \cong f^* P_{c'}$. 
It follows that the homeomorphism $f: Q \to Q'$ induces a $T^m$-bundle isomorphism
\[
\alpha: P_c \to P_{c'},
\]
which is, in particular, a $T^m$-equivariant homeomorphism.
By using the automorphism $\psi:T^m \to \psi(T^m) \cong T^m$ on each fiber of $P_{c'}$, we obtain a $\psi$-weakly equivariant homeomorphism
\[
\beta: P_{c'} \to P_{c'}.
\]
We then define a $\psi$-weakly equivariant homeomorphism (note that if $\psi = \mathrm{id}$, this is an equivariant homeomorphism)
\[
\widehat{f} := \beta \circ \alpha: P_c \to P_{c'}.
\]
We represent a point $x \in P_c$ (resp.\ $x' \in P_{c'}$) as $(p,t)$ (resp.\ $(p',t')$), where $p = \xi(x)$, $t \in \xi^{-1}(p) \cong T$, and $p' = \xi'(x')$, $t' \in \xi'^{-1}(p') \cong T$.  
When we write $\widehat{f}(p,t) = (p',t')$, note that $p' = f \circ \xi(p,t) = f(p)$.
Let
\[
\begin{array}{rccc}
q_{X(Q, \lambda, c)}: & P_c & \to & X(Q, \lambda, c) \\ 
& (p, t) & \mapsto & [p, t]_{\lambda}
\end{array}
\quad \text{and} \quad
\begin{array}{rccc}
q_{X(Q', \lambda', c')}: & P_{c'} & \to & X(Q', \lambda', c') \\ 
& (p', t') & \mapsto & [p', t']_{\lambda'}
\end{array}
\]
be the quotient maps defining the respective canonical models.
We then define a map
\[
\begin{array}{rccc}
\widetilde{f}: & X(Q, \lambda, c) &\to& X(Q', \lambda', c') \\ 
& [p, t]_{\lambda} & \mapsto & [\widehat{f}(p,t)]_{\lambda'}.
\end{array}
\]
In the following four steps, we show that this map is well-defined and is a $\psi$-weakly equivariant homeomorphism.

\paragraph{Step 1: $\widetilde{f}$ is well-defined.}
Suppose that $[p,t]_{\lambda} = [q,s]_{\lambda}$.  
By the definition of the equivalence relation $\sim_{\lambda}$ on $P_c$, we have
\[
[p,t]_{\lambda} = [q,s]_{\lambda} \iff 
\begin{cases}
  p = q, \\[1mm]
  (p,t) \text{ and } (q,s) \text{ lie in the same } \lambda(S)\text{-orbit for some stratum } S \ni p.
\end{cases}
\]
Then there exists $u \in \lambda(S)$ satisfying
\[
(q,s) = u \cdot (p,t).
\]
Since $\widehat{f}$ is a $\psi$-weakly equivariant homeomorphism, it follows that
\[
\widehat{f}(q,s) = \widehat{f}(u \cdot (p,t)) = \psi(u) \cdot \widehat{f}(p,t).
\]
Thus $\widehat{f}(q,s)=:(f(q), s')$ and $\widehat{f}(p,t)=:(f(p), t')$ lie in the same $\psi(\lambda(S))$-orbit.
By the definition of the weak isomorphism $f$, we have $\psi(\lambda(S)) = \lambda'(f(S))$.  
The equivalence relation $\sim_{\lambda'}$ on $P_{c'}$ is given by
\begin{align*}
&[f(p), t']_{\lambda'} = [f(q), s']_{\lambda'} \\ &\iff
\begin{cases}
f(p) = f(q), \\[1mm]
(f(p), t') \text{ and } (f(q), s') \text{ lie in the same } \lambda'(S')\text{-orbit for some stratum } S' \ni f(p).
\end{cases}
\end{align*}
Hence, we obtain
\[
[\widehat{f}(p,t)]_{\lambda'} = [\widehat{f}(q,s)]_{\lambda'},
\]
which shows that $\widetilde{f}$ is well-defined.

\paragraph{Step 2: $\widetilde{f}$ is bijective.}
We define the inverse map
\[
\begin{array}{rccc}
& \widetilde{f}^{-1}: X(Q', \lambda', c') &\to& X(Q, \lambda, c) \\
& \rotatebox{90}{$\in$} & & \rotatebox{90}{$\in$} \\
& [p', t']_{\lambda'} &\mapsto& [\widehat{f}^{-1}(p',t')]_{\lambda}.
\end{array}
\]
Since $\widehat{f}^{-1}: P_{c'} \to P_c$ is a $\psi^{-1}$-weakly equivariant homeomorphism, an analogous argument as in Step~1 shows that $\widetilde{f}^{-1}$ is well-defined.  
Therefore, $\widetilde{f}$ is bijective.

\paragraph{Step 3: $\widetilde{f}$ is a homeomorphism.}
By the definition of $\widetilde{f}$, the following diagram commutes:
\[
\begin{tikzcd}[ampersand replacement=\&]
P_c \arrow[r, "\widehat{f}"', "\cong"] \arrow[d, "q_{X(Q, \lambda, c)}"] \&
  P_{c'} \arrow[d, "q_{X(Q', \lambda', c')}"] \\
X(Q, \lambda, c) \arrow[r, "\widetilde{f}"] \&
  X(Q', \lambda', c').
\end{tikzcd}
\]
Since $\widehat{f}$ is a homeomorphism and $\widetilde{f}$ is the induced map between the quotient spaces, it follows from the commutativity of the diagram that $\widetilde{f}$ is also a homeomorphism.

\paragraph{Step 4: $\widetilde{f}$ is $\psi$-weakly equivariant.}
It follows from the commutative diagram above that for any $s \in T^m$, we have
\begin{align*}
\widetilde{f}(s \cdot [p,t]_{\lambda})
&= [\widehat{f}(s \cdot (p,t))]_{\lambda'} \\
&= [\psi(s) \cdot \widehat{f}(p,t)]_{\lambda'} \\
&= \psi(s) \cdot [\widehat{f}(p,t)]_{\lambda'} \\
&= \psi(s) \cdot \widetilde{f}([p,t]_{\lambda}).
\end{align*}
Therefore, $\widetilde{f}$ is $\psi$-weakly equivariant.
\end{proof}

\section{Model space $Y(Q, \lambda, c)$}\label{section: model space}
Let $(Q, \lambda, c)$ be a characteristic data (see Definition~\ref{def of char data}).
To prove the classification theorem (Theorem \ref{classification theorem intro}), we construct a {\it model space $Y(Q, \lambda, c)$} that is equivariantly homeomorphic to the canonical model $X(Q, \lambda, c)$ (see Lemma \ref{lem a}).
The construction of the model space follows ideas from {{\cite[p.9, Construction 3.5]{Ayz18}}}.
The name {\it model space} follows the terminology used by Ayzenberg in {\cite{Ayz18}}.

\subsection{Small open neighborhood}
In this subsection, to simplify the arguments, we introduce a {\it small open neighborhood} in a topological stratified pseudomanifold.

\begin{definition}[small open neighborhood]\label{def of small open nbh}
Let $Q$ be a topological stratified pseudomanifold and $S \subset Q$ be a stratum. For $p \in {S} \subset Q$, an open neighborhood $U_p$ of $p$ in $Q$ is said to be {\it small} if for any $p' \in U_p$, the stratum $S'$ such that $\overline{S'} \ni p'$ satisfies that ${S} \subset \overline{S'}$, where $\overline{S'}$ denotes the closure of $S'$.
\end{definition}

\begin{ex}
In the three figures below, the left open neighborhood of $x$ is small, while the others are not small. However, the right one is not a small open neighborhood of $x$, but it is a small open neighborhood of $v$.

\begin{tikzpicture}[scale=0.6]
    \fill[gray!20] (0,0) -- (6,0) -- (4,2.5) -- cycle; 
    \draw[thick] (0,0) -- (6,0) -- (4,2.5) -- cycle; 

    \fill (0,0) circle (2pt);
    \fill (6,0) circle (2pt);
    \fill (4,2.5) circle (2pt);

    \fill[red!20] (2.75,0) arc[start angle=0,end angle=180,radius=0.75] -- cycle;

    \draw[dashed] (2.75,0) arc[start angle=0,end angle=180,radius=0.75];

    \draw (2.75,0) circle (2pt);
    \draw (1.25,0) circle (2pt);

    \fill (2,0) circle (2pt);

    \node at (0,-0.3) {$v$};
    \node at (2,-0.3) {$x$};
   \node at (4.5,-0.3) {$S$};
    \node at (3,-1) {small};
    \node[red] at (2.85, 0.75) {$U_x$};

\begin{scope}[xshift=7.5cm]
    \fill[gray!20] (0,0) -- (6,0) -- (4,2.5) -- cycle; 
    \draw[thick] (0,0) -- (6,0) -- (4,2.5) -- cycle; 

    \fill (0,0) circle (2pt);
    \fill (6,0) circle (2pt);
    \fill (4,2.5) circle (2pt);

    \fill[red!20] (2.01, 1.257) -- (0.87, 0.544) -- (3.25, 0) arc[start angle=0,end angle=89.55,radius=1.25] -- cycle;
    \fill[red!20] (3.25, 0) -- (0.75, 0) -- (0.75, 0) arc[start angle=180,end angle=154.43,radius=1.25] -- cycle;
    \draw[dashed] (3.25,0) arc[start angle=0,end angle=89.55,radius=1.25];
    \draw[dashed] (0.75,0) arc[start angle=180,end angle=154.43,radius=1.25];

    \draw (3.25,0) circle (2pt);
    \draw (0.75,0) circle (2pt);
    \draw (2.0098, 1.25) circle (2pt);
    \draw (0.874, 0.543) circle (2pt);

    \fill (2,0) circle (2pt);

 \node at (0,-0.3) {$v$};
    \node at (2,-0.3) {$x$};
  \node at (4.5,-0.3) {$S$};
    \node at (3,-1) {NOT small};
\end{scope}

\begin{scope}[xshift=15cm]
    \fill[gray!20] (0,0) -- (6,0) -- (4,2.5) -- cycle; 
    \draw[thick] (0,0) -- (6,0) -- (4,2.5) -- cycle; 


    \fill (6,0) circle (2pt);
    \fill (4,2.5) circle (2pt);

    \fill[red!20] (0,0)  -- (2.25, 0) arc[start angle=0,end angle=32,radius=2.25] -- cycle;

    \draw[dashed] (2.25,0) arc[start angle=0,end angle=32,radius=2.25];

    \draw (2.25,0) circle (2pt);
    \draw (1.904,1.197) circle (2pt);

    \fill (0,0) circle (2pt);

    \fill (2,0) circle (2pt);

 \node at (0,-0.3) {$v$};
    \node at (2,-0.3) {$x$};
  \node at (4.5,-0.3) {$S$};
    \node at (3,-1) {NOT small};

\end{scope}
\end{tikzpicture}
\end{ex}

\begin{rem}\label{rem: small nbh}
By Lemma \ref{lem of small open nbh} and Remark \ref{rem of small open nbh}, if the topological stratified pseudomanifold is compact, then for each point there exists a contractible, small open neighborhood.
\end{rem}

\subsection{Construction and remark}

We now construct a {\it model space $Y(Q, \lambda, c)$}. 

\begin{cons}[model space]\label{def of model space}
Let $(Q, \lambda, c)$ be a characteristic data.
We construct a {\it model space $Y(Q, \lambda, c)$} as follows.
As an abstract set, we put
\[
Y(Q, \lambda, c)=\bigsqcup_{p \in Q}T/T_p,
\]
where $T_p=\lambda(S)\, (\cong T^{n-i})$, where $S$ is a stratum in $Q$ which contains $p$.
Define a topology on $Y(Q, \lambda, c)$ as follows. For any $y \in Y(Q, \lambda, c)=\bigsqcup_{p \in Q}T/T_p$, $y$ can be denoted by $(p, [t_p])$ for some $p \in Q$ and $[t_p] \in T/T_p$. Let $U_p$ denote a contractible, small open neighborhood of $p$ in $Q$ (see Definition \ref{def of small open nbh} and Remark \ref{rem: small nbh}). We also take an open neighborhood $V$ of $[t_p]$ in $T/T_p$. Then for each $p' \in U_p$, we define the natural projection
\[
q_{p'}:T/T_{p'} \to T/T_p.
\]
This projection is well-defined since $U_p$ is small so that $T_p \supset T_{p'}$. We define the subset of $Y(Q, \lambda, c)$ of the form
\[
\bigsqcup_{p' \in U_p}q_{p'}^{-1}(V)
\]
as a basis for the topology around $y=(p,[t_p])$.
Moreover, we define a $T$-action on $Y(Q, \lambda, c)$ as follows:
\begin{align}\label{T-action on model space}
\begin{array}{rccc}&T \times Y(Q, \lambda, c)&\to&Y(Q, \lambda, c)\\
  & \rotatebox{90}{$\in$}&               & \rotatebox{90}{$\in$} \\
&\left( s,(p,[t_p])\right) &\mapsto&\left( p,[st_p] \right)
\end{array}
\end{align}
We can easily show that its orbit space is $Q$. 
\end{cons}

We begin by proving the following lemma.

\begin{lem}
Let $(Q, \lambda, c)$ be a characteristic data.
The $T$-action \eqref{T-action on model space} on the model space $Y(Q, \lambda, c)$ is continuous.
\end{lem}
\begin{proof}
Let $p \in Q$ and $V$ be an open subset in $T/T_p$.
We define $\varphi_s:Y(Q, \lambda, c) \to Y(Q, \lambda, c)$ as the bijective map given by $(p,[t_p]) \mapsto (p,[st_p])$, where $s\in T$.
For each open basis $\bigsqcup_{p' \in U_p}q_{p'}^{-1}(V)$ in $Y(Q, \lambda, c)$, we claim that there is an open subset $W$ in $T/T_p$ such that 
\[
\varphi_s^{-1}  \left( \bigsqcup_{p' \in U_p}q_{p'}^{-1}(V) \right)  =\bigsqcup_{p' \in U_p}q_{p'}^{-1}(W).
\]

Consider the following commutative diagram:
\begin{align*}
\begin{tikzcd}[ampersand replacement=\&]
T/T_{p'} \ar{r}{q_{p'}} \ar{d}{s} \& T/T_{p} \ar{d}{s}\\
  T/T_{p'} \ar{r}{q_{p'}} \& T/T_{p},
\end{tikzcd}
\end{align*}
where $s[t_p]=[s t_p]$ for $[t_p] \in T/T_p$. By the commutativity, we have
 \[\bigsqcup_{p' \in U_p}q_{p'}^{-1}\left(s^{-1} (V) \right) = \bigsqcup_{p' \in U_p} {s}^{-1} \left( q_{p'}^{-1}(V) \right).\]
Since $s^{-1}(V)$ is an open subset in $T/T_p$, we may put $W=s^{-1}(V)$. 
Therefore, we have
\begin{align*}
\bigsqcup_{p' \in U_p}q_{p'}^{-1}\left(W \right)
=
\bigsqcup_{p' \in U_p} {s}^{-1} \left( q_{p'}^{-1}(V) \right)
=
\varphi_{s^{-1}}  \left( \bigsqcup_{p' \in U_p}q_{p'}^{-1}(V) \right) 
=
\varphi_s^{-1}  \left( \bigsqcup_{p' \in U_p}q_{p'}^{-1}(V) \right).
\end{align*}
This shows that $\varphi_{s}$ is continuous. This establishes the statement.
\end{proof}

\begin{rem}\label{rem of trivial}
If $p \in Q_{l+n} \setminus Q_{l+n-1}$, we can take an open neighborhood $U_p \subset Q_{l+n} \setminus Q_{l+n-1}$. Then for any $p' \in U_p$, we have $T/T_{p'}=T$. Hence, the projection $q_{p'} : T/T_{p'} \to T/T_p$ is the identity map on $T$. This implies that the subspace
\[
\bigsqcup_{p \in Q_{l+n} \setminus Q_{l+n-1}}T
\]
of the model space $Y(Q, \lambda, c)$ is a $T$-principal bundle over $Q_{l+n} \setminus Q_{l+n-1}$.
By assumption, the top strata $Q_{l+n} \setminus Q_{l+n-1}$ is homotopy equivalent to $Q$ itself.
Therefore, the inclusion map $i: Q_{l+n} \setminus Q_{l+n-1} \to Q$ induces an isomorphism
\[
i^{\ast}: H^2(Q; \mathbb{Z}^m) \to H^2(Q_{l+n} \setminus Q_{l+n-1}; \mathbb{Z}^m).
\]
The $T$-principal bundle over $Q_{l+n} \setminus Q_{l+n-1}$ is the restriction of the $T$-principal bundle over $Q$.
Since the $T$-principal bundle over $Q$ is classified by the Chern class $c \in H^2(Q; \mathbb{Z}^m)$, the $T$-principal bundle over $Q_{l+n} \setminus Q_{l+n-1}$ is also classified by $c$.
\end{rem}

In Remark~\ref{rem of trivial}, we observed that the subspace
$
\bigsqcup_{p \in Q_{l+n} \setminus Q_{l+n-1}} T
$
of the model space $Y(Q, \lambda, c)$ is the $T$-principal bundle over $Q_{l+n} \setminus Q_{l+n-1}$ whose Chern class is $c \in H^2(Q; \mathbb{Z}^m)$.  
Let 
$
\xi \colon P_c \to Q
$
denote the $T$-principal bundle whose Chern class is $c$.
We represent a point $x \in P_c$ as $(p,t)$, where $p = \xi(x)$ and $t \in \xi^{-1}(p)$.
The $T$-action \eqref{T-action on model space} on $Y(Q, \lambda, c)$ is induced from the $T$-action on $P_c$ as follows:
For each $p \in Q$, consider the two projections
\[
\xi: P_c \to Q, \quad \tau : Y(Q, \lambda, c) \to Q.
\]
The fiber $\xi^{-1}(p)$ is $T$-equivariantly homeomorphic to $T$, and 
\[
\tau^{-1}(p) = \bigsqcup_{p \in \{ p \}} T/T_p
\]
is $T$-equivariantly homeomorphic to $T/T_p$.
We define a map
\[
\rho_p: \xi^{-1}(p) \to \tau^{-1}(p)
\]
such that the following diagram commutes:
\begin{align}\label{comm: P_c to Y}
\begin{tikzcd}[ampersand replacement=\&]
\xi^{-1}(p) \ar{r}{\rho_p} \ar{d}{\theta_p}[']{\cong} \& \tau^{-1}(p) \ar{d}{\psi_p}[']{\cong} \\
T \rar{\pi_p} \& T/T_p
\end{tikzcd}
\end{align}
where $\theta_p$ and $\psi_p$ are $T$-equivariant homeomorphisms, and $\pi_p$ is the natural projection. Since $\pi_p$ is also $T$-equivariant, it follows that $\rho_p$ is also a $T$-equivariant map.
Then we define
\begin{align}\label{map: P_c to Y}
\begin{array}{rccc}
q_{Y(Q, \lambda, c)}: & P_c & \to & Y(Q, \lambda, c) \\
& \rotatebox{90}{$\in$} & & \rotatebox{90}{$\in$} \\
& (p, t) & \mapsto & (p, \rho_{p}(t))
\end{array}
\end{align}
Identifying $\xi^{-1}(p)$ with $T$ and $\tau^{-1}(p)$ with $T/T_p$, and writing $\rho_p(t) = [t_p]$ for $t \in T$, this map becomes
\[
\begin{array}{rccc}
q_{Y(Q, \lambda, c)}: & P_c & \to & Y(Q, \lambda, c) \\
& \rotatebox{90}{$\in$} & & \rotatebox{90}{$\in$} \\
& (p, t) & \mapsto & (p, [t_p])
\end{array}
\]
For $s \in T$, the action \eqref{T-action on model space} defined in Construction~\ref{def of model space}
\[
s \cdot (p, [t_p]) = (p, [s t_p])
\]
is given by
\begin{align}\label{action: P_c to Y}
s \cdot (p, \rho_p(t)) = (p, \rho_p(s t)).
\end{align}

\subsection{Basic properties}

We next describe the basic properties of the model space, beginning with the Hausdorff property.

\begin{lem}\label{lem of Y is Hausdorff}
The model space $Y(Q, \lambda, c)$ is Hausdorff.
\end{lem}
\begin{proof}
Let $\tau : Y(Q, \lambda, c) \to Q$ be the orbit projection. 
Consider two points
$(p_1, [t_{p_1}]), (p_2, [t_{p_2}]) \in Y(Q, \lambda, c)$ with $(p_1, [t_{p_1}]) \neq (p_2, [t_{p_2}])$.
In the case where
$p_1 \neq p_2$, since $Q$ is Hausdorff, there exist open sets $U_1$ and $U_2$ such that $p_1 \in U_1$, $p_2 \in U_2$ and $U_1 \cap U_2 = \emptyset$. 
The preimages 
$\tau^{-1}(U_{1})=\bigsqcup_{p \in U_{1}} T/T_p$ and $\tau^{-1}(U_{2})=\bigsqcup_{p \in U_{2}} T/T_p$
are open subsets in $Y(Q, \lambda, c)$. Since $U_{1}\cap U_{2}=\emptyset$, we have
\[
\bigsqcup_{p \in U_{1}} (T/T_p) \cap \bigsqcup_{p \in U_{2}} (T/T_p) =\emptyset.
\]
In the case where $p_1 = p_2=p$, we have $[t_{p_1}] \neq [t_{p_2}] \in T/T_p$. Since $T/T_p$ is Hausdorff, there exist open subsets $V_1$ and $V_2$ in $T/T_p$ such that $[t_{p_1}] \in V_1$, $[t_{p_2}] \in V_2$ and $V_1 \cap V_2 = \emptyset$. Then, the sets
$\bigsqcup_{p' \in U_p} q_{p'}^{-1}(V_1)$ and $\bigsqcup_{p' \in U_p} q_{p'}^{-1}(V_2)$ are open subsets in $Y(Q, \lambda, c)$. Since $V_1 \cap V_2 = \emptyset$, 
\[
q_{p'}^{-1}(V_1) \cap q_{p'}^{-1}(V_2) = q_{p'}^{-1}(V_1 \cap V_2) = q_{p'}^{-1}(\emptyset) = \emptyset.
\]
Thus,
we have 
\[
\bigsqcup_{p' \in U_p} q_{p'}^{-1}(V_1) \cap \bigsqcup_{p' \in U_p} q_{p'}^{-1}(V_2) = \emptyset.
\]
\end{proof}

We conclude this section by proving the following lemma.

\begin{lem}\label{lem a}
For a characteristic data $(Q, \lambda, c)$, the model space $Y(Q, \lambda, c)$ is equivariantly homeomorphic to the canonical model $X(Q, \lambda, c)$.
\end{lem}

\begin{proof}
Let $\xi: P_c \to Q$ be the principal $T$-bundle whose Chern class is $c$.
We represent a point $x \in P_c$ as $(p,t)$, where $p = \xi(x)$ and $t \in \xi^{-1}(p)$.
Let 
\[
\begin{array}{rccc}
q_{X(Q, \lambda, c)}: & P_c & \to & X(Q, \lambda, c) \\
& \rotatebox{90}{$\in$} & & \rotatebox{90}{$\in$} \\
& (p, t) & \mapsto & [p, t]
\end{array}
\]
be the quotient map induced by the equivalence relation defining $X(Q, \lambda, c)$.
Let
\[
q_{Y(Q, \lambda, c)}: P_c \to Y(Q, \lambda, c)
\]
be the map defined in \eqref{map: P_c to Y}.
We now define a map
\[
\begin{array}{rccc}
f: & X(Q, \lambda, c) & \to & Y(Q, \lambda, c) \\
& \rotatebox{90}{$\in$} & & \rotatebox{90}{$\in$} \\
& [p, t] & \mapsto & (p, [t]).
\end{array}
\]
In the following four steps, we show that this map is well-defined and is a $T$-equivariant homeomorphism.

\paragraph{Step 1: $f$ is well-defined.}
By the definition of the canonical model $X(Q, \lambda, c)$, for $(p,t), (q,s) \in~P_c$,
\[
[p,t]=[q,s] \iff 
\begin{cases}
  p = q, \\
  \text{$(p,t)$ and $(q,s)$ lie in the same } \lambda(S) = T_p \text{-orbit for some stratum } S \ni p.
\end{cases}
\]
Assuming $p = q$, we have $t, s \in \xi^{-1}(p)$, hence $\theta_p(t), \theta_p(s) \in T$ (see the commutative diagram \eqref{comm: P_c to Y}).
The second condition then implies that $\theta_p(t)^{-1} \theta_p(s) \in T_p$.
On the other hand, we have
\[
(p,[t])= (p, \rho_p(t)) = (q, \rho_q(s)) =(q,[s]) \iff p = q \text{ and } \rho_p(t) = \rho_q(s).
\]
Assuming $p = q$, we have $\rho_p(t) = \rho_p(s)$.
Since $\psi_p$ in \eqref{comm: P_c to Y} is a $T$-equivariant homeomorphism, the condition $\rho_p(t) = \rho_p(s)$ is equivalent to
\[
\psi_p \circ \rho_p(t) = \psi_p \circ \rho_p(s) \in T/T_p.
\]
By the commutativity of \eqref{comm: P_c to Y}, we have $\psi_p \circ \rho_p = \pi_p \circ \theta_p$, where $\pi_p: T \to T/T_p$ is the natural projection.
Thus, the above condition $\rho_p(t)=\rho_p(s)$ is equivalent to $\pi_p \circ \theta_p(t)=\pi_p \circ \theta_p(s)$, which holds if and only if $\theta_p(t)^{-1} \theta_p(s) \in T_p$.
Hence, if $[p,t]=[q,s]$, then $(p, \rho_p(t))=(q, \rho_q(s))$.
Therefore, the map $f$ is well-defined.

\paragraph{Step 2: $f$ is bijective.}
We define a map
\[
\begin{array}{rccc}
f^{-1}: & Y(Q, \lambda, c) & \to & X(Q, \lambda, c) \\
& \rotatebox{90}{$\in$} & & \rotatebox{90}{$\in$} \\
& (p, \rho_p(t)) & \mapsto & [p, t]
\end{array}
\]
By the argument in Step 1, this map is well-defined and is the inverse of $f$.  
Therefore, $f$ is bijective.

\paragraph{Step 3: $f$ is $T$-equivariant.}
By the definition of the $T$-action on the canonical model $X(Q, \lambda, c)$ and on the model space $Y(Q, \lambda, c)$ (see \eqref{action: P_c to Y}), and by the definition of $f$, for any $s \in T$, we have
\begin{align*}
f(s \cdot [p,t]) &= f([p, s t]) \\
&= (p, \rho_p(s t)) \\
&= s \cdot (p, \rho_p(t)) \\
&= s \cdot f([p,t]).
\end{align*} 
Thus, $f$ is $T$-equivariant.

\paragraph{Step 4: $f$ is a homeomorphism.}
We first prove that $f$ is continuous.  
Let $\bigsqcup_{p' \in U_p} q_{p'}^{-1}(V)$ be a basic open subset in $Y(Q, \lambda, c)$, where $V \subset T/T_p$ is open and $U_p \subset Q$ is a contractible, small open neighborhood of $p$.
By the commutativity of the following diagram:
\[
\begin{tikzcd}[ampersand replacement=\&]
\pi_{p}^{-1}(V) \arrow[r, phantom, "\subset"] \& T \arrow[twoheadrightarrow]{rr}{\pi_p} \arrow[twoheadrightarrow]{rd}{\pi_{p'}} \& \& T/T_p \arrow[r, phantom, "\supset"] \& V \\
\& \& T/T_{p'} \arrow[twoheadrightarrow]{ru}{q_{p'}} \& \&
\end{tikzcd}
\]
the preimage under $q_{Y(Q, \lambda, c)}$ is given by
\[
q_{Y(Q, \lambda, c)}^{-1}\left( \bigsqcup_{p' \in U_p} q_{p'}^{-1}(V) \right) = \{ (p', t) \in P_c \mid p' \in U_p,\ t \in \pi_p^{-1}(V) \}.
\]
Since $U_p$ is contractible open and $\pi_p^{-1}(V) \subset T$ is open, this is an open subset of $P_c$ homeomorphic to $U_p \times \pi_p^{-1}(V)$ (so $q_{Y(Q, \lambda, c)}$ is continuous).  
We now consider the commutative diagram:
\[
\begin{tikzcd}[ampersand replacement=\&]
\& P_c \ar{ld}[swap]{q_{X(Q, \lambda, c)}} \ar{rd}{q_{Y(Q, \lambda, c)}} \& \\
X(Q, \lambda, c) \ar{rr}{f} \& \& Y(Q, \lambda, c)
\end{tikzcd}
\]
It follows that
\[
q_{X(Q, \lambda, c)}^{-1} \left( f^{-1} \left( \bigsqcup_{p' \in U_p} q_{p'}^{-1}(V) \right) \right)
= q_{Y(Q, \lambda, c)}^{-1} \left( \bigsqcup_{p' \in U_p} q_{p'}^{-1}(V) \right),
\]
which is open in $P_c$.  
Because $X(Q, \lambda, c)$ is equipped with the quotient topology, this implies that $f^{-1} \left( \bigsqcup_{p' \in U_p} q_{p'}^{-1}(V) \right)$ is open in $X(Q, \lambda, c)$.
Therefore, $f$ is continuous.
Since $X(Q, \lambda, c)$ is compact (by Lemma~\ref{compactness of canonical model}) and $Y(Q, \lambda, c)$ is Hausdorff (by Lemma~\ref{lem of Y is Hausdorff}), the continuous bijection $f$ is a homeomorphism.
\end{proof}

\begin{rem}\label{canonical model is Hausdorff}
By this lemma, since the model space $Y(Q, \lambda, c)$ is a Hausdorff space, it follows that the canonical model $X(Q, \lambda, c)$ is also a Hausdorff space (see Lemma \ref{Hausdorff of canonical model}).
\end{rem}

\section{Classification}\label{section 11}
In this section, we prove the classification theorem which is the main theorem of this paper:

\begin{thm}[classification theorem]\label{classification theorem}
Let $X$, $X'$ be locally standard $T$-pseudomanifolds, and $Q$, $Q'$ be their orbit spaces, respectively. Assume that the top strata $Q_{l+n}\setminus Q_{l+n-1}$ (resp. $Q'_{l+n}\setminus Q'_{l+n-1}$) of $Q$ (resp. $Q'$) is homotopy equivalent to $Q$ (resp. $Q'$) itself. Then, the following two statements are equivalent:
\begin{enumerate} \item The characteristic data $(Q,\lambda,c)$ of $X$ and the characteristic data $(Q',\lambda',c')$ of $X'$ are (weakly) isomorphic; \item $X$ and $X'$ are (weakly) $T$-equivariantly homeomorphic.
\end{enumerate}
\end{thm}

\subsection{Key lemma}
In this subsection, we prove the following key lemma toward the classification theorem.

\begin{lem}\label{main thm}
Let $X$ be a locally standard $T$-pseudomanifold. Assume that the top strata $Q_{l+n}\setminus Q_{l+n-1}$ of the orbit space $Q=X/T$ is homotopy equivalent to $Q$ itself. Then, $X$ is $T$-equivariantly homeomorphic to the canonical model $X(Q,\lambda,c)$, where $(Q,\lambda, c)$ is the characteristic data of $X$.
\end{lem}

To prove Lemma \ref{main thm}, we establish the following lemma.
The proof of the following lemma is inspired by the argument in \cite[Proposition 3.7]{Ayz18}.

\begin{lem}\label{lem b}
Let $X$ be a locally standard $T$-pseudomanifold. Assume that the top strata $Q_{l+n}\setminus Q_{l+n-1}$ of the orbit space $Q=X/T$ is homotopy equivalent to $Q$ itself.
Then, the model space $Y(Q, \lambda, c)$ is $T$-equivariantly homeomorphic to $X$, where $(Q,\lambda, c)$ is the characteristic data of $X$.
\end{lem}
\begin{proof}
Let $\pi: X \to Q$ and $\tau: Y(Q, \lambda, c) \to Q$ be orbit projections.
For any subset $A \subset Q$, we use the notation $\mathring{A}:= A \setminus Q_{l+n-1}$.
Since $\tau^{-1}(\mathring{Q}) \subset Y(Q, \lambda, c)$ and $\pi^{-1}(\mathring{Q}) \subset X$ are $T$-principal bundles over $\mathring{Q}$, there exist homotopic continuous maps $f, g: \mathring{Q} \to BT$ such that
\[
\tau^{-1}(\mathring{Q}) \cong f^{\ast}ET, \quad \pi^{-1}(\mathring{Q}) \cong g^{\ast} ET,
\]
where $BT$ is the classifying space and $ET$ is the universal bundle.
We then choose bundle isomorphisms
\[
\alpha: \tau^{-1}(\mathring{Q}) \to f^{\ast}ET, \quad \beta: \pi^{-1}(\mathring{Q}) \to g^{\ast} ET.
\]
Since $\iota: \mathring{Q} \hookrightarrow Q$ is a homotopy equivalence, any map from $\mathring{Q}$ to $BT$ extends to $Q$ uniquely up to homotopy. Thus, there exist homotopic maps $F, G: Q \to BT$ such that
\[
f = F \circ \iota, \quad g = G \circ \iota,
\]
and the following diagrams commute:
\[
\begin{tikzcd}[ampersand replacement=\&]
\tau^{-1}(\mathring{Q})   \ar{r}{\alpha}[']{\cong}        
\&
f^{\ast}ET              \ar[r, hook, "\subset"]                 \ar{d}
\&
F^{\ast}ET              \ar{r}{\overline{F}}       \ar{d}
\&
ET   \ar{d}
\\
   \&
\mathring{Q}     \ar[r, hook, "\subset"]
\&
Q              \ar{r}{F}
\&
BT
\\
\end{tikzcd}
\]
\[
\begin{tikzcd}[ampersand replacement=\&]
\pi^{-1}(\mathring{Q})   \ar{r}{\beta}[']{\cong}        
\&
g^{\ast}ET              \ar[r, hook, "\subset"]                    \ar{d}
\&
G^{\ast}ET              \ar{r}{\overline{G}}       \ar{d}
\&
ET   \ar{d}
\\
   \&
\mathring{Q}     \ar[r, hook, "\subset"]
\&
Q              \ar{r}{G}
\&
BT
\end{tikzcd}
\]
where $\overline{F}$ and $\overline{G}$ are the natural projections of the pull-back bundles.
In what follows, we use the following notation:
\[
P_{\tau}:=F^{\ast}ET,
\quad
P_{\tau}|_{\mathring{Q}}:=f^{\ast}ET,
\quad
P_{\pi}:=G^{\ast}ET,
\quad
P_{\pi}|_{\mathring{Q}}:=g^{\ast}ET.
\]
Since the Chern classes of $P_{\tau}$ and $P_{\pi}$ are both equal to $c \in H^2(Q; \mathbb{Z}^m)$, there exists a bundle isomorphism
\[
H: P_{\tau} \to P_{\pi}.
\]
We then define a bundle isomorphism
\[
h: \tau^{-1}(\mathring{Q}) \to \pi^{-1}(\mathring{Q})
\]
such that the following diagram commutes:
\[
\begin{tikzcd}[ampersand replacement=\&]
P_{\tau}              \rar{H}[']{\cong}   
\& P_{\pi} 
\\
P_{\tau}|_{\mathring{Q}}                \rar{H|_{\mathring{Q}}}[']{\cong}  
\ar[u, hook, "{\rotatebox{90}{$\subset$}}"]
\& P_{\pi}|_{\mathring{Q}}               \ar[u, hook, "{\rotatebox{90}{$\subset$}}"]
\\
\tau^{-1}(\mathring{Q})           \rar[red]{h}[']{\cong}      \ar{u}{\alpha}[']{\cong} \& \pi^{-1}(\mathring{Q})                     \ar{u}{\beta}[']{\cong} 
\end{tikzcd}
\]
We claim that $h$ extends to a $T$-equivariant homeomorphism 
\[
\widehat{h} : Y(Q, \lambda, c)=\tau^{-1}(Q) \to \pi^{-1}(Q)=X.
\]
We show this in the following 3 steps.

\begin{enumerate}[label=\underline{{Step} \arabic*.}, leftmargin=*]
    \item 
For a point $p \in Q_{l+n-1}$, let $ U_p $ be a contractible, small open neighborhood of $ p $ in $Q$.
By choosing $U_p$ sufficiently small, we may assume that local triviality holds.
By restricting the bundles to $U_p$ and taking the quotients by $T_p$, we obtain the following commutative diagram:
\[
\begin{tikzcd}[ampersand replacement=\&]
P_{\tau}|_{U_p}/T_p                   \rar{H|_{U_p}/T_p}
\&
P_{\pi}|_{U_p}/T_p
\\
\tau^{-1}(\mathring{U_p})/T_p           \rar{h|_{\mathring{U_p}}/T_p}      \uar{\alpha'_{p}}
\&
\pi^{-1}(\mathring{U_p})/T_p                       \uar{\beta'_{p}}
\end{tikzcd}
\]
where $\alpha'_{p}$ and $\beta'_{p}$ are the injective bundle morphisms obtained by restricting $\alpha$ and $\beta$, taking the quotient by $T_p$, and composing with the respective inclusions.
The goal of Step 1 is to show that $h|_{\mathring{U}_p}/T_p$ extends to a $T/T_p$-bundle isomorphism
\[
\widetilde{h_p}: \tau^{-1}(U_p)/T_p \to \pi^{-1}(U_p)/T_p.
\]
Since $U_p$ is contractible, the bundles
\[
P_{\tau}|_{U_p}/T_p,
\quad
P_{\pi}|_{U_p}/T_p,
\quad
\tau^{-1}(U_p)/T_p,
\quad
\pi^{-1}(U_p)/T_p
\]
are $T/T_p$-trivial bundles.
Considering $P_{\tau}|_{U_p}/T_p$ and $\tau^{-1}(U_p)/T_p$, we have:
\[
\begin{tikzcd}[ampersand replacement=\&]
\tau^{-1}(\mathring{U_p})/T_p     \ar[r, hook, "\subset"]    \ar{d}
\&
\tau^{-1}(U_p)/T_p      \ar{r}{\alpha'_{p}}[']{\cong}            \ar{d}
\&
P_{\tau}|_{U_p}/T_p            \ar[r, phantom, ":="]
\&
 F^{\ast} \circ j^{\ast} E T/T_{p}
\\  
\mathring{U_p}           \ar[r, hook, "\iota"]
\&
U_p                   \rar{F}
\&
BT           \rar{j}
\&
BT/T_{p}
\end{tikzcd}
\]
where $j$ is the natural quotient.
Thus, the following diagram commutes:
\[
\begin{tikzcd}[ampersand replacement=\&]
\tau^{-1}(\mathring{U_p})/T_p                       \ar[r, hook, "\subset"]      \ar{d}{\alpha'_{p}|_{\mathring{U_p}}}[']{\cong}
\&
\tau^{-1}(U_p)/T_p                                     \ar{d}{\alpha'_{p}}[']{\cong}
\\
P_{\tau}|_{\mathring{U_p}}/T_p                   \ar[r, hook, "\subset"]
\&
P_{\tau}|_{U_p}/T_p
\end{tikzcd}
\]
where $P_{\tau}|_{\mathring{U_p}}/T_p:=\iota^{\ast} \circ P_{\tau}|_{U_p}/T_p =\iota^{\ast} \circ F^{\ast} \circ j^{\ast} E T/T_{p}.$
We now consider the following diagram:
\[
\begin{tikzcd}[ampersand replacement=\&]
\tau^{-1}(\mathring{U_p})/T_p                 \ar[rrr, hook, "\subset"]      \ar{dd}{\alpha'_{p}|_{\mathring{U_p}}}[']{\cong}    \ar{rd}
\&
\&
\&
\tau^{-1}(U_p)/T_p             \ar{dd}{\alpha'_{p}}[']{\cong}      \ar{ld}
\\
\&
\mathring{U_p}          \ar[r, hook, "\subset"]
\&
U_p
\&
\\
P_{\tau}|_{\mathring{U_p}}/T_p           \ar[rrr, hook, "\subset"]      \ar{ru}
\&
\&
\&
P_{\tau}|_{U_p}/T_p              \ar{lu}
\end{tikzcd}
\]
The upper and lower quadrilaterals in this diagram are commutative.
Furthermore, since all of them are pull-back bundles, the left and right triangles also commute by the classification of principal bundles.
Simmilarly, we have the following commutative diagram:
\[
\begin{tikzcd}[ampersand replacement=\&]
\pi^{-1}(\mathring{U_p})/T_p                 \ar[rrr, hook, "\subset"]      \ar{dd}{\beta'_{p}|_{\mathring{U_p}}}[']{\cong}    \ar{rd}
\&
\&
\&
\pi^{-1}(U_p)/T_p             \ar{dd}{\beta'_{p}}[']{\cong}      \ar{ld}
\\
\&
\mathring{U_p}          \ar[r, hook, "\subset"]
\&
U_p
\&
\\
P_{\pi}|_{\mathring{U_p}}/T_p           \ar[rrr, hook, "\subset"]      \ar{ru}
\&
\&
\&
P_{\pi}|_{U_p}/T_p              \ar{lu}
\end{tikzcd}
\]
Therefore, we obtain the following diagram (the definition of $\widetilde{h}_p$ will be given immediately below):
\[
\begin{tikzcd}[ampersand replacement=\&]
P_{\tau}|_{\mathring{U_p}}/T_p
\ar[rrr, hook, "\subset"]      \ar{ddd}{H|_{\mathring{U_p}}/T_p}[']{\cong}    
\&
\&
\&
P_{\tau}|_{U_p}/T_p             \ar{ddd}{H|_{U_p}/T_p}[']{\cong}      
\\  
\&
\tau^{-1}(\mathring{U_p})/T_p          \ar[r, hook, "\subset"]   \dar{\cong}[']{h|_{\mathring{U_p}}/T_p}
        \ar{lu}{{\alpha'_{p}|_{\mathring{U_p}}}}[']{\cong}   
\&
\tau^{-1}(U_p)/T_p     \ar{ru}{\alpha'_{p}}[']{\cong}   
   \dar[red]{\cong}[']{\widetilde{h_p}}
\&
\\   
\&
\pi^{-1}(\mathring{U_p})/T_p          \ar[r, hook, "\subset"]
               \ar{ld}{\cong}[']{{\beta'_{p}|_{\mathring{U_p}}}}
\&
\pi^{-1}(U_p)/T_p         \ar{rd}{\beta'_{p}}[']{\cong}
\&
\\   
P_{\pi}|_{\mathring{U_p}}/T_p           \ar[rrr, hook, "\subset"]      
\&
\&
\&
P_{\pi}|_{U_p}/T_p              
\end{tikzcd}
\]
In this diagram, the commutativity of the left quadrilaterals follows from the definition of $h|_{\mathring{U_p}}/T_p$,
and the commutativity of the upper and lower quadrilaterals follows from the two preceding commutative diagrams.
We then define a bundle isomorphism $\widetilde{h_p}: \tau^{-1}(U_p)/T_p \to \pi^{-1}(U_p)/T_p$ by
\[
\widetilde{h_p}:=(\beta'_{p})^{-1} \circ H|_{U_p}/T_p \circ \alpha'_{p}.
\]
With this definition, all quadrilaterals of the diagram commute.
By the commutativity of the diagram, $\widetilde{h_p}$ is an extension of $h|_{\mathring{U_p}}/T_p$


    \item
We next construct a continuous map 
\[
\widehat{h_p} : \tau^{-1}(U_p) \to \pi^{-1}(U_p),
\]
which extends the restricted map 
\[
h |_{\tau^{-1}(\mathring{U_p})}: \tau^{-1}(\mathring{U_p}) \to \pi^{-1}(\mathring{U_p})
\]
of $h$.
For each $p' \in U_p \cap Q_{l+n-1}$, since $T_{p'}$ acts on $\tau^{-1}(p')$ and $\pi^{-1}(p')$ trivially, the orbit projections
\[
\phi' : \pi^{-1}(p') \to \pi^{-1}(p')/T_{p'}, \quad 
\psi' : \tau^{-1}(p') \to \tau^{-1}(p')/T_{p'}
\]
are the identity maps.
Therefore, there exists a $T/T_{p'}$-equivariant homeomorphism $f_{p'}$ such that the following diagram commutes:
\begin{align}\label{diagram 5.1}
\begin{tikzcd}[ampersand replacement=\&]
T/T_{p'} \rar[equal] \& \tau^{-1}(p') \arrow{rrr}{f_{p'}}  \dar[equal]{\psi'} \&   \& \&
\pi^{-1}(p') \dar[equal]{\phi'}   \&        \\
 T/T_{p'} \rar[equal]  \&  \tau^{-1}(p') / T_{p'} \ar{rrr}{\widetilde{h_{p'}}|_{\tau^{-1}(p') / T_{p'}}}[']{\cong} \& \& \&  \pi^{-1}(p') / T_{p'}  
\end{tikzcd}
\end{align}
Here, $\widetilde{h_{p'}}|_{\tau^{-1}(p') / T_{p'}}$ denotes the restriction of the $T/T_{p'}$-equivariant homeomorphism $\widetilde{h_{p'}} : \tau^{-1}(U_{p'})/T_{p'} \to \pi^{-1}(U_{p'})/T_{p'}$, given in Step 1, to the fiber $\tau^{-1}(p') / T_{p'}$, where $U_{p'} \subset U_p$ is a small neighborhood of $p'$.
Then, we define $\widehat{h_p}: \tau^{-1}(U_p) \to \pi^{-1}(U_p)$ by
\[
\widehat{h_p}(p', t_{p'}) :=
\begin{cases}
f_{p'}(p', t_{p'}) & \text{if } (p', t_{p'}) \in \bigsqcup_{p' \in U_p \cap Q_{l+n-1}} T/T_{p'}, \\
h(p', t_{p'}) & \text{if } (p', t_{p'}) \in \bigsqcup_{p' \in \mathring{U_p}} T/T_{p'}.
\end{cases}
\]
This is an extended map of 
\[
h|_{\tau^{-1}(\mathring{U_p})}: \tau^{-1}(\mathring{U_p}) \to \pi^{-1}(\mathring{U_p}).
\]
We next show that the map $\widehat{h_{p}}$ is continuous.
\begin{itemize}
\item
Since $\widehat{h_{p}}:\tau^{-1}(U_{p})\to \pi^{-1}(U_{p})$ is the extension of 
\[
h|_{\tau^{-1}(\mathring{U_p})}: \tau^{-1}(\mathring{U_p}) \to \pi^{-1}(\mathring{U_p}),
\]
on $\mathring{U_p}$,
\[
\widehat{h_{p}}|_{\tau^{-1}(\mathring{U_p})}:\tau^{-1}(\mathring{U_p}) \to \pi^{-1}(\mathring{U_p})
\]
descends to
\[
\widetilde{h_{p}}|_{\tau^{-1}(\mathring{U_p})/T_p}=h/T_{p}:\tau^{-1}(\mathring{U_p})/T_{p}\to \pi^{-1}(\mathring{U_p})/T_{p}
\]
on $\mathring{U_p}$.

\item
On the other hand,
on $U_{p}\cap Q_{l+n-1}$,
by the commutativity of the diagram (\ref{diagram 5.1}) as above,
\[
\widehat{h_{p}}|_{\tau^{-1}(U_{p}\cap Q_{l+n-1})} : \tau^{-1}(U_{p}\cap Q_{l+n-1}) \to \pi^{-1}(U_{p}\cap Q_{l+n-1})
\]
also descends to 
\[
\widetilde{h_{p}}|_{\tau^{-1}(U_{p}\cap Q_{l+n-1})/T_{p}}:\tau^{-1}(U_{p}\cap Q_{l+n-1})/T_{p}\to \pi^{-1}(U_{p}\cap Q_{l+n-1})/T_{p}
\]
on $U_{p}\cap Q_{l+n-1}$.
\end{itemize}
Therefore, we have the following commutative diagram: 
\begin{align}
\begin{tikzcd}[ampersand replacement=\&]
\tau^{-1}(U_p) \arrow{r}{\widehat{h_p}} \arrow{d}{\psi} \&  \pi^{-1}(U_p) \arrow{d}{\phi}     \\
 \tau^{-1}(U_p) / T_p \ar{r}{\widetilde{h_p}}[']{\cong} \&  \pi^{-1}(U_p) / T_p  
\end{tikzcd}
\label{diagram 5.2}
\end{align}
where $\psi$ and $\phi$ are the orbit projections. 
Let $A \subset \pi^{-1}(U_p)$ be an open set.
By the commutative diagram, we have 
\[
\widehat{h_p}^{-1}(A) = \psi^{-1} \left( \widetilde{h_p}^{-1} \left( \phi(A) \right) \right).
\]
Since $\phi$ is an open map, and $\widetilde{h_p}$ and $\psi$ are continuous, it follows that $\widehat{h_p}^{-1}(A)$ is an open set. Therefore, $\widehat{h_p}$ is a continuous map.

\item
We now construct the unique extension
\[
\widehat{h}: Y(Q, \lambda, c) \to X
\]
of $h:\tau^{-1}(\mathring{Q}) \to \pi^{-1}(\mathring{Q})$ by gluing together the family of maps $\{ \widehat{h_p} \}_{p \in Q}$, where each $\widehat{h_{p}}:\tau^{-1}(U_{p})\to \pi^{-1}(U_{p})$ is a continuous map constructed in Step~2.
For each $p,\,q \in Q$, take a contractible, small open neighborhood $U_{p},\,U_{q}$, respectively.
Then, by Step 2, we have the equivariant homeomorphisms
\[
\widehat{h_p}: \tau^{-1}(U_p) \to \pi^{-1}(U_p)
\]
and
\[
\widehat{h_q}: \tau^{-1}(U_q) \to \pi^{-1}(U_q).
\]
Suppose
\[
U_{p} \cap U_{q} \neq \emptyset.
\]
Then we claim that 
\[
\widehat{h_{p}}|_{\tau^{-1}(U_p \cap U_q)}=\widehat{h_{q}}|_{\tau^{-1}(U_p \cap U_q)}.
\]
On the subset $\tau^{-1}(U_p \cap U_q \setminus Q_{l+n-1}) \subset \tau^{-1}(U_p \cap U_q)$, since $\widehat{h_p}$ and $\widehat{h_q}$ are both extensions of the restriction $h|_{\tau^{-1}(U_p \cap U_q \setminus Q_{l+n-1})}$ of $h$, we have
\[
\widehat{h_{p}}|_{\tau^{-1}(U_p \cap U_q \setminus Q_{l+n-1})} =h|_{\tau^{-1}(U_p \cap U_q \setminus Q_{l+n-1})}=\widehat{h_{q}}|_{\tau^{-1}(U_p \cap U_q \setminus Q_{l+n-1})}.
\]
Since $\tau$ is an open map and $U_{p} \cap U_{q} \setminus Q_{l+n-1}$ is dense in $U_{p} \cap U_{q}$, it follows that
\[
\tau^{-1}(U_{p} \cap U_{q} \setminus Q_{l+n-1}) \subset \tau^{-1}(U_{p} \cap U_{q})
\]
is a dense subset (see \cite{LC25}). 
By Lemma \ref{dense coincide}, we have
\[
\widehat{h_{p}}|_{\tau^{-1}(U_p \cap U_q)}=\widehat{h_{q}}|_{\tau^{-1}(U_p \cap U_q)}.
\]
Therefore, we can glue the continuous maps $\widehat{h_p}$ and $\widehat{h_q}$ along $U_p \cap U_q$. Thus we can define a continuous map 
\[
\widehat{h_p} \cup \widehat{h_q} =: \widehat{h_{pq}} : \tau^{-1}(U_p \cup U_q) \to \pi^{-1}(U_p \cup U_q),
\]
where $\widehat{h_p} \cup \widehat{h_q}$ is the map obtained by gluing $\widehat{h_p}$ and $\widehat{h_q}$ together.
Let $\{ U_p \}_{p \in Q}$ be an open cover of $Q$ by small neighborhoods around each point.
By gluing the maps $\{ \widehat{h_p}\}_{p \in Q}$, we obtain a continuous map
\[
\bigcup_{p \in Q} \widehat{h_p}=:\widehat{h} : Y(Q, \lambda, c)=\tau^{-1}(Q) \to \pi^{-1}(Q)=X.
\]
\end{enumerate}

We next show that $\widehat{h}$ is a bijection and $T$-equivariant.
To prove this, it suffices to show that for every $p \in Q$, the restriction of $\widehat{h}$ to the fiber $\tau^{-1}(p)$, namely $\widehat{h}|_{\tau^{-1}(p)}: \tau^{-1}(p) \to \pi^{-1}(p)$, is a bijection and $T$-equivariant.
We consider the following commutative diagram obtained by restricting diagram (\ref{diagram 5.2}) to the fiber over $p$:
\begin{align}\label{diagram fiber over p}
\begin{tikzcd}[ampersand replacement=\&]
\tau^{-1}(p) \arrow{rrr}{\widehat{h}|_{\tau^{-1}(p)}=\widehat{h_p}|_{\tau^{-1}(p)} } \arrow{d}{\psi |_{\tau^{-1}(p)}} \&   \& \&
\pi^{-1}(p) \arrow{d}{\phi |_{\pi^{-1}(p)}}     \\
 \tau^{-1}(p) / T_p \ar{rrr}{\widetilde{h_p}|_{\tau^{-1}(p) / T_p}}[']{\cong} \& \& \&  \pi^{-1}(p) / T_p  
\end{tikzcd}
\end{align}
Since $T_p$ acts on $\tau^{-1}(p)$ and $\pi^{-1}(p)$ trivially, $\psi |_{\tau^{-1}(p)}$ and $\phi |_{\pi^{-1}(p)}$ are the identities.
Therefore, $\widehat{h}|_{\tau^{-1}(p)}$ is also a bijection.
Moreover, since the lower horizontal map $\widetilde{h_p}|_{\tau^{-1}(p) / T_p}$ is a $T/T_p$-bundle isomorphism, $\widehat{h}|_{\tau^{-1}(p)}$ is also a $T/T_p$-equivariant map.
To show that $\widehat{h}|_{\tau^{-1}(p)}$ is a $T$-equivariant map, we consider the following diagram:
\[
\begin{tikzcd}[ampersand replacement=\&]
\tau^{-1}(p) \ar{rrrrr}{\widehat{h}|_{\tau^{-1}(p)}} \ar{ddd}{t} \ar{rd}{\psi |_{\tau^{-1}(p)}}   \& 
   \ar[phantom,"\textcircled{2}"]{drrr}   \&  \&
\&
\&
      \pi^{-1}(p) \ar{ddd}{t}  \ar{ld}[']{\phi |_{\pi^{-1}(p)}}  \& \\
\ar[phantom,"\textcircled{1}"]{rd} \&
  \tau^{-1}(p)/T_p \ar{rrr}{\widetilde{h_p}|_{\tau^{-1}(p) / T_p}}[']{\cong} \ar{d}{t} \ar[phantom,"\textcircled{5}"]{rrrd}\&
\&
\&
    \pi^{-1}(p)/T_p \ar{d}{t} \&
     \ar[phantom,"\textcircled{3}"]{ld}  \& \\
 \&
  \tau^{-1}(p)/T_p \ar{rrr}{\widetilde{h_p}|_{\tau^{-1}(p) / T_p}}[']{\cong} \&
\&
\&
    \pi^{-1}(p)/T_p \&
       \& \\
\tau^{-1}(p) \ar{rrrrr}{\widehat{h}|_{\tau^{-1}(p)}} \ar{ru}[']{\psi |_{\tau^{-1}(p)}} \&
 \&
     \&
\&
\ar[phantom,"\textcircled{4}"]{ulll} \&
      \pi^{-1}(p) \ar{lu}{\phi |_{\pi^{-1}(p)}} \& \\
\end{tikzcd}
\]
Here, the map $t$ denotes the action of $t \in T$.
In this diagram, regions \textcircled{1} and \textcircled{3} are commutative since $\psi$ and $\phi$ are orbit projections.
Regions \textcircled{2} and \textcircled{4} are commutative by diagram (\ref{diagram fiber over p}). Region \textcircled{5} is commutative since $\widetilde{h_p}|_{\tau^{-1}(p) / T_p}$ is a $T/T_p$-equivariant homeomorphism.
To show that $\widehat{h}|_{\tau^{-1}(p)}$ is $T$-equivariant, it suffices to verify that the outer square in this diagram commutes; that is,
\[
\widehat{h}|_{\tau^{-1}(p)} \circ t =t \circ \widehat{h}|_{\tau^{-1}(p)}.
\]
This follows from the fact that regions \textcircled{1} through \textcircled{5} commute, and that
$\psi |_{\tau^{-1}(p)}$ and $\phi |_{\pi^{-1}(p)}$ are the identity maps (in particular, we can take the inverse map).

Therefore, the continuous map $\widehat{h}$ is an equivariant bijection. Thus, since $Y(Q, \lambda, c)$ is compact and $X$ is Hausdorff, it follows that $\widehat{h}:Y(Q, \lambda, c) \to X$ is a $T$-equivariant homeomorphism.
\end{proof}

Combining Lemma \ref{lem a} and Lemma \ref{lem b}, we obtain Lemma \ref{main thm}.

\begin{proof}[Proof of Lemma \ref{main thm}]
By Lemma \ref{lem a} and Lemma \ref{lem b}, we have that the equivariant homeomorphisms $X(Q, \lambda, c) \cong Y(Q, \lambda, c) \cong X$.
This shows Lemma \ref{main thm}.
\end{proof}

By combining Lemma \ref{main thm} with Proposition \ref{inv of T-pseudomanifold}, we obtain the following corollary to Lemma~\ref{main thm}.

\begin{cor}\label{canonical model-1}
Let $(Q, \lambda, c)$ be the characteristic data of a locally standard $T$-pseudomanifold $X$. Assume that the top strata $Q_{l+n}\setminus Q_{l+n-1}$ is homotopy equivalent to $Q=X/T$ itself. Then the canonical model $X(Q, \lambda, c)$ is also a locally standard $T$-pseudomanifold.
\end{cor}
\begin{proof}
By Lemma~\ref{main thm}, we have an equivariant homeomorphism
\[
X \cong X(Q, \lambda, c).
\]
Since $X$ is a locally standard $T$-pseudomanifold, Proposition~\ref{inv of T-pseudomanifold} implies that $X(Q, \lambda, c)$ is also a locally standard $T$-pseudomanifold.
\end{proof}

\subsection{Proof of the classification theorem}
We now prove the classification theorem.

\begin{proof}[Proof of Theorem \ref{classification theorem}]
Assume that the characteristic datum $(Q,\lambda,c)$ of $X$ and $(Q',\lambda',c')$ of $X'$ are (weakly) isomorphic.
By Theorem \ref{thm of equivariant homeo of canonical model},
the canonical models are (weakly) $T$-equivariantly homeomorphic:
\begin{align}\label{eq homeo X(Q, lambda)}
X(Q, \lambda, c) \cong X(Q', \lambda', c').
\end{align}
Furthermore, by Lemma \ref{main thm}, we obtain $T$-equivariant homeomorphisms
\begin{align}\label{eq homeo T-space}
X(Q, \lambda, c) \cong X, \quad
X(Q', \lambda', c') \cong X'
\end{align}
Combining \eqref{eq homeo X(Q, lambda)} and \eqref{eq homeo T-space}, we conclude that $X$ and $X'$ are (weakly) $T$-equivariantly homeomorphic.

The converse follows from Proposition \ref{prop in 4}.
\end{proof}

\begin{rem}
Note that Proposition \ref{prop in 4} asserts that the characteristic data is invariant under (weakly) $T$-equivariant homeomorphisms of locally standard $T$-pseudomanifolds.
From this proposition and Theorem \ref{classification theorem}, we see that locally standard $T$-pseudomanifolds whose orbit spaces have the set of top strata that is homotopy equivalent to the entire space are completely classified, up to (weakly) $T$-equivariant homeomorphism, by their characteristic datum.
\end{rem}

\section{The canonical model $X(Q, \lambda, c)$ is a locally standard $T$-pseudomanifold}\label{section 12}

Let $(Q, \lambda, c)$ be a characteristic data.  
Recall from Remark~\ref{not orbit space} that $Q$ is not necessarily the orbit space of a locally standard $T$-pseudomanifold.
In this section, we show that $(Q, \lambda, c)$ gives rise to a locally standard $T$-pseudomanifold.
Namely, we prove the following theorem.
\begin{thm}\label{canonical model is a T-pseudomanifold}
Let $Q$ be a topological stratified pseudomanifold, $\lambda$ be its characteristic functor, and $c\in H^{2}(Q;\mathbb{Z}^{m})$. Then, the canonical model $X(Q,\lambda,c)$ is a locally standard $T$-pseudomanifold.
\end{thm}
\begin{rem}
In this section, we do not assume that the top strata $Q \setminus Q_{l+n-1}$ is homotopy equivalent to $Q$.
In the previous sections, for a locally standard $T$-pseudomanifold $X$, we define the Chern class $c \in H^2(X/T; \mathbb{Z}^m)$ as the Chern class of the free orbits.
The assumption of homotopy equivalence was required for this purpose (see Section~\ref{section 4} for details).
However, once we simply specify a cohomology class $c \in H^2(Q;\mathbb{Z}^m)$, we can construct a space from the principal bundle $P_c \to Q$ whose Chern class is $c$.
Moreover, this homotopy equivalence assumption is not necessary for defining the characteristic functor.
Therefore, in exactly the same manner as in Definition \ref{def of canonical model}, we can define a space equipped with a torus action.
\end{rem}

\begin{rem}\label{not assume homotopy} 
Since the definition of a weak isomorphism between characteristic datum (see Definition~\ref{def of isomorphism of char pairs})
is independent of the homotopy equivalence assumption,
the same definition applies in the absence of this assumption.
 Furthermore, the proof of Theorem~\ref{thm of equivariant homeo of canonical model} does not depend on the homotopy equivalence assumption. Therefore, Theorem~\ref{thm of equivariant homeo of canonical model} holds even without this assumption.
\end{rem}

In subsection \ref{prepare-3}, we will prove this theorem, i.e., $X(Q, \lambda, c)$ satisfies the conditions in Definition \ref{def of T-pseudomanifold}. In subsections \ref{prepare-1}, \ref{prepare-1-2}, \ref{prepare-2} and \ref{subsubsection of proof of auxiliary lemma}, we prepare to prove it.

\subsection{$(X(Q, \lambda, c), \mathfrak{X})$ is a manifold stratified space}\label{prepare-1}

Let us briefly review Section \ref{section canonical model}.
Given a characteristic data $(Q, \lambda, c)$ (note that we do not assume that $Q\setminus Q_{l+n-1}$ is homotopy equivalent to $Q$), the canonical model $X(Q, \lambda, c)$ is defined as the quotient space $P_c/{\sim}$, where $\xi: P_c \to Q$ is the $T^m$-principal bundle whose Chern class is $c$ (see Definition \ref{def of canonical model}).
Moreover, the canonical model $X(Q, \lambda, c)$ admits a filtration
\begin{align} \tag{8.1}
\mathfrak{X} : X(Q, \lambda, c) 
 \supset 
X(Q_{l+n-1}, \lambda_{l+n-1}, c) 
\supset \cdots \supset
X(Q_{l+i}, \lambda_{l+i}, c) 
\supset \cdots \supset
X(Q_l, \lambda_l, c) 
\supsetneq \emptyset ,
\end{align}
where $X(Q_{l+i}, \lambda_{l+i}, c) = (P_c)|_{Q_{l+i}}/{\sim}$.
Here, $(P_c)|_{Q_{l+i}}$ denotes the restriction of the bundle $P_c$ to the subspace $Q_{l+i} \subset Q$.

We first verify that $(X(Q, \lambda, c), \mathfrak{X})$ is a manifold stratified space (see Definition~\ref{def of manifold stratified space}).  
It is easy to see that $(X(Q, \lambda, c), \mathfrak{X})$ is a stratified space.  
Therefore, it suffices to prove the following lemma.

\begin{lem}\label{condition of pseudomanifold-1}
Each connected component of 
\[
X(Q_{l+i}, \lambda_{l+i}, c) \setminus X(Q_{l+i-1}, \lambda_{l+i-1}, c) = (P_c)|_{Q_{l+i} \setminus Q_{l+i-1}} / {\sim}
\]
is an $(l+2i+(m-n))$-dimensional manifold for $0 \leq i \leq n$.
\end{lem}

\begin{proof}
By the definition of the equivalence relation (see Definition \ref{def of canonical model}), for an $(l+i)$-dimensional stratum $S \subset Q$,
\begin{align*}
(P_c)|_{S}/ {\sim} 
&= (P_c)|_{S}/\lambda(S) .
\end{align*}
Note that $(P_c)|_{S}/\lambda(S)$ is a $T^m/\lambda(S)(\cong T^{m-(n-i)})$-principal bundle.
Since $S$ is an $(l+i)$-dimensional manifold, it follows that $(P_c)|_{S}/ {\sim} $ is an $(l+2i+(m-n))$-dimensional manifold. Since each connected component of $(P_c)|_{Q_{l+i} \setminus Q_{l+i-1}}/{\sim}$ is of the form $(P_c)|_{S}/ {\sim} $, we establish the statement.
\end{proof}

\subsection{$X(Q, \lambda, c)$ satisfies the condition Definition \ref{def of T-pseudomanifold}-1}\label{prepare-1-2}

In this subsection, we prove the following lemma.
\begin{lem}\label{condition of pseudomanifold-2}
The space
\[
X(Q, \lambda, c) \setminus X(Q_{l+n-1}, \lambda_{l+n-1}, c) = (P_c)|_{Q_{l+n} \setminus Q_{l+n-1}} / {\sim}
\]
is dense in $X(Q, \lambda, c)$.
\end{lem}
\begin{proof}
Since $Q$ is a topological stratified pseudomanifold, $Q_{l+n} \setminus Q_{l+n-1}$ is dense in $Q$ (see Definition~\ref{def of pseudomanifold}-2). Consider the quotient map
\[
P_c \to P_c /{\sim}=X(Q, \lambda, c).
\]
This map is continuous and surjective. Since $(P_c)|_{Q_{l+n} \setminus Q_{l+n-1}}$ is dense in $P_c$, its image
\[
(P_c)|_{Q_{l+n} \setminus Q_{l+n-1}} / {\sim}
\]
is also dense in $X(Q, \lambda, c)$.
\end{proof}

\subsection{The link $L_x$ of $x$}\label{prepare-2}

We next prove that $X(Q, \lambda, c)$ satisfies the Definition \ref{def of T-pseudomanifold}-2 (c). This will be proved in Lemma~\ref{universality of quotient}.
To state Lemma~\ref{universality of quotient} precisely, we need to prepare a link $L_x$ in Construction \ref{cons of link}.
Assume that $l+n\neq 0$, as in Definition \ref{def of T-pseudomanifold}-2 (c).
Let $\pi: P_{c}/{\sim}\to Q$ be the orbit projection.
If we take $x\in (P_{c})|_{Q_{l+i}\setminus Q_{l+i-1}}/{\sim}$ for $0\le i\le n$, then there exists a stratum $S\subset Q_{l+i}\setminus Q_{l+i-1}$ such that $\pi(x)\in S$.

Since $Q$ is a topological stratified pseudomanifold, we can choose a triple $(U_{\pi(x)}, L_{\pi(x)}, \varphi_{\pi(x)})$ (see Definition \ref{def of pseudomanifold}-3), where $U_{\pi(x)}$ is a contractible small open neighborhood.

In order to construct a link $L_x$ of $x$, we use the following two lemmas.

\begin{lem}\label{correspondence strata}
For each $0 \leq j \leq n-i-1$, the homeomorphism in Definition \ref{def of pseudomanifold}-3 (c)
\begin{align*}
\varphi_{\pi(x)}|_{U_{\pi(x)} \cap Q_{l+i+j+1}} : U_{\pi(x)} \cap Q_{l+i+j+1} & \to
O_{\pi(x)} \times \cone((L_{\pi(x)})_j)
\end{align*}
induces a homeomorphism
\[
\varphi'_{j} : U_{\pi(x)} \cap (Q_{l+i+j+1} \setminus Q_{l+i+j}) \to O_{\pi(x)} \times  \left( (L_{\pi(x)})_j \setminus (L_{\pi(x)})_{j-1} \right) \times (0,1)
\]
such that
for any stratum ${R} \subset (L_{\pi(x)})_j \setminus (L_{\pi(x)})_{j-1}$, there exists a stratum $S' \subset Q_{l+i+j+1} \setminus Q_{l+i+j}$ satisfying
\begin{align}\label{unique face}
\varphi'_{j}(U_{\pi(x)} \cap {S'}) = O_{\pi(x)} \times {R} \times (0,1).
\end{align}
\end{lem}

\begin{proof}
For $0 \leq j \leq n-i-1$,
restricting the map
$
\varphi_{\pi(x)}|_{U_{\pi(x)} \cap Q_{l+i+j+1}} 
$
to the subspace $U_{\pi(x)} \cap (Q_{l+i+j+1} \setminus Q_{l+i+j})$, we obtain the homeomorphism
\begin{align*}
\varphi_{\pi(x)}|_{U_{\pi(x)} \cap (Q_{l+i+j+1} \setminus Q_{l+i+j})}:
U_{\pi(x)} \cap (Q_{l+i+j+1} \setminus Q_{l+i+j}) \to
O_{\pi(x)} \times \cone((L_{\pi(x)})_j) \setminus \cone((L_{\pi(x)})_{j-1}).
\end{align*}
From the definition of the open cone (see Definition \ref{def of open cone}), we also have a homeomorphism
\begin{align}\label{homeo by removing cone vertex}
\psi : O_{\pi(x)} \times \cone(L_{\pi(x)}) \setminus \{ \text{cone vertex} \} \to O_{\pi(x)} \times L_{\pi(x)} \times (0,1).
\end{align}
Since $(L_{\pi(x)})_j \supset (L_{\pi(x)})_{j-1}$ and their open cones have the common cone vertex, the space $\cone((L_{\pi(x)})_j) \setminus  \cone((L_{\pi(x)})_{j-1})$ is removing the cone vertex from $\cone \left( (L_{\pi(x)})_j \setminus (L_{\pi(x)})_{j-1} \right)$.
Therefore, restricting $\psi$ yields the homeomorphism
\begin{align*}
\psi|_{O_{\pi(x)} \times \cone((L_{\pi(x)})_j) 
\setminus \cone((L_{\pi(x)})_{j-1})} :\quad
& O_{\pi(x)} \times \cone((L_{\pi(x)})_j) 
  \setminus \cone((L_{\pi(x)})_{j-1}) \\
& \to O_{\pi(x)} \times \left( (L_{\pi(x)})_j 
  \setminus (L_{\pi(x)})_{j-1} \right) \times (0,1).
\end{align*}
By composing these homeomorphisms, we obtain the homeomorphism
\begin{align}\label{removing the cone vertex}
\varphi'_{j}:
U_{\pi(x)} \cap (Q_{l+i+j+1} \setminus Q_{l+i+j})
\xrightarrow{\ \varphi|_{U_{\pi(x)} \cap (Q_{l+i+j+1} \setminus Q_{l+i+j})} \ } 
O_{\pi(x)} \times \cone((L_{\pi(x)})_j) \setminus \cone((L_{\pi(x)})_{j-1}) & \notag \\
\xrightarrow{\ \psi|_{O_{\pi(x)} \times \cone((L_{\pi(x)})_j) \setminus \cone((L_{\pi(x)})_{j-1})} \ } 
O_{\pi(x)} \times \left( (L_{\pi(x)})_j \setminus (L_{\pi(x)})_{j-1} \right)  \times (0,1).&
\end{align}
Since the homeomorphism preserves connected components,
the homeomorphism
(\ref{removing the cone vertex}) shows that
for any stratum ${R} \subset (L_{\pi(x)})_j \setminus (L_{\pi(x)})_{j-1}$,
there exists a unique stratum $S' \subset Q_{l+i+j+1} \setminus Q_{l+i+j}$ such that
\[
\varphi'_{j}(U_{\pi(x)} \cap {S'}) = O_{\pi(x)} \times {R} \times (0,1).
\]
\end{proof}

The following figures illustrate the isomorphism (\ref{unique face}) via the local homeomorphism near a point in the case where $O_{\pi(x)}$ is a single point.
The left figure below shows the neighborhood $U_{\pi(x)}$ in $Q$, while the right figure below shows the corresponding cone structure.
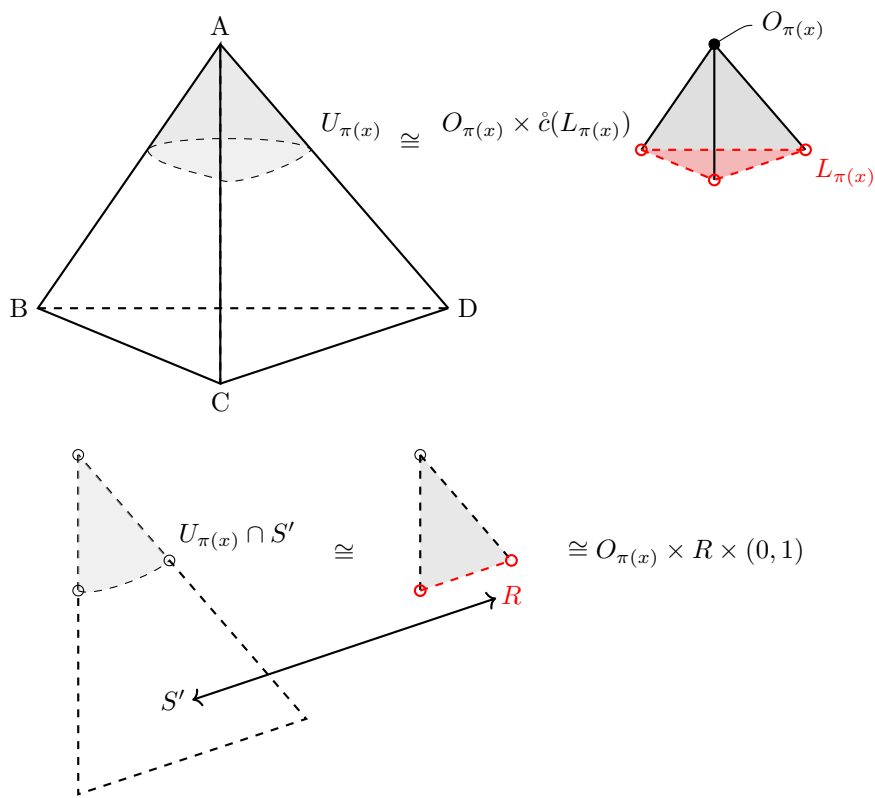
\[
\begin{tikzpicture}

  \coordinate (A) at (0, 4.5);
  \coordinate (B) at (-2.4,1);
  \coordinate (C) at (0,0);
  \coordinate (D) at (3,1);
  \coordinate (AB) at (-0.96, 3.1);
  \coordinate (AC) at (0, 2.7);
  \coordinate (AD) at (1.2, 3.1);

  \fill[gray!20, opacity=0.6] (A) -- (AB)
    .. controls (-0.9, 2.9) and (-0.3, 2.8) .. (AC)
    -- cycle;
  \fill[gray!20, opacity=0.6] (A) -- (AC)
    .. controls (0.2, 2.6) and (1.0, 2.9) .. (AD)
    -- cycle;
  \fill[gray!20, opacity=0.6] (A) -- (AD)
    .. controls (1.0, 3.3) and (-0.8, 3.3) .. (AB)
    -- cycle;

  \draw[dashed]
    (AB) .. controls (-0.9, 2.9) and (-0.3, 2.8) .. (AC)
         .. controls (0.2, 2.6) and (1.0, 2.9) .. (AD)
         .. controls (1.0, 3.3) and (-0.8, 3.3) .. (AB);

  \draw[thick] (A) -- (B) -- (C) -- cycle;
  \draw[dashed, thick] (A) -- (C);
  \draw[thick] (A) -- (D);
  \draw[dashed, thick] (B) -- (D);
  \draw[thick] (C) -- (D);

  \node at (A) [above] {A};
  \node at (B) [left] {B};
  \node at (C) [below] {C};
  \node at (D) [right] {D};
\node at (AD) [above right] {$U_{\pi(x)}$};

  \node at (2.5,3.2) {$\cong$}; 

\begin{scope}[shift={(6.5,0)}]
  \coordinate (A2) at (0, 4.5);
  \coordinate (AB2) at (-0.96, 3.1);
  \coordinate (AC2) at (0, 2.7);
  \coordinate (AD2) at (1.2, 3.1);

  \fill[gray!30, opacity=0.6] (A2) -- (AB2) -- (AC2) -- cycle;
  \fill[gray!30, opacity=0.6] (A2) -- (AC2) -- (AD2) -- cycle;
  \fill[gray!30, opacity=0.6] (A2) -- (AD2) -- (AB2) -- cycle;

  \fill[red!40, opacity=0.6] (AB2) -- (AC2) -- (AD2) -- cycle;

  \draw[thick] (A2) -- (AB2);
  \draw[thick] (A2) -- (AC2);
  \draw[thick] (A2) -- (AD2);

  \draw[red, thick, dashed] (AB2) -- (AC2) -- (AD2) -- cycle;

  \filldraw[black] (A2) circle (2pt);
  \draw[red, thick] (AB2) circle (2pt);
  \draw[red, thick] (AC2) circle (2pt);
  \draw[red, thick] (AD2) circle (2pt);

\node[red] at (AD2) [below right] {$L_{\pi(x)}$};
\draw (A2) .. controls (0.4,4.75) .. ++(0.5,0.25) node[right] {$O_{\pi(x)}$};
\node at (AB2) [above left] {$O_{\pi(x)} \times \mathring{c}(L_{\pi(x)})$};
\end{scope}
\end{tikzpicture}
\]
\[
\begin{tikzpicture}

  \begin{scope}
    \coordinate (A) at (0, 4.5);
    \coordinate (C) at (0, 0);
    \coordinate (D) at (3, 1);

    \coordinate (AC) at (0, 2.7);
    \coordinate (AD) at (1.2, 3.1);

    \draw[thick, dashed] (A) -- (C) -- (D) -- cycle;

    \fill[gray!20, opacity=0.6]
      (A) -- (AC)
      .. controls (0.2, 2.6) and (1.0, 2.9) .. (AD)
      -- cycle;

    \draw[dashed]
      (AC) .. controls (0.2, 2.6) and (1.0, 2.9) .. (AD);

    \draw[black] (AC) circle (2pt);
    \draw[black] (AD) circle (2pt);
    \draw[black] (A) circle (2pt);
\node at (AD) [above right] {$U_{\pi(x)} \cap S'$};
\node at (1.25,1.25) {$S'$};
  \end{scope}

  \node at (3.5,3.2) {$\cong$};
  \node at (8,3.2)  {$\cong O_{\pi(x)} \times {R} \times (0,1)$};

  \begin{scope}[shift={(4.5,0)}]
    \coordinate (A) at (0, 4.5);
    \coordinate (AC) at (0, 2.7);
    \coordinate (AD) at (1.2, 3.1);

    \fill[gray!30, opacity=0.6]
      (A) -- (AC) -- (AD) -- cycle;

    \draw[thick, dashed] (A) -- (AC);
    \draw[thick, dashed] (A) -- (AD);
    \draw[red, thick, dashed] (AC) -- (AD);

    \draw[black] (A) circle (2pt);
    \draw[red, thick] (AC) circle (2pt);
    \draw[red, thick] (AD) circle (2pt);

\node[red] at (AD) [below=6pt] {$R$};
  \end{scope}
\draw[<->, thick]
  (1.5, 1.25) -- ++(4, 1.35);
\end{tikzpicture}
\]

Lemma \ref{correspondence strata} establishes a correspondence between strata, which can be interpreted as a functor between poset categories.

\begin{lem}\label{lem: functor nu}
We denote by $\nu : R \mapsto S'$ the correspondence between ${R} \subset (L_{\pi(x)})_j \setminus (L_{\pi(x)})_{j-1}$ and $S' \subset Q_{l+i+j+1} \setminus Q_{l+i+j}$ determined by \eqref{unique face}.
This correspondence defines a functor
\[
\nu : \mathcal{S}(L_{\pi(x)}) \to \mathcal{S}(Q).
\]
\end{lem}

\begin{proof}
Let $R_1, R_2 \in \mathcal{S}(L_{\pi(x)})$, and
$S_{1}^{'} = \nu (R_1), S_{2}^{'} =\nu (R_2) \in \mathcal{S}(Q)$.
To prove that $\nu$ defines a functor, it suffices to verify that if
$R_1 \subset \overline{R_2}$, then $S_{1}^{'} \subset \overline{S_{2}^{'}}$.
Consider the homeomorphism in Definition \ref{def of pseudomanifold}-3 (c):
\[
\varphi_{\pi(x)}^{-1} : O_{\pi(x)} \times \cone(L_{\pi(x)})
\to
U_{\pi(x)}.
\]
By restricting the homeomorphism, we obtain the continuous map
\[
\varphi_{\pi(x)}^{-1}|_{O_{\pi(x)} \times \cone(L_{\pi(x)}) \setminus \{ \text{cone vertex} \}} :
O_{\pi(x)} \times \cone(L_{\pi(x)}) \setminus \{ \text{cone vertex} \}
\to
U_{\pi(x)}.
\]
We define the composition as follows:
\begin{align*}
\Phi: 
O_{\pi(x)} \times \cone(L_{\pi(x)}) \times (0,1)
&\xrightarrow{\ \psi^{-1} \ }
O_{\pi(x)} \times \cone(L_{\pi(x)}) \setminus \{ \text{cone vertex} \} \\
&\xrightarrow{\ \varphi_{\pi(x)}^{-1}|_{O_{\pi(x)} \times \cone(L_{\pi(x)}) \setminus \{ \text{cone vertex} \}} \ }
U_{\pi(x)},
\end{align*}
where $\psi$ is the homeomorphism in \eqref{homeo by removing cone vertex}. This composition $\Phi$ is continuous.
Furthermore, for each $j$,
by restricting 
\begin{itemize}
\item
$\psi^{-1}$
to
$(\psi|_{O_{\pi(x)} \times \cone((L_{\pi(x)})_j) \setminus \cone((L_{\pi(x)})_{j-1})})^{-1} = \psi^{-1}|_{O_{\pi(x) \times ( (L_{\pi(x)})_j \setminus (L_{\pi(x)})_{j-1} ) \times (0,1)}}$,

\item
$\varphi_{\pi(x)}^{-1}|_{O_{\pi(x)} \times \cone(L_{\pi(x)}) \setminus \{ \text{cone vertex} \}}$
to
\[
\varphi_{\pi(x)}^{-1}|_{O_{\pi(x)} \times \cone(L_{\pi(x)})_j \setminus \cone(L_{\pi(x)})_{j-1}}
=
(\varphi|_{U_{\pi(x)} \cap (Q_{l+i+j+1} \setminus Q_{l+i+j})})^{-1},
\]
\end{itemize}
we have the homeomorphism $(\varphi'_{j})^{-1}$ as the restriction of $\Phi$.
Therefore, we obtain
\[
\Phi \left( O_{\pi(x)} \times R_1 \times (0,1) \right)
=
U_{\pi(x)} \cap S_{1}^{'}.
\]
By the properties of closures and the continuity of $\Phi$, we obtain:
\begin{align*}
U_{\pi(x)} \cap S_{1}^{'}
&=
\Phi \left( O_{\pi(x)} \times R_1 \times (0,1) \right) \\
&\subset 
\Phi \left( O_{\pi(x)} \times \overline{R_2} \times (0,1) \right) \\
&\subset
\Phi \left( \overline{O_{\pi(x)} \times {R_2} \times (0,1)} \right) \\
&\subset
\overline{\Phi \left( O_{\pi(x)} \times {R_2} \times (0,1) \right)} \\
&=
\overline{U_{\pi(x)} \cap S_{2}^{'}}.
\end{align*}
Since the closure $\overline{U_{\pi(x)} \cap S_{2}^{'}}$ is taken in 
$U_{\pi(x)}$, we have
\[
\overline{U_{\pi(x)} \cap S_{2}^{'}}=U_{\pi(x)} \cap \overline{S_{2}^{'}}.
\]
Therefore,
\[
S_{1}^{'} \cap \overline{S_{2}^{'}} \neq \emptyset.
\]
Since $S_{1}^{'}$ and $S_{2}^{'}$ are strata of the topological stratified pseudomanifold $Q$, and such pseudomanifolds satisfy the frontier condition (see Proposition \ref{frontier condition}), we conclude that
$S_{1}^{'} \subset \overline{S_{2}^{'}}$.
\end{proof}

We construct a link $L_x$ of $x \in (P_c)|_{Q_{l+i} \setminus Q_{l+i-1}}/{\sim}$ as follows.

\begin{cons}[link $L_x$]\label{cons of link}
Recall that $U_{\pi(x)}$ is a contractible small open neighborhood of $\pi(x)$.
Then $S \subset Q_{l+i} \setminus Q_{l+i-1}$ is the minimal stratum intersecting $U_{\pi(x)}$.
That is, for any $(l+i+j+1)$-dimensional stratum $S'$ with $-1 \le j \le n - i - 1$ such that $U_{\pi(x)} \cap S' \neq \emptyset$, we have $\lambda(S') \subset \lambda(S)$. 
Thus, by applying the Lemma~\ref{lem: functor nu}, we can define a functor
$\mu$ by composing $\nu$ and $\lambda$:
\begin{align*}
\begin{tikzcd}[ampersand replacement=\&, row sep=0.25cm]
\mu= \lambda\circ \nu : \!\!\!\!\!\!\!\!\!\!\!\!\!\!\!\!\!\! \&  \mathcal{S}(L_{\pi(x)})^{\mathrm{op}}     \rar
      \& \mathcal{T}_{\lambda(S)}
\\[-10pt]
\&  \rotatebox{90}{$\in$}  \&  \rotatebox{90}{$\in$}    \\[-10pt]
\&  R         \rar[mapsto]
    \& \lambda(S') ,
\end{tikzcd}
\end{align*}
where $S'=\nu(R)$, and $\mathcal{T}_{\lambda(S)}$ denotes the category of torus subgroups of $\lambda(S) \cong T^{n-i}$.
Since $\lambda(S') \cong T^{n-(i+j+1)}=T^{(n-i-1)-j}$ and $\lambda$ is a characteristic functor, it follows that $\mu$ is also a characteristic functor on $L_{\pi(x)}$.
Note that, since the small open neighborhood $U_{\pi(x)}$ is contractible, the principal bundle over it is trivial.
Viewing $L_{\pi(x)}$ as embedded in $U_{\pi(x)}$ (see Definition~\ref{def of pseudomanifold}-3), the principal bundle over $L_{\pi(x)}$ may be regarded as induced from that over $U_{\pi(x)}$.
Therefore, we define a $\lambda(S) \,(=T_x \cong T^{n-i})$-space (see Definition \ref{def of canonical model}) as follows:
\begin{align}\label{def of L_x}
L_x := X(L_{\pi(x)}, \mu, 0) = L_{\pi(x)} \times \lambda(S)/{\sim}_{\mu},
\end{align}
where ${\sim}_{\mu}$ is the equivalence relation on $L_{\pi(x)} \times \lambda(S)$ defined by $\mu$.
\end{cons}

The following lemma is obtained by the classical Lie theory.

\begin{lem}\label{torus split}
Let $Q$ be a second-countable, compact, Hausdorff topological stratified pseudomanifold and suppose that $\lambda$ is a characteristic functor on $Q$.
For any stratum $S \subset Q$, 
there exists a torus $T_S'$  such that $T_S' \cong T^m/\lambda(S)$ and $T^m \cong \lambda(S) \times T_S'$.
\end{lem}
\begin{proof}
We use the basic theory of Lie groups (see {\cite[Chapter 3]{War83}}).
Let $\mathfrak{t} \cong \mathbb{R}^m$ be the Lie algebra of $T^m$. Since $\lambda(S) \subset T^m$ is a torus subgroup, its Lie algebra $\mathfrak{s}$ is a subspace of $\mathfrak{t}$. Then, there exists a vector subspace $\mathfrak{s}'$ isomorphic to $\mathfrak{t}/\mathfrak{s}$ such that $\mathfrak{t} = \mathfrak{s} \oplus \mathfrak{s}'$. The Lie group $T_S'$ corresponding to $\mathfrak{s}'$ is isomorphic to $T^m/\lambda(S)$.
Taking exponential on $\mathfrak{t}$, we have $T^m \cong \lambda(S) \times T_S'$.
\end{proof}

We next introduce an auxiliary lemma that will be used in the proof of Lemma \ref{universality of quotient}.

\begin{lem}\label{auxiliary lemma}
Let $(Q, \lambda, c)$ be a characteristic data, and let $x \in X(Q, \lambda, c)$.
Suppose that $S \subset Q$ is a stratum containing $\pi(x)$, where $\pi : X(Q, \lambda, c) \to Q$ denotes the orbit projection.
Then there exists a $T$-equivariant homeomorphism:
\[
U_{\pi(x)} \times  \lambda(S) / {\sim}_\lambda 
\cong
O_{\pi(x)} \times \cone(L_{\pi(x)} \times \lambda(S)/{\sim}_{\mu}),
\]
where $U_{\pi(x)}$ is a contractible small open neighborhood of $\pi(x)$ and $O_{\pi(x)} \subset \mathbb{R}^l$ is a contractible open subset.
\end{lem}

Note that we do not assume that $Q\setminus Q_{l+n-1}$ is homotopy equivalent to $Q$.
The proof of this lemma will be given in the next subsection \ref{subsubsection of proof of auxiliary lemma}.
Assuming this lemma, we may prove the following lemma.
It shows that $X(Q, \lambda, c)$ satisfies Definition \ref{def of T-pseudomanifold}-2 (c).

\begin{lem}\label{universality of quotient}
Let $(Q, \lambda, c)$ be a characteristic data and $x \in X(Q, \lambda, c)$. Let $\pi : X(Q, \lambda, c) \to Q$ denote the orbit projection.
Then, there is the following weakly equivariant homeomorphism:
\[
\pi^{-1}(U_{\pi(x)})
\cong
O^l \times \big(  \Omega \times U(1)^{m-n}   \big)    \times    \cone(L_x),
\]
where $U_{\pi(x)}$ is a contractible small open neighborhood of $\pi(x)$, $O^l \subset \mathbb{R}^l$ is a contractible open subset and $\Omega \subset (\mathbb{C}^{\times})^i$ is a $U(1)^i$-invariant open subset.
\end{lem}

\begin{proof}
Let $S \subset Q$ be the stratum that contains $\pi(x)$.
By Lemma \ref{torus split}, there exists $T_S'$ such that $T_S' \cong T^m/\lambda(S) \cong T^{i+(m-n)}$ and $T^m \cong  T_S'  \times \lambda(S).$
Since $S$ is the minimal stratum intersecting $U_{\pi(x)}$,  if $(p, t) \sim_\lambda (p,s) \in U_{\pi(x)} \times T^m / {\sim}_\lambda$ (note that $U_{\pi(x)}$ is contractible), then there exists $S' \ni p$ such that $t^{-1}s \in \lambda(S') \subset \lambda(S)$.
Therefore, we obtain the following decomposition:
\begin{align*}
\pi^{-1}(U_{\pi(x)})=
U_{\pi(x)} \times T^m  / {\sim}_\lambda
&\cong
\left( U_{\pi(x)} \times  \lambda(S) / {\sim}_\lambda \right) \times T_{S}' 
\\
&\cong
\left( U_{\pi(x)} \times  \lambda(S) / {\sim}_\lambda \right) \times T^{i+(m-n)}.
\end{align*}
Here, each $\cong$ denotes a weakly $T^m$–equivariant homeomorphism.
By Lemma \ref{auxiliary lemma}, $\pi^{-1}(U_{\pi(x)})$ is weakly equivariantly homeomorphic to
\begin{align}\label{3.11-A}
O_{\pi(x)} \times \cone(L_{\pi(x)} \times \lambda(S)/{\sim}_{\mu}) \times T^{i+(m-n)}.
\end{align}
Since $O_{\pi(x)}$ is a contractible open subset in $\mathbb{R}^{l+i}$ and $(\mathbb{C}^{\times})^i \cong \mathbb{R}^i \times T^i$, we have the following weakly equivariant homeomorphism:
\begin{align}\label{3.11-B}
O_{\pi(x)} \times T^{i+(m-n)} \cong O^l \times \left( \Omega \times U(1)^{m-n} \right),
\end{align}
where $O^l \subset \mathbb{R}^l$ is a contractible open subset and $\Omega \subset (\mathbb{C}^{\times})^i$ is a $U(1)^i$-invariant open subset.
By (\ref{3.11-A}), (\ref{3.11-B}) and (\ref{def of L_x}), we obtain the following weakly equivariant homeomorphism:
\begin{align*}
\pi^{-1}(U_{\pi(x)})
&\cong
O_{\pi(x)} \times \cone(L_{\pi(x)} \times \lambda(S)/{\sim}_{\mu}) \times T^{i+(m-n)} \\
&\cong
O^l \times \left( \Omega \times U(1)^{m-n} \right) \times \cone(L_{\pi(x)} \times \lambda(S)/{\sim}_{\mu})  \\
&\cong
O^l \times \big(  \Omega \times U(1)^{m-n}   \big) \times \cone(L_x).
\end{align*}
\end{proof}

\subsection{Proof of Lemma \ref{auxiliary lemma}}\label{subsubsection of proof of auxiliary lemma}
In this subsection, we give a proof of Lemma \ref{auxiliary lemma}.

By Definition \ref{def of pseudomanifold}, we have a homeomorphism
\begin{align}\label{base homeomorphism}
\varphi_{\pi(x)}: U_{\pi(x)} \to O_{\pi(x)} \times \cone(L_{\pi(x)}).
\end{align}
Using the definition of open cone (see Definition \ref{def of open cone}), we obtain an equivariant homeomorphism
\begin{align}
U_{\pi(x)} \times  \lambda(S) / {\sim}_\lambda 
&\cong
 O_{\pi(x)} \times \cone(L_{\pi(x)}) \times  \lambda(S) / {\sim}_\lambda'  \notag \\
&=
 O_{\pi(x)} \times \left( L_{\pi(x)} \times [0,1)/L_{\pi(x)} \times \{0\} \right) \times  \lambda(S) / {\sim}_\lambda' .
\end{align}
Here, the equivalence relation $\sim_{\lambda}'$ on $O_{\pi(x)} \times \left( L_{\pi(x)} \times [0,1)/L_{\pi(x)} \times \{0\} \right) \times  \lambda(S) / {\sim}_\lambda'$ is defined by 
\[
(a,[v,r],t)=(\varphi_{\pi(x)}(p),t) \sim_\lambda' (\varphi_{\pi(x)}(p),s)=(a,[v,r],s)
\overset{\mathrm{def}}{\iff}
(p,t) \sim_\lambda (p,s) \in  U_{\pi(x)} \times  \lambda(S) / {\sim}_\lambda ,
\]
for $(a,[v,r],t), (a,[v,r],s) \in O_{\pi(x)} \times \left( L_{\pi(x)} \times [0,1)/L_{\pi(x)} \times \{0\} \right)  \times \lambda(S)$, where $p \in U_{\pi(x)}$ satisfies $\varphi_{\pi(x)}(p) = (a, [v,r])$.
Consider the following diagram:
\[
\begin{tikzcd}[ampersand replacement=\&]
 O_{\pi(x)} \times L_{\pi(x)} \times \lambda(S) \times [0,1) 
 \ar{d}{p_1}[']{/{\sim}_\mu}             \ar{rd}{\text{open cone}}[']{c_2}
\&
\\
O_{\pi(x)} \times \left( L_{\pi(x)} \times \lambda(S)/{\sim}_\mu \right)  \times [0,1)
\ar{d}{c_1}[']{\text{open cone}}
  \&
O_{\pi(x)} \times \cone(L_{\pi(x)}) \times \lambda(S)
\ar{d}{/{\sim}_{\lambda}'}[']{p_2}
\\
O_{\pi(x)} \times \cone \left( L_{\pi(x)} \times \lambda(S)/{\sim}_\mu \right) 
  \&
O_{\pi(x)} \times \cone(L_{\pi(x)}) \times \lambda(S) / {\sim}_{\lambda}'
\end{tikzcd}
\]
Note that the maps $c_1 \circ p_1$ and $p_2 \circ c_2$ are equivariant quotient maps.
We will prove that the following maps are well-defined:
\begin{align}\label{quotient map A}
\begin{array}{rccc}
f:&O_{\pi(x)} \times \cone \left( L_{\pi(x)} \times \lambda(S)/{\sim}_\mu \right) 
&\to
&O_{\pi(x)} \times \cone(L_{\pi(x)}) \times \lambda(S) / {\sim}_{\lambda}' \\
 & \rotatebox{90}{$\in$}&               & \rotatebox{90}{$\in$} \\
&c_1 \circ p_1(a, v, t, r)
&\mapsto
&p_2 \circ c_2(a, v, t, r)
\end{array}
\end{align}
and
\begin{align}\label{quotient map B}
\begin{array}{rccc}
g:&O_{\pi(x)} \times \cone(L_{\pi(x)}) \times \lambda(S) / {\sim}_{\lambda}' 
&\to
&O_{\pi(x)} \times \cone \left( L_{\pi(x)} \times \lambda(S)/{\sim}_\mu \right) 
\\
 & \rotatebox{90}{$\in$}&               & \rotatebox{90}{$\in$} \\
&p_2 \circ c_2(a, v, t, r)
&\mapsto
&c_1 \circ p_1(a, v, t, r)
\end{array}
\end{align}
where $(a, v, t, r)\in  O_{\pi(x)} \times L_{\pi(x)} \times \lambda(S) \times [0,1).$
We first show that $f$ is well-defined. To prove this, we show the following claim.
\begin{claim}
For any two elements
$
(a_1, v_1, t_1, r_1), (a_2, v_2, t_2, r_2) \in  O_{\pi(x)} \times L_{\pi(x)} \times \lambda(S) \times [0,1)
$,
if
\[
c_1 \circ p_1(a_1, v_1, t_1, r_1) = c_1 \circ p_1(a_2, v_2, t_2, r_2),
\]
then
\[
p_2 \circ c_2(a_1, v_1, t_1, r_1) = p_2 \circ c_2(a_2, v_2, t_2, r_2).
\]
\end{claim}
\begin{proof}
\begin{enumerate}[label=\underline{{Step} \arabic*.}, leftmargin=*]
    \item 
Suppose
\begin{align*}
c_1 \circ p_1(a_1, v_1, t_1, r_1) = c_1 \circ p_1(a_2, v_2, t_2, r_2)
\quad \text{and} \quad
p_1(a_1, v_1, t_1, r_1) \neq p_1(a_2, v_2, t_2, r_2).
\end{align*}
By the definition of $c_1$ (i.e., by the definition of the open cone), the restriction map 
$c_1 |_{O_{\pi(x)} \times \left( L_{\pi(x)} \times \lambda(S)/{\sim}_\mu \right)  \times (0,1)}$
is injective. Since $p_1(a_1, v_1, t_1, r_1) \neq p_1(a_2, v_2, t_2, r_2)$,
we have
$c_1 \circ p_1(a_1, v_1, t_1, r_1) \neq c_1 \circ p_1(a_2, v_2, t_2, r_2)$.
Therefore, $r_1=r_2=0$.
Furthermore, since the $O_{\pi(x)}$ factor is projected identically by $p_1$ and $c_1$,
we obtain $a_1=a_2$.
Therefore,
$p_2 \circ c_2(a_1, v_1, t_1, r_1)= p_2 (a_1, [v, 0], t_1)$ and
$p_2 \circ c_2(a_2, v_2, t_2, r_2)= p_2 (a_1, [v, 0], t_2)$,
where $[v, 0] \in \cone(L_{\pi(x)})$ is the cone vertex.
Consider the inverse of the homeomorphism $\varphi_{\pi(x)}$ which defined in \eqref{base homeomorphism}
\[
\varphi_{\pi(x)}^{-1}: O_{\pi(x)} \times \cone(L_{\pi(x)}) \to U_{\pi(x)}.
\]
By Definition \ref{def of pseudomanifold}-3 (c), 
we have
$\varphi_{\pi(x)}^{-1}(a_1, [v, 0]) \in U_{\pi(x)} \cap Q_{l+i}$.
Since $U_{\pi(x)}$ is small, $S$ is the minimal stratum that intersects $U_{\pi(x)}$. Therefore, we obtain $U_{\pi(x)}\cap Q_{l+i}=U_{\pi(x)}\cap S$. Hence,
$\varphi_{\pi(x)}^{-1}(a_1, [v, 0]) \in S$.
To show that 
$p_2 (a_1, [v, 0], t_1) = p_2 (a_1, [v, 0], t_2)$,
it suffices to show that $t_{1}^{-1} t_2 \in \lambda(S)$. This always holds since $t_1, t_2 \in \lambda(S)$.

\item
Next, suppose that $p_1(a_1, v_1, t_1, r_1) = p_1(a_2, v_2, t_2, r_2)$.
If $r_{1}=r_{2}=0$, then $[v_1,r_1], [v_2,r_2] \in \cone(L_{\pi(x)})$ are both the cone vertex, so we can apply the similar argument to Step 1.
Therefore, it suffices to consider the case where
$r_1, r_2 \neq 0$.
Then, by the definition of $p_1$, we have
$a_1=a_2$, $r_1=r_2 \neq 0$, $v_1=v_2$ and $t_{1}^{-1} t_2 \in \mu(R)$ for some stratum $R \subset L_{\pi(x)}$ such that $R \ni v_1$.
Hence, 
$(a_1, v_1, t_1, r_1), (a_2, v_2, t_2, r_2) \in O_{\pi(x)} \times R \times \lambda(S) \times (0,1)$.
By (\ref{unique face}), we have a homeomorphism
\[
O_{\pi(x)} \times R \times (0,1) \cong U_{\pi(x)} \cap S',
\]
where $S'=\nu(R) \subset Q$.
Therefore, to show that
\[
p_2 \circ c_2(a_1, v_1, t_1, r_1)=p_2(a_1, [v_1, r_1], t_1)=p_2(a_1, [v_1, r_1], t_2)=p_2 \circ c_2(a_2, v_2, t_2, r_2),
\]
it suffices to show that $t_{1}^{-1} t_2 \in \lambda(S')$. This always holds since
$\lambda(S')=\mu(R)$ and $t_{1}^{-1} t_2 \in \mu(R)$.
\end{enumerate}
\end{proof}
We next show that $g$ is well-defined. To prove this, we establish the following claim.
\begin{claim}
For any two elements
$
(a_1, v_1, t_1, r_1), (a_2, v_2, t_2, r_2) \in  O_{\pi(x)} \times L_{\pi(x)} \times \lambda(S) \times [0,1)
$,
if
\[
p_2 \circ c_2(a_1, v_1, t_1, r_1) = p_2 \circ c_2(a_2, v_2, t_2, r_2),
\]
then
\[
c_1 \circ p_1(a_1, v_1, t_1, r_1) = c_1 \circ p_1(a_2, v_2, t_2, r_2).
\]
\end{claim}
\begin{proof}
\begin{enumerate}[label=\underline{{Step} \arabic*.}, leftmargin=*]
\item
Suppose that
\[
p_2 \circ c_2(a_1, v_1, t_1, r_1) = p_2 \circ c_2(a_2, v_2, t_2, r_2)
\quad \text{and} \quad
c_2(a_1, v_1, t_1, r_1) \neq c_2(a_2, v_2, t_2, r_2).
\]
Then, we immediately have $a_1=a_2$.
By the definition of $c_2$, if $r_1=r_2=0$, then
$c_2(a_1, v_1, t_1, r_1) = c_2(a_2, v_2, t_2, r_2)$.
This contradicts our assumption. Therefore, we obtain
$v_1=v_2$, $r_1=r_2\neq0$ and $t_{1}^{-1} t_2 \in \lambda(S')$, for some stratum $S' \subset Q$ such that 
\[
U_{\pi(x)} \cap S' \ni \varphi_{\pi(x)}^{-1} \circ \rho \circ c_2(a_1, v_1, r_1, t_1).
\]
Here, 
\[
\rho : O_{\pi(x)} \times \cone(L_{\pi(x)}) \times \lambda(S) 
\to O_{\pi(x)} \times \cone(L_{\pi(x)})
\]
is the natural projection.
Since
$r_1=r_2\neq0$, it follows from (\ref{unique face}) that
there exists a unique stratum
$R \subset L_{\pi(x)}$ such that $\nu(R)=S'$ and $v_1 \in R$.
Since $\lambda(S')=\mu(R)$ and $t_{1}^{-1} t_2 \in \lambda(S')$, we have $t_{1}^{-1} t_2 \in \mu(R)$.
Therefore, 
\[
p_1(a_1, v_1, t_1, r_1)= p_1(a_2, v_2, t_2, r_2),
\]
and in particular,
$c_1 \circ p_1(a_1, v_1, t_1, r_1) = c_1 \circ p_1(a_2, v_2, t_2, r_2).$

\item
Next, suppose
$c_2(a_1, v_1, t_1, r_1) = c_2(a_2, v_2, t_2, r_2)$.
Then one of the following two cases must hold:
\begin{itemize}
\item
$a_1=a_2$, $v_1=v_2$, $r_1=r_2\neq0$ and $t_1=t_2$,

\item
$a_1=a_2$, $r_1=r_2=0$ and $t_1=t_2$.
\end{itemize}
In the first case, the equality
\[
c_1 \circ p_1(a_1, v_1, t_1, r_1) = c_1 \circ p_1(a_2, v_2, t_2, r_2)
\]
is obvious.
In the second case, since $r_1=r_2=0$, the definition of $c_1$ implies that the same equality holds.
\end{enumerate}
\end{proof}

We next show that $f$ is a homeomorphism.
\begin{claim}
Both $f$ and $g$ are continuous, and $g=f^{-1}$. Hence, $f$ is a homeomorphism.
\end{claim}
\begin{proof}
By the definitions of $f$ and $g$, we have $g=f^{-1}$.
Moreover, the following identities hold:
$f \circ (c_1 \circ p_1)=p_2 \circ c_2$ and $g \circ (p_2 \circ c_2)=c_1 \circ p_1$. 
Since both
$c_1 \circ p_1$ and $p_2 \circ c_2$
are quotient maps, it follows that $f$ and $g$ are continuous.
Therefore, $f$ is a homeomorphism.
\end{proof}

We finally show that $f$ is an equivariant map.
For any $s \in \lambda(S)$, since $c_1 \circ p_1$ and $p_2 \circ c_2$ are equivariant, we have
\begin{align*}
f(s \cdot c_1 \circ p_1 (a, v, t, r) ) &= f( c_1 \circ p_1 (a, v, s \cdot t, r) ) \\
&= p_2 \circ c_2(a, v, s \cdot t, r) \\
&= s \cdot \left( p_2 \circ c_2(a, v,  t, r) \right) .
\end{align*}
Therefore, $f$ is equivarinat.
Hence, we obtain the following equivariant homeomorphism:
\[
U_{\pi(x)} \times  \lambda(S) / {\sim}_\lambda 
\cong
O_{\pi(x)} \times \cone(L_{\pi(x)} \times \lambda(S)/{\sim}_{\mu}).
\]
This proves Lemma \ref{auxiliary lemma}.

\subsection{The canonical model is a locally standard $T$-pseudomanifold}\label{prepare-3}
We now prove Theorem \ref{canonical model is a T-pseudomanifold}, i.e., the canonical model $X(Q, \lambda, c)$ is a locally standard $T$-pseudomanifold.

\begin{proof}
By Lemmas~\ref{compactness of canonical model}--\ref{2nd of canonical model}, 
the canonical model $X(Q, \lambda, c)$ is compact, Hausdorff, and second-countable. 
Moreover, by Lemma~\ref{condition of pseudomanifold-1}, it is a manifold stratified space (this implies \eqref{cond-1} in Section \ref{Notation}).

For the case $n = 0$ and $l = 0$, the topological stratified pseudomanifold $Q$ is a set of points with the discrete topology. 
Its compactness implies that $Q$ is a finite set, say $Q = \{ p_1, \ldots, p_s \}$. 
Therefore, we have
\[
X(Q, \lambda, c) = \{ p_1, \ldots, p_s \} \times T^m / {\sim} \cong T^m \sqcup \cdots \sqcup T^m.
\]
In this case, $X(Q, \lambda, c)$ is a locally standard $T$-pseudomanifold.

For the remainder of the proof, we assume $n \geq 0$, $l \geq 0$, and $l + n \neq 0$.  
We consider the natural filtration of $X(Q, \lambda, c) = X(Q, \lambda, c)_{l+m+n}$ given by
\[
\mathfrak{X} : 
X(Q, \lambda, c)_{l+m+n}
\supsetneq
X(Q, \lambda, c)_{l+m+n-2}
\supset \cdots \supset
X(Q, \lambda, c)_{l+2i+(m-n)}
\supset \cdots \supset
X(Q, \lambda, c)_{l+(m-n)}
\supsetneq \emptyset,
\]
where
\[
X(Q, \lambda, c)_{l+2i+(m-n)} := X(Q_{l+i}, \lambda_{l+i}, c) = Q_{l+i} \times T^m / {\sim_{\lambda_{l+i}}}.
\]
By Lemma~\ref{condition of pseudomanifold-2}, $X(Q, \lambda, c)$ satisfies Definition~\ref{def of T-pseudomanifold}-1.
It is easy to check that \eqref{cond-2} and \eqref{cond-3} in Section~\ref{Notation}.
We now claim that $X(Q, \lambda, c)_{l+2i+(m-n)}$ consists of all $k$-dimensional orbits with $k \leq i + (m-n)$.  
Let $p \in Q_{l+i} \setminus Q_{l+i-1}$, and let $S$ be the stratum containing $p$.  
By Definition~\ref{def of char data}, the isotropy group satisfies $\lambda(S) \cong T^{n-i}$.  
Therefore, the orbit over $p$ in $X(Q, \lambda, c)$ has dimension $m - (n - i) = i + (m - n)$, which proves the claim.

It is enough to verify Definition~\ref{def of T-pseudomanifold}-2. We use induction on the dimension of the orbit space.
If $\dim Q = l+n = 0$, then
\[
X(Q, \lambda, c) \cong T^m \sqcup \cdots \sqcup T^m
\]
as shown above.  
Assume that for all $Q$ with $\dim Q < l+n$, the space $X(Q, \lambda, c)$ is a locally standard $T$-pseudomanifold.
We follow the notation in Construction~\ref{cons of link}.  
Let $U_x = \pi^{-1}(U_{\pi(x)})$.  
Since $U_{\pi(x)}$ is an open neighborhood of $\pi(x)$, the set $U_x$ is a $T^m$-invariant open neighborhood of $x$.  
By Construction~\ref{cons of link}, we have
\[
L_x = X(L_{\pi(x)}, \mu, 0).
\]
By the induction hypothesis, $L_x$ is a locally standard $\lambda(S) (= T_x)$-pseudomanifold.  
By Lemma~\ref{universality of quotient}, there is a weakly equivariant homeomorphism
\[
\varphi_x: U_x \to O^l \times \left( \Omega \times U(1)^{m-n} \right) \times \cone(L_x),
\]
where $T^m$ acts on $\left( \Omega \times U(1)^{m-n} \right) \times \cone(L_x)$ via the isomorphism
\[
T^m \cong T^m / \lambda(S) \times \lambda(S) = T^m / T_x \times T_x.
\]
This shows that the triple $(U_x, L_x, \varphi_x)$ satisfies Definition~\ref{def of T-pseudomanifold}-2.
\end{proof}

\section{Unit sphere with standard $T^m$-action}\label{section: unit sphere}
In this section, we discuss the unit sphere with the standard $T^m$-action as a fundamental example of a locally standard $T$-pseudomanifold.

\begin{definition}[unit sphere with the standard $T^m$-action]\label{def of standard sphere}
We call the following spheres with torus actions the {\it unit spheres with the standard $T^m=U(1)^m$-action}.

\begin{enumerate}
\item {\it Odd-dimensional case:}
The unit sphere in $\mathbb{C}^m$ defined by
\[
\mathbb{S}^{2m-1} := \left\{ (z_1, \dots, z_m) \in \mathbb{C}^m \;\middle|\; \sum_{i=1}^m |z_i|^2 = 1 \right\}
\]
is equipped with the standard (coordinatewise) $T^m$-action given by
\[
(t_1, \dots, t_m) \cdot (z_1, \dots, z_m) := (t_1 z_1, \dots, t_m z_m),
\]
for $(t_1, \dots, t_m) \in T^m$ and $(z_1, \dots, z_m) \in \mathbb{S}^{2m-1}$.

\item {\it Even-dimensional case:}
	The unit sphere in $\mathbb{C}^{m} \times \mathbb{R}$ defined by
\[
\mathbb{S}^{2m} := \left\{ (z_1, \dots, z_m, x) \in \mathbb{C}^m \times \mathbb{R} \;\middle|\; \sum_{i=1}^n |z_i|^2 + x^2 = 1 \right\}
\]
is equipped with the standard $T^m$-action given by
\[
(t_1, \dots, t_m) \cdot (z_1, \dots, z_m, x) := (t_1 z_1, \dots, t_m z_m, x),
\]
for $(t_1, \dots, t_m) \in T^m$ and $(z_1, \dots, z_m, x) \in \mathbb{S}^{2m}$.
\end{enumerate}
\end{definition}

By the following proposition, the unit sphere with the standard $T^m$-action admits the structure of a locally standard $T$-pseudomanifold.

\begin{prop}\label{prop: standard sphere}
A unit sphere with the standard $T^m$-action admits the structure of a locally standard $T$-pseudomanifold. Specifically:
\begin{enumerate}
\item {\it Odd-dimensional case:}
Let $X=\mathbb{S}^{2m-1}$, and let $\mathfrak{X}$ denote the filtration by orbit dimension.
Then $(X, \mathfrak{X})$ is a $(2m - 1)$-dimensional locally standard $T^m$-pseudomanifold, with $n = m - 1$ and $l = 0$.

\item {\it Even-dimensional case:}
Let $X=\mathbb{S}^{2m}$, and let $\mathfrak{X}$ denote the filtration by orbit dimension.
Then $(X, \mathfrak{X})$ is a $2m$-dimensional locally standard $T^m$-pseudomanifold, with $n = m$ and $l = 0$.
\end{enumerate}
\end{prop}

\begin{proof}
{\it Odd-dimensional case:}
The unit sphere $\mathbb{S}^{2m-1}$ with the standard $T^m$-action can be realized as the {\it moment-angle manifold} associated with the $(m-1)$-simplex $\Delta^{m-1}$ (see \cite[Example 6.1.6]{BP12} and \cite[Chapter 6]{BP12} for details).
By \cite[Proposition 6.2.2]{BP12}, $\mathbb{S}^{2m-1}$ is $T^m$-equivariantly homeomorphic to the canonical model labeled by the standard basis on $\Delta^{m-1}$.
Since $\Delta^{m-1}$ is an $(m-1)$-dimensional polytope, we have $n = m - 1$ and $l = 0$.
Moreover, by Proposition~\ref{polytope is pseudomanifold}, $\Delta^{m-1}$ is a topological stratified pseudomanifold.
Therefore, by Theorem~\ref{canonical model is a T-pseudomanifold} the canonical model over $\Delta^{m-1}$ is a locally standard $T^m$-pseudomanifold.
Hence, by Proposition~\ref{inv of T-pseudomanifold}, $\mathbb{S}^{2m-1}$ is also a locally standard $T^m$-pseudomanifold.

\textit{Even-dimensional case:}
The unit sphere $\mathbb{S}^{2m}$ with the standard $T^m$-action is a {\it locally standard torus manifold} (see \cite[Example 7.4.11]{BP12}, and also \cite[Chapter 7.4]{BP12} and \cite{MP06} for details).
In particular, $\mathbb{S}^{2m}$ is a {\it compact locally standard $T^m$-manifold} with a non-empty fixed point set.
Its orbit space is an $m$-dimensional manifold with corners that has exactly two vertices, and thus we have $n = m$ and $l = 0$.
By using the facts which will be proved in Section~\ref{sec: applications}, any compact locally standard $T$-manifold is a locally standard $T$-pseudomanifold (see Proposition~\ref{prop: locally standard T-manifold}).
Therefore, $\mathbb{S}^{2m}$ also admits such a structure.
\end{proof}

\begin{rem}\label{rem: moment-angle manifold}
According to \cite[Proposition 6.2]{Dav14}, any {\it smooth $T$-manifold, modeled on the
standard representation, and that the bundle of principal orbits is trivial}, is $T$-equivariantly diffeomorphic to a moment-angle manifold.
Furthermore, by applying \cite[Theorem 6.2.4]{BP12} and \cite[Proposition 6.2.2]{BP12}, moment-angle manifolds can be realized $T$-equivariantly homeomorphically as canonical models.
Therefore, such spaces are locally standard $T$-pseudomanifolds by Theorem~\ref{canonical model is a T-pseudomanifold}.

For {\it moment-angle manifolds over manifolds with corners}, see also \cite[Remark 14.4]{KK25}.
\end{rem}

The following lemma shows that the standard chart is the open cone of the odd-dimensional unit sphere with the standard $T$-action.

\begin{lem}\label{lem of sphere}
Suppose that $T^m=U(1)^m$ act on $\mathbb{C}^m$ coordinatewise, i.e.,
\[
(t_1, \dots, t_m) \cdot (z_1, \dots, z_m) := (t_1 z_1, \dots, t_m z_m),
\]
for $(t_1, \dots, t_m) \in T^m$ and $(z_1, \dots, z_m) \in \mathbb{C}^m$.
Then $\mathbb{C}^m$ and $\cone(\mathbb{S}^{2m-1})$ are $T^m$-equivariant homeomorphic, 
where $T^m$ acts on $\cone(\mathbb{S}^{2m-1})$ by
\[
t \cdot [x, r] := [t \cdot x, r],
\] 
for $[x, r] \in \cone(\mathbb{S}^{2m-1})= \mathbb{S}^{2m-1} \times [0,1)/(\mathbb{S}^{2m-1} \times \{ 0 \})$.
\end{lem}

\begin{proof}
We define a map
\[
\Phi : \mathbb{C}^m \to \cone(\mathbb{S}^{2m-1}); \quad
z \mapsto
\begin{cases}
\left[\left( \dfrac{z}{\|z\|}, \dfrac{\|z\|}{1 + \|z\|} \right)\right] & \text{if } z \neq 0, \\
 [x, 0] \text{ (cone vertex)}& \text{if } z = 0,
\end{cases}
\]
where $x \in \mathbb{S}^{2m-1}$ is an arbitrary point.
If $z \to 0$, then $\dfrac{\|z\|}{1 + \|z\|} \to 0$, and thus
\[
\Phi(z) \to [x, 0],
\]
which is the cone vertex. Therefore, $\Phi$ is continuous at $z=0$.
Since it is clearly continuous at all other points, we conclude that $\Phi$ is continuous.
We define its inverse
\[
\Psi : \cone(\mathbb{S}^{2m-1}) \to \mathbb{C}^n; \quad
[x, r] \mapsto
\dfrac{r}{1 - r} \cdot x
\]
is also continuous. Therefore, ${\Phi}$ is a homeomorphism.
Since $\|t \cdot z\| = \|z\|$ for $t \in T^m$, we obtain
\[
{\Phi}(t \cdot z) =
\left[ \left( \frac{t \cdot z}{\|t \cdot z\|}, \frac{\|t \cdot z\|}{1 + \|t \cdot z\|} \right) \right]
=
\left[ \left( t \cdot \frac{z}{\|z\|}, \frac{\|z\|}{1 + \|z\|} \right) \right]
= t \cdot {\Phi}(z).
\]
Hence, ${\Phi}$ is a $T^m$-equivariant homeomorphism.
\end{proof}

\section{Applications to the class of topological spaces with $T$-action}\label{sec: applications}

In this section, we apply our main results to several significant classes arising in toric geometry and toric topology.
In particular, we verify the following theorem:
\begin{thm}\label{relations}
The class of locally standard $T$-pseudomanifolds contains both the class of complete toric varieties and the class of compact locally standard $T$-manifolds.
\end{thm}

\subsection{Complete toric variety}
In this subsection, we prove the following proposition.
\begin{prop}\label{complete toric is locally standard}
A complete toric variety is a locally standard $T$-pseudomanifold.
\end{prop}

Let $X_{\Sigma}$ be a complete toric variety.  
As shown in \cite{Jor98} and \cite{CLS11}, $X_{\Sigma}$ is equivariantly homeomorphic to a canonical model:
\begin{align}\label{MacPherson thm}
X_{\Sigma} \cong X(B, \lambda, 0),
\end{align}
where $B$ denotes the unit closed ball equipped with a filtration $\mathfrak{B}$ determined by the complete fan $\Sigma$ associated with $X_{\Sigma}$, as described below.  
The following construction is taken from \cite{Jor98} and \cite{CLS11}.

\begin{cons}[filtration of $B^n$]\label{filtration construction}
Let $\Sigma$ be a complete fan in $\mathbb{R}^n$.
Let $S^{n-1}$ denote the unit sphere in $\mathbb{R}^n$. 
The intersection of each cone $\sigma$ in $\Sigma$ with $S^{n-1}$ defines a {\it spherical complex} $C_{\Sigma}$ on $S^{n-1}$.
Each cone $\sigma$ determines a {\it spherical dual} $\widehat{\sigma}$ in the unit closed ball $B^n$ as follows.
Consider a barycentric subdivision of $C_{\Sigma}$, denoted by $C'_{\Sigma}$.  
If $\sigma = \{0\}$ (the zero-dimensional cone in $\Sigma$), then we define $\widehat{\sigma} := B^n$. If $\sigma$ is not the zero-dimensional cone,
\begin{align}\label{def of spherical dual}
\widehat{\sigma}:=
\bigcup \{  \theta \in C'_{\Sigma} \mid \text{vertices of } \theta \text{ are barycenters of } \tau \cap S^{n-1} \text{ with } \sigma \subset \tau \in \Sigma \}
\end{align}
The spherical dual is illustrated in Figure~\ref{fig: spherical}.
\begin{figure}[htbp]
\begin{tikzpicture}[scale=0.8]
\begin{scope}[shift={(0,0)}, scale=0.1, line cap=round, line join=round]
\draw[line width=0.4mm] (20,32) ++(240:26) arc (240:360:26);
\draw[line width=0.4mm] (30,22) .. controls (20,16) .. (18,14);
\draw[line width=0.4mm] (30,22) .. controls (34,20) .. (36,18);
\draw[line width=0.4mm] (18,14) .. controls (20,11.2) .. (24,8);
\draw[line width=0.4mm] (36,18) .. controls (34.6,13.8) .. (32,10);
\draw[line width=0.4mm] (24,8) .. controls (27.7,8.4) .. (32,10);
\draw[line width=0.4mm] (30,22) .. controls (30.9,26.5) .. (30,32);
\draw[line width=0.4mm] (18,14) .. controls (11.4,16.6) .. (8,22);
\draw[line width=0.4mm] (8,22) .. controls (9.6,28.4) .. (14,34);
\draw[line width=0.4mm] (14,34) .. controls (23.9,35.4) .. (30,32);
\draw[line width=0.4mm] (30,32) .. controls (36.8,33.2) .. (42,32);
\draw[line width=0.4mm] (42,32) .. controls (44.4,27.4) .. (43.3,20.4);
\draw[line width=0.4mm] (43.3,20.4) .. controls (38.9,18.3) .. (36,18);
\draw[line width=0.4mm] (30,32) .. controls (29.2,35.8) .. (30,40);
\draw[line width=0.4mm] (42,32) .. controls (43.3,32.9) .. (44,36);
\draw[line width=0.4mm] (14,34) .. controls (12.5,35.2) .. (11.3,37.8);
\draw[line width=0.4mm] (8,22) .. controls (6,21.7) .. (4.5,22.1);
\draw[line width=0.4mm] (18,14) .. controls (17.2,11.3) .. (14,10);
\draw[line width=0.4mm] (24,8) .. controls (22.1,6.6) .. (18,6);
\draw[line width=0.4mm] (32,10) .. controls (32.2,9.5) .. (32.5,9.3);
\node at (44,10) {$C_{\Sigma}$};
\end{scope}

\begin{scope}[shift={(6,0)}, scale=0.1, line cap=round, line join=round]
\draw[line width=0.4mm] (20,32) ++(240:26) arc (240:360:26);
\draw[line width=0.4mm] (30,22) .. controls (20,16) .. (18,14);
\draw[line width=0.4mm] (30,22) .. controls (34,20) .. (36,18);
\draw[line width=0.4mm] (18,14) .. controls (20,11.2) .. (24,8);
\draw[line width=0.4mm] (36,18) .. controls (34.6,13.8) .. (32,10);
\draw[line width=0.4mm] (24,8) .. controls (27.7,8.4) .. (32,10);
\draw[line width=0.4mm] (30,22) .. controls (30.9,26.5) .. (30,32);
\draw[line width=0.4mm] (18,14) .. controls (11.4,16.6) .. (8,22);
\draw[line width=0.4mm] (8,22) .. controls (9.6,28.4) .. (14,34);
\draw[line width=0.4mm] (14,34) .. controls (23.9,35.4) .. (30,32);
\draw[line width=0.4mm] (30,32) .. controls (36.8,33.2) .. (42,32);
\draw[line width=0.4mm] (42,32) .. controls (44.4,27.4) .. (43.3,20.4);
\draw[line width=0.4mm] (43.3,20.4) .. controls (38.9,18.3) .. (36,18);
\draw[line width=0.4mm] (30,32) .. controls (29.2,35.8) .. (30,40);
\draw[line width=0.4mm] (42,32) .. controls (43.3,32.9) .. (44,36);
\draw[line width=0.4mm] (14,34) .. controls (12.5,35.2) .. (11.3,37.8);
\draw[line width=0.4mm] (8,22) .. controls (6,21.7) .. (4.5,22.1);
\draw[line width=0.4mm] (18,14) .. controls (17.2,11.3) .. (14,10);
\draw[line width=0.4mm] (24,8) .. controls (22.1,6.6) .. (18,6);
\draw[line width=0.4mm] (32,10) .. controls (32.2,9.5) .. (32.5,9.3);
\fill (20,32) circle (1.2);
\draw[thick] (20,32) -- (50,2);
\filldraw[red] (30,22) circle (1);
\node at (44,4) {$\sigma$};
\node[red] at (20,22) {$\sigma \cap S^{n-1}$};
\end{scope}

\begin{scope}[shift={(12,0)}, scale=0.1, line cap=round, line join=round]
\filldraw[fill=gray!20, draw=none] (20,32) -- (50,2) -- (16,-4) -- cycle;
\draw[line width=0.4mm] (20,32) ++(240:26) arc (240:360:26);
\draw[red][line width=0.4mm] (30,22) .. controls (20,16) .. (18,14);
\draw[line width=0.4mm] (30,22) .. controls (34,20) .. (36,18);
\draw[line width=0.4mm] (18,14) .. controls (20,11.2) .. (24,8);
\draw[line width=0.4mm] (36,18) .. controls (34.6,13.8) .. (32,10);
\draw[line width=0.4mm] (24,8) .. controls (27.7,8.4) .. (32,10);
\draw[line width=0.4mm] (30,22) .. controls (30.9,26.5) .. (30,32);
\draw[line width=0.4mm] (18,14) .. controls (11.4,16.6) .. (8,22);
\draw[line width=0.4mm] (8,22) .. controls (9.6,28.4) .. (14,34);
\draw[line width=0.4mm] (14,34) .. controls (23.9,35.4) .. (30,32);
\draw[line width=0.4mm] (30,32) .. controls (36.8,33.2) .. (42,32);
\draw[line width=0.4mm] (42,32) .. controls (44.4,27.4) .. (43.3,20.4);
\draw[line width=0.4mm] (43.3,20.4) .. controls (38.9,18.3) .. (36,18);
\draw[line width=0.4mm] (30,32) .. controls (29.2,35.8) .. (30,40);
\draw[line width=0.4mm] (42,32) .. controls (43.3,32.9) .. (44,36);
\draw[line width=0.4mm] (14,34) .. controls (12.5,35.2) .. (11.3,37.8);
\draw[line width=0.4mm] (8,22) .. controls (6,21.7) .. (4.5,22.1);
\draw[line width=0.4mm] (18,14) .. controls (17.2,11.3) .. (14,10);
\draw[line width=0.4mm] (24,8) .. controls (22.1,6.6) .. (18,6);
\draw[line width=0.4mm] (32,10) .. controls (32.2,9.5) .. (32.5,9.3);
\fill (20,32) circle (1.2);
\draw[thick] (20,32) -- (50,2);
\draw[thick] (20,32) -- (16,-4);
\filldraw[red] (18,14) circle (1);
\filldraw[red] (30,22) circle (1);
\node at (34,4) {$\sigma$};
\node[red] at (20,22) {$\sigma \cap S^{n-1}$};
\end{scope}

\begin{scope}[shift={(0,-5.5)}, scale=0.1, line cap=round, line join=round]
\draw[line width=0.4mm] (20,32) ++(240:26) arc (240:360:26);
\draw[line width=0.4mm] (30,22) .. controls (20,16) .. (18,14);
\draw[line width=0.4mm] (30,22) .. controls (34,20) .. (36,18);
\draw[line width=0.4mm] (18,14) .. controls (20,11.2) .. (24,8);
\draw[line width=0.4mm] (36,18) .. controls (34.6,13.8) .. (32,10);
\draw[line width=0.4mm] (24,8) .. controls (27.7,8.4) .. (32,10);
\draw[line width=0.4mm] (30,22) .. controls (30.9,26.5) .. (30,32);
\draw[line width=0.4mm] (18,14) .. controls (11.4,16.6) .. (8,22);
\draw[line width=0.4mm] (8,22) .. controls (9.6,28.4) .. (14,34);
\draw[line width=0.4mm] (14,34) .. controls (23.9,35.4) .. (30,32);
\draw[line width=0.4mm] (30,32) .. controls (36.8,33.2) .. (42,32);
\draw[line width=0.4mm] (42,32) .. controls (44.4,27.4) .. (43.3,20.4);
\draw[line width=0.4mm] (43.3,20.4) .. controls (38.9,18.3) .. (36,18);
\draw[line width=0.4mm] (30,32) .. controls (29.2,35.8) .. (30,40);
\draw[line width=0.4mm] (42,32) .. controls (43.3,32.9) .. (44,36);
\draw[line width=0.4mm] (14,34) .. controls (12.5,35.2) .. (11.3,37.8);
\draw[line width=0.4mm] (8,22) .. controls (6,21.7) .. (4.5,22.1);
\draw[line width=0.4mm] (18,14) .. controls (17.2,11.3) .. (14,10);
\draw[line width=0.4mm] (24,8) .. controls (22.1,6.6) .. (18,6);
\draw[line width=0.4mm] (32,10) .. controls (32.2,9.5) .. (32.5,9.3);
\filldraw[black] (20,24.8) circle (1);
\filldraw[black] (28,14.4) circle (1);
\filldraw[black] (8,22) circle (1);
\filldraw[black] (14,34) circle (1);
\filldraw[black] (42,32) circle (1);
\filldraw[black] (43.3,20.4) circle (1);
\filldraw[black] (36.26,24.88) circle (1);
\filldraw[black] (30,32) circle (1);
\filldraw[black] (18,14) circle (1);
\filldraw[black] (36,18) circle (1);
\filldraw[black] (32,10) circle (1);
\filldraw[black] (24,8) circle (1);
\filldraw[black] (30,22) circle (1);
\filldraw[black] (30.6,26.5) circle (1);
\filldraw[black] (24,18.5) circle (1);
\filldraw[black] (34,20) circle (1);
\filldraw[black] (23.9,34.5) circle (1);
\filldraw[black] (11.4,17.5) circle (1);
\filldraw[black] (9.8,28.4) circle (1);
\filldraw[black] (20,11.2) circle (1);
\filldraw[black] (34.6,13.8) circle (1);
\filldraw[black] (27.7,8.4) circle (1);
\filldraw[black] (36.8,32.8) circle (1);
\filldraw[black] (44,27.4) circle (1);
\filldraw[black] (40,18.9) circle (1);
\draw[black, line width=0.5mm] (20,24.8) -- (23.9,34.5);
\draw[black, line width=0.5mm] (20,24.8) -- (11.4,17.5);
\draw[black, line width=0.5mm] (20,24.8) -- (9.8,28.4);
\draw[black, line width=0.5mm] (20,24.8) -- (30.6,26.5);
\draw[black, line width=0.5mm] (20,24.8) -- (24,18.5);
\draw[black, line width=0.5mm] (20,24.8) -- (8,22);
\draw[black, line width=0.5mm] (20,24.8) -- (14,34);
\draw[black, line width=0.5mm] (20,24.8) -- (18,14);
\draw[black, line width=0.5mm] (20,24.8) -- (30,32);
\draw[black, line width=0.5mm] (20,24.8) -- (30,22);
\draw[black, line width=0.5mm] (36.26,24.88) -- (42,32);
\draw[black, line width=0.5mm] (36.26,24.88) -- (43.3,20.4);
\draw[black, line width=0.5mm] (36.26,24.88) -- (30,22);
\draw[black, line width=0.5mm] (36.26,24.88) -- (30,32);
\draw[black, line width=0.5mm] (36.26,24.88) -- (36,18);
\draw[black, line width=0.5mm] (36.26,24.88) -- (36.8,32.8);
\draw[black, line width=0.5mm] (36.26,24.88) -- (44,27.4);
\draw[black, line width=0.5mm] (36.26,24.88) -- (40,18.9);
\draw[black, line width=0.5mm] (36.26,24.88) -- (34,20);
\draw[black, line width=0.5mm] (36.26,24.88) -- (30.6,26.5);
\draw[black, line width=0.5mm] (28,14.4) -- (30,22);
\draw[black, line width=0.5mm] (28,14.4) -- (18,14);
\draw[black, line width=0.5mm] (28,14.4) -- (36,18);
\draw[black, line width=0.5mm] (28,14.4) -- (32,10);
\draw[black, line width=0.5mm] (28,14.4) -- (24,8);
\draw[black, line width=0.5mm] (28,14.4) -- (20,11.2);
\draw[black, line width=0.5mm] (28,14.4) -- (34.6,13.8);
\draw[black, line width=0.5mm] (28,14.4) -- (27.7,8.4);
\draw[black, line width=0.5mm] (28,14.4) -- (24,18.5);
\draw[black, line width=0.5mm] (28,14.4) -- (34,20);
\node at (44,10) {$C_{\Sigma}'$};
\end{scope}

\begin{scope}[shift={(12,-5.5)}, scale=0.1, line cap=round, line join=round]
\draw[gray, line width=0.4mm] (20,32) ++(240:26) arc (240:360:26);
\draw[gray, line width=0.4mm] (30,22) .. controls (20,16) .. (18,14);
\draw[gray, line width=0.4mm] (30,22) .. controls (34,20) .. (36,18);
\draw[gray, line width=0.4mm] (18,14) .. controls (20,11.2) .. (24,8);
\draw[gray, line width=0.4mm] (36,18) .. controls (34.6,13.8) .. (32,10);
\draw[gray, line width=0.4mm] (24,8) .. controls (27.7,8.4) .. (32,10);
\draw[gray, line width=0.4mm] (30,22) .. controls (30.9,26.5) .. (30,32);
\draw[gray, line width=0.4mm] (18,14) .. controls (11.4,16.6) .. (8,22);
\draw[gray, line width=0.4mm] (8,22) .. controls (9.6,28.4) .. (14,34);
\draw[gray, line width=0.4mm] (14,34) .. controls (23.9,35.4) .. (30,32);
\draw[gray, line width=0.4mm] (30,32) .. controls (36.8,33.2) .. (42,32);
\draw[gray, line width=0.4mm] (42,32) .. controls (44.4,27.4) .. (43.3,20.4);
\draw[gray, line width=0.4mm] (43.3,20.4) .. controls (38.9,18.3) .. (36,18);
\draw[gray, line width=0.4mm] (30,32) .. controls (29.2,35.8) .. (30,40);
\draw[gray, line width=0.4mm] (42,32) .. controls (43.3,32.9) .. (44,36);
\draw[gray, line width=0.4mm] (14,34) .. controls (12.5,35.2) .. (11.3,37.8);
\draw[gray, line width=0.4mm] (8,22) .. controls (6,21.7) .. (4.5,22.1);
\draw[gray, line width=0.4mm] (18,14) .. controls (17.2,11.3) .. (14,10);
\draw[gray, line width=0.4mm] (24,8) .. controls (22.1,6.6) .. (18,6);
\draw[gray, line width=0.4mm] (32,10) .. controls (32.2,9.5) .. (32.5,9.3);
\filldraw[red] (20,24.8) circle (1);
\filldraw[red] (28,14.4) circle (1);

\draw[red, line width=0.5mm] (20,24.8) -- (24,18.5);
\draw[red, line width=0.5mm] (28,14.4) -- (24,18.5);
\node[red] at (31.5,12.5) {$\widehat{\sigma}$};
\end{scope}

\begin{scope}[shift={(6,-5.5)}, scale=0.1, line cap=round, line join=round]
\draw[gray, line width=0.4mm] (20,32) ++(240:26) arc (240:360:26);
\draw[gray, line width=0.4mm] (30,22) .. controls (20,16) .. (18,14);
\draw[gray, line width=0.4mm] (30,22) .. controls (34,20) .. (36,18);
\draw[gray, line width=0.4mm] (18,14) .. controls (20,11.2) .. (24,8);
\draw[gray, line width=0.4mm] (36,18) .. controls (34.6,13.8) .. (32,10);
\draw[gray, line width=0.4mm] (24,8) .. controls (27.7,8.4) .. (32,10);
\draw[gray, line width=0.4mm] (30,22) .. controls (30.9,26.5) .. (30,32);
\draw[gray, line width=0.4mm] (18,14) .. controls (11.4,16.6) .. (8,22);
\draw[gray, line width=0.4mm] (8,22) .. controls (9.6,28.4) .. (14,34);
\draw[gray, line width=0.4mm] (14,34) .. controls (23.9,35.4) .. (30,32);
\draw[gray, line width=0.4mm] (30,32) .. controls (36.8,33.2) .. (42,32);
\draw[gray, line width=0.4mm] (42,32) .. controls (44.4,27.4) .. (43.3,20.4);
\draw[gray, line width=0.4mm] (43.3,20.4) .. controls (38.9,18.3) .. (36,18);
\draw[gray, line width=0.4mm] (30,32) .. controls (29.2,35.8) .. (30,40);
\draw[gray, line width=0.4mm] (42,32) .. controls (43.3,32.9) .. (44,36);
\draw[gray, line width=0.4mm] (14,34) .. controls (12.5,35.2) .. (11.3,37.8);
\draw[gray, line width=0.4mm] (8,22) .. controls (6,21.7) .. (4.5,22.1);
\draw[gray, line width=0.4mm] (18,14) .. controls (17.2,11.3) .. (14,10);
\draw[gray, line width=0.4mm] (24,8) .. controls (22.1,6.6) .. (18,6);
\draw[gray, line width=0.4mm] (32,10) .. controls (32.2,9.5) .. (32.5,9.3);
\filldraw[fill=red!30, draw=none] (30.6,26.5) -- (36.26,24.88) -- (34,20) -- (28,14.4) --(24,18.5) --(20,24.8) -- cycle;
\filldraw[red] (20,24.8) circle (1);
\filldraw[red] (36.26,24.88) circle (1);
\filldraw[red] (28,14.4) circle (1);

\draw[dashed, red, line width=0.25mm] (30,22) -- (30.6,26.5);  
\draw[dashed, red, line width=0.25mm] (30,22) -- (24,18.5);

\draw[dashed, red, line width=0.25mm] (30,22) -- (34,20);
\draw[red, line width=0.5mm] (20,24.8) -- (30.6,26.5);
\draw[red, line width=0.5mm] (20,24.8) -- (24,18.5);

\draw[dashed, red, line width=0.25mm] (20,24.8) -- (30,22);
\draw[dashed, red, line width=0.25mm] (36.26,24.88) -- (30,22);
\draw[red, line width=0.5mm] (36.26,24.88) -- (34,20); 

\draw[red, line width=0.5mm] (36.26,24.88) -- (30.6,26.5); 
\draw[dashed, red, line width=0.25mm] (28,14.4) -- (30,22);
\draw[red, line width=0.5mm] (28,14.4) -- (24,18.5); 
\draw[red, line width=0.5mm] (28,14.4) -- (34,20); 

\node[red] at (31.5,12.5) {$\widehat{\sigma}$};
\end{scope}
\end{tikzpicture}
\centering
\caption{$C_{\Sigma}$, $C_{\Sigma}'$ and spherical duals}
  \label{fig: spherical}   
\end{figure}
The collection of spherical duals defines a filtration of $B^n$:
\[
\mathfrak{B}: B^n = Q_n \supset \cdots \supset Q_i \supset \cdots \supset \emptyset,
\]
where $Q_i = \bigcup \{ \widehat{\sigma} \mid \sigma \in \Sigma ,\, \mathrm{codim}(\sigma) =i \}$.
\end{cons}

To prove Proposition~\ref{complete toric is locally standard}, we prepare the following proposition.

\begin{prop}\label{ball is a topological stratified pseudomanifold}
The ball $(B^n, \mathfrak{B})$ is a topological stratified pseudomanifold.
\end{prop}

The proof of this proposition will be given in the next subsection~\ref{ss: ball is pseudomanifold}.
Assuming this proposition, we may prove Proposition \ref{complete toric is locally standard}.

\begin{proof}[Proof of Proposition \ref{complete toric is locally standard}]
By Proposition \ref{ball is a topological stratified pseudomanifold} and Theorem \ref{canonical model is a T-pseudomanifold}, the canonical model $X(B, \lambda, 0)$ is a locally standard $T$-pseudomanifold.
By \eqref{MacPherson thm}, we have $X_{\Sigma} \cong X(B, \lambda, 0)$.
Therefore, by Proposition \ref{inv of T-pseudomanifold}, $X_{\Sigma}$ is also a locally standard $T$-pseudomanifold.
\end{proof}

The class of complete toric varieties contains all projective toric varieties. Therefore, projective toric varieties are also locally standard $T$-pseudomanifolds.

\begin{rem}
In \cite[Proposition 19]{BS25}, it is shown that any projective toric variety admits the structure of a locally standard $T$-pseudomanifold.
A key difference from our approach lies in how to construct the $T$-invariant open subsets satisfying Definition \ref{def of T-pseudomanifold}-2.
In \cite{BS25}, such a subset is directly constructed on the projective toric variety $X_P$, as the preimage of a certain open subset of the polytope $P$ (see \cite[Proposition 14]{BS25}) under the projection $X_P \to P$ induced by the moment map (see \cite[Chapter 12]{CLS11}).

On the other hand, in our case, we take a $T$-invariant open subset on the canonical model that satisfies Definition \ref{def of T-pseudomanifold}-2 (see Subsection \ref{prepare-3}) and apply the equivariant homeomorphism from Lemma \ref{main thm}.
\end{rem}

\subsection{$(B^n, \mathfrak{B})$ is a topological stratified pseudomanifold}\label{ss: ball is pseudomanifold}
In this subsection, we prove Proposition~\ref{ball is a topological stratified pseudomanifold}, which was stated in the previous subsection.
The proof of the proposition is given in Subsubsection~\ref{ballがpseudomanifoldであること}.
In Subsubsections~\ref{接空間からの記号} -- \ref{図と証明の説明}, we prepare the necessary notation and constructions for the proof.

\subsubsection{Notation}\label{接空間からの記号}
We consider the essential case where $n>0$ and $0 \le i < n$, and let $p \in Q_i \setminus Q_{i-1}$ be a point.
In this case, $p$ lies on the boundary $\partial B^n = S^{n-1}$.
Let $S_p$ denote the $i$-dimensional stratum containing $p$.
In this subsubsection, we establish the notation required for the proof.

\begin{itemize}
\item
Let $T_p S^{n-1}$ be the tangent space of $S^{n-1}$ at $p$.
Then we have the decomposition
\begin{align*}
T_p S^{n-1} \cong T_p S_p \times T_p N_p
\end{align*}
where $T_p S_p \cong \mathbb{R}^i$ is the $i$-dimensional plane tangent to $S_p$ at $p$, and $T_p N_p \cong \mathbb{R}^{n-i-1}$ is the subspace of $T_p S^{n-1}$ orthogonal to $T_p S_p$ and passing through $p$.

\item
Let $\mathcal{H}_p$ be a hyperplane obtained by translating $T_pS^{n-1}$ slightly toward the origin, and denote this translation by $t: T_p S^{n-1} \to \mathcal{H}_p$.
We denote
\[
\mathcal{H}_p S_p := t(T_p S_p), \quad \mathcal{H}_p N_p:=t(T_p N_p).
\]

\item
We set
\begin{align*}
\widetilde{T_p S^{n-1}}:=\mathcal{H}_p \cap B^n \cong B^{n-1}, \quad
\widetilde{T_p S_p}:= \mathcal{H}_p S_p \cap B^n \cong B^{i}, \quad
\widetilde{T_p N_p}:=\mathcal{H}_p N_p \cap B^n \cong B^{n-i-1}.
\end{align*}

Then we have the following decomposition:
\[
\widetilde{T_p S^{n-1}} \cong \widetilde{T_p S_p} \times \widetilde{T_p N_p}.
\]

\[
\begin{tikzpicture}[scale=1]
\begin{scope}
\draw (3.6,2) circle (1);
\draw[thick] (2.4,3.3) -- (5.6,3.3) -- (4.8,2.7) -- (1.6,2.7) -- cycle;
    \fill (3.6,3) circle (2pt);
\fill[red, opacity=0.5, blend mode=screen] (2.4,3.3) -- (5.6,3.3) -- (4.8,2.7) -- (1.6,2.7) -- cycle;

\draw (3.6,2.3) ellipse (0.95 and 0.11);

\node[above] at (6,2) {$\xrightarrow{\hspace{1cm}}$};
\node[above] at (6,2.15) {$t$};

\node[right, above] at (5.6,3.3) {\textcolor{red}{$T_p S^{n-1}$}};
\node[above] at (3.6, 3.3) {$p$};
\end{scope}

\begin{scope}[xshift=5cm]
\draw (3.6,2) circle (1);
\draw (3.6,2.3) ellipse (0.95 and 0.11);
\draw[thick] (2.4,2.6) -- (5.6,2.6) -- (4.8,2.0) -- (1.6,2.0) -- cycle;
\fill[red, opacity=0.5, blend mode=screen] (2.4,2.6) -- (5.6,2.6) -- (4.8,2.0) -- (1.6,2.0) -- cycle;

    \fill (3.6,3) circle (2pt);
\node[right, above] at (5.6,2.6) {\textcolor{red}{$\mathcal{H}_p$}};
\node[above] at (6,2) {$\supset$};
\end{scope}

\begin{scope}[xshift=10cm]
\draw (3.6,2) circle (1);

\draw[thick] (2.4,2.6) -- (5.6,2.6) -- (4.8,2.0) -- (1.6,2.0) -- cycle;

\fill[red, opacity=0.5, blend mode=screen] (3.6,2.3) ellipse (0.95 and 0.11);
\draw (3.6,2.3) ellipse (0.95 and 0.11);
\node[right, above] at (5,2.6) {\textcolor{red}{$\widetilde{T_p S^{n-1}}$}};
\end{scope}
\end{tikzpicture}
\]

\item
Let
\[
f : \widetilde{T_p S^{n-1}} \to \mathcal{H}_p^{[+]} \cap S^{n-1}
\]
be the projection from the center of the ball,
where $\mathcal{H}_p^{[+]}$ denotes the closed half-space of $\mathcal{H}_p$ containing $p$.
This map $f$ is a homeomorphism, and its boundary satisfies
$\partial \left( \widetilde{T_p S^{n-1}} \right) = \mathcal{H}_p \cap~S^{n-1}$; that is, $f$ restricts to the identity map on the boundary (see the figure below).
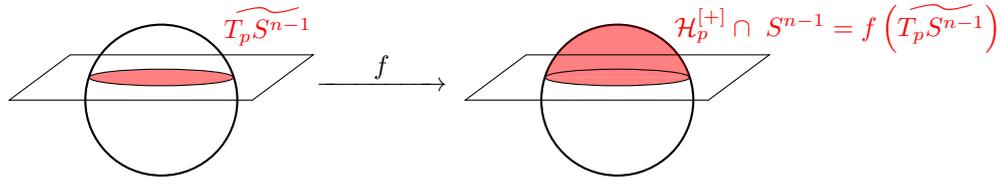
\[
\begin{tikzpicture}[scale=1]
\begin{scope}
\draw[thick] (3.6,2) circle (1);

\draw (2.4,2.6) -- (5.6,2.6) -- (4.8,2.0) -- (1.6,2.0) -- cycle;

\fill[red, opacity=0.5, blend mode=screen] (3.6,2.3) ellipse (0.95 and 0.11);
\draw (3.6,2.3) ellipse (0.95 and 0.11);

\node[above] at (6.5,2) {$\xrightarrow{\hspace{1.5cm}}$};
\node[above] at (6.5,2.15) {$f$};
\node[right, above] at (5,2.6) {\textcolor{red}{$\widetilde{T_p S^{n-1}}$}};
\end{scope}

\begin{scope}[xshift=6cm]
\draw[thick] (3.6,2) circle (1);

\fill[red, opacity=0.5]
  (2.65,2.3)
  arc[start angle=162.5, end angle=17.5, radius=1]  
  -- (4.55,2.3)  
  -- cycle;      

  \fill[red, opacity=0.5]
    (3.6,2.3) ++(-0.95,0)                
    arc[start angle=180, end angle=360,  
        x radius=0.95, y radius=0.11]
    -- (4.55,2.3) -- (2.65,2.3) -- cycle;
\draw (3.6,2.3) ellipse (0.95 and 0.11);
\draw (2.4,2.6) -- (5.6,2.6) -- (4.8,2.0) -- (1.6,2.0) -- cycle;

\node[red][right, above] at (6.5,2.5) {$\mathcal{H}_{p}^{[+]} \cap~S^{n-1}=f\left( \widetilde{T_p S^{n-1}} \right)$};
\end{scope}
\end{tikzpicture}
\]

\item
Choose $\mathcal{H}_p$ sufficiently close to $p$, and assume that the following two conditions are satisfied:
\begin{enumerate}
\item
The restriction of $f$ to $\widetilde{T_p S_p}$ maps onto $\mathcal{H}^{[+]}_{p} \cap S_p$; that is,
\[
f|_{\widetilde{T_p S_p}} : \widetilde{T_p S_p} \to \mathcal{H}^{[+]}_{p} \cap S_p
\]
is a homeomorphism.
This is equivalent to
\begin{align}\label{S_pは制限できる}
 f \left( \widetilde{T_p S^{n-1}} \right)  \cap S_p = f \left(  \widetilde{T_p S_{p}} \right),
\end{align}
as illustrated in the figure below.
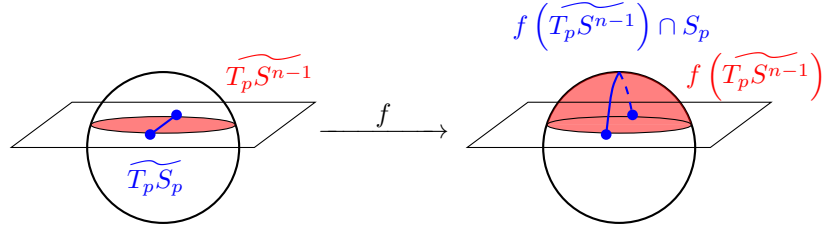
\[
\begin{tikzpicture}[scale=1]
\begin{scope}
\draw[thick] (3.6,2) circle (1);

\draw (2.4,2.6) -- (5.6,2.6) -- (4.8,2.0) -- (1.6,2.0) -- cycle;

\fill[red, opacity=0.5, blend mode=screen] (3.6,2.3) ellipse (0.95 and 0.11);
\draw (3.6,2.3) ellipse (0.95 and 0.11);

\node[above] at (6.5,2) {$\xrightarrow{\hspace{1.5cm}}$};
\node[above] at (6.5,2.15) {$f$};
\node[right, above] at (5,2.6) {\textcolor{red}{$\widetilde{T_p S^{n-1}}$}};

\begin{scope}
    \clip (3.6,2.3) ellipse (0.95 and 0.11); 

    \draw[blue, thick] (4,2.6) -- (3.2, 2.0);
\end{scope}
    \fill[blue] (3.429, 2.172) circle (2pt);
    \fill[blue] (3.776, 2.432) circle (2pt);
\node[blue]  at (3.5, 1.6) {$\widetilde{T_pS_p}$};
\end{scope}

\begin{scope}[xshift=6cm]
\draw[thick] (3.6,2) circle (1);

\fill[red, opacity=0.5]
  (2.65,2.3)
  arc[start angle=162.5, end angle=17.5, radius=1]  
  -- (4.55,2.3)  
  -- cycle;      

  \fill[red, opacity=0.5]
    (3.6,2.3) ++(-0.95,0)                
    arc[start angle=180, end angle=360,  
        x radius=0.95, y radius=0.11]
    -- (4.55,2.3) -- (2.65,2.3) -- cycle;
\draw (3.6,2.3) ellipse (0.95 and 0.11);
\draw (2.4,2.6) -- (5.6,2.6) -- (4.8,2.0) -- (1.6,2.0) -- cycle;

\node[right, above] at (5.4,2.5) {\textcolor{red}{$f\left( \widetilde{T_p S^{n-1}} \right)$}};
\end{scope}

\begin{scope}[xshift=6cm]
\draw[blue, thick] 
    (3.429, 2.172) 
    .. controls (3.5, 2.8) .. 
    (3.6, 3);
\draw[blue, thick, dashed] 
    (3.6, 3) 
    .. controls (3.7, 2.8) .. 
    (3.776, 2.432);
    \fill[blue] (3.429, 2.172) circle (2pt);
    \fill[blue] (3.776, 2.432) circle (2pt);

\node[blue]  at (3.5, 3.6) {$ f \left( \widetilde{T_p S^{n-1}} \right)  \cap S_p$};

\end{scope}
\end{tikzpicture}
\]

\item
$S_p$ is the minimal stratum on $f\bigl(\widetilde{T_p S^{n-1}}\bigr)$, that is,
\begin{align}\label{S_pが最小のstratum}
f \left(  \widetilde{T_p S^{n-1}} \right) \cap Q_i =f \left(  \widetilde{T_p S^{n-1}} \right) \cap S_p.
\end{align}
In particular, we have
\begin{align}\label{i-1で空}
f \left(  \widetilde{T_p S^{n-1}} \right) \cap Q_{i-1} = \emptyset.
\end{align}
\end{enumerate}
\end{itemize}

\subsubsection{Induced filtrations}\label{sss: 基本のfiltration}

In this subsubsection, 
we introduce filtrations on $\widetilde{T_p S^{n-1}}$, $\widetilde{T_p S_p}$ and $\widetilde{T_p N_p}$, which are defined in the previous subsubsection.
These filtrations are induced by the subspace filtrations of $B^n$ under the homeomorphism $f$, when these sets are viewed as subspaces of $B^n$ via $f$.  
We explain the details below.

\begin{itemize}

\item
Since $f \left( \widetilde{T_p S^{n-1}} \right) = \mathcal{H}^{[+]}_p \cap S^{n-1}$ is a subspace of $B^n$, it is equipped with the subspace filtration induced from $B^n$.
Through the homeomorphism
\[
f^{-1} : f(\widetilde{T_p S^{n-1}}) \to \widetilde{T_p S^{n-1}},
\]
we induce a filtration on $\widetilde{T_p S^{n-1}}$ as follows:
\begin{align}\label{接空間のfiltration}
\widetilde{T_p S^{n-1}} &= f^{-1}\left(  f \left( \widetilde{T_p S^{n-1}} \right) \cap Q_{n-1} \right)
\supset
 f^{-1}\left(  f \left( \widetilde{T_p S^{n-1}} \right) \cap Q_{n-2} \right)
\supset \cdots \notag
\\ 
&\cdots
\supset
 f^{-1}\left(  f \left( \widetilde{T_p S^{n-1}} \right) \cap Q_{k} \right)
\supset \cdots \notag
\\ 
&\cdots
\supset
 f^{-1}\left(  f \left( \widetilde{T_p S^{n-1}} \right) \cap Q_{i} \right)
\supset 
 f^{-1}\left(  f \left( \widetilde{T_p S^{n-1}} \right) \cap Q_{i-1} \right)
=
\emptyset.
\end{align}
Here,
$
 f^{-1}\left(  f \left( \widetilde{T_p S^{n-1}} \right) \cap Q_{i-1} \right)
=
\emptyset
$
follows from \eqref{i-1で空}.  
We denote this filtration by $\mathfrak{V}$.
That is, for $i \le k \le n-1$, we define
\[
V_k :=   f^{-1}\left(  f \left( \widetilde{T_p S^{n-1}} \right) \cap Q_{k} \right),
\]
so that
\begin{align}\label{filtration V}
\mathfrak{V}: \widetilde{T_p S^{n-1}}=V_{n-1} \supset V_{n-2}
\supset \cdots \supset
V_k
\supset \cdots \supset
V_i
\supset \emptyset.
\end{align}

\item
From \eqref{S_pは制限できる} and \eqref{S_pが最小のstratum}, we have
\begin{align*}
\widetilde{T_p S_p}
=
f^{-1}\left(
 f \left( \widetilde{T_p S^{n-1}} \right)  \cap S_p
\right)
=
f^{-1}\left(
 f \left( \widetilde{T_p S^{n-1}} \right)  \cap Q_i
\right)
=V_i.
\end{align*}
By \eqref{接空間のfiltration}, $V_i$ is equipped with the trivial filtration.
Hence, $\widetilde{T_p S_p}$, as a subspace of $\widetilde{T_p S^{n-1}}$, is equipped with the trivial filtration.  
Therefore, the nontrivial part of the filtration information of $\widetilde{T_p S^{n-1}} (\cong \widetilde{T_p S_p} \times \widetilde{T_p N_p})$
is entirely contained in $\widetilde{T_p N_p}$.

\item
Since $\widetilde{T_p N_p} \subset \widetilde{T_p S^{n-1}}$,
$\widetilde{T_p N_p}$ is equipped with a subspace filtration:
\begin{align}\label{N_pのfiltration}
\widetilde{T_p N_p} &= f^{-1}\left(  f \left( \widetilde{T_p N_p} \right) \cap Q_{n-1} \right)
\supset
 f^{-1}\left(  f \left( \widetilde{T_p N_p} \right) \cap Q_{n-2} \right)
\supset \cdots \notag
\\
&\cdots
\supset
 f^{-1}\left(  f \left( \widetilde{T_p N_p} \right) \cap Q_{k} \right)
\supset \cdots \supset
 f^{-1}\left(  f \left( \widetilde{T_p N_p} \right) \cap Q_{i} \right)
\supset \emptyset.
\end{align}

We denote this filtration by $\mathfrak{W}$.  
That is, for $i \le k \le n-1$, we set
\[
W_k :=   f^{-1}\left(  f \left( \widetilde{T_p N_p} \right) \cap Q_{k} \right)
\]
so that the filtration can be written as
\begin{align}\label{filtration W}
\mathfrak{W}: \widetilde{T_p N_p}=W_{n-1} \supset W_{n-2}
\supset \cdots \supset
W_k
\supset \cdots \supset
W_i
\supset \emptyset
\end{align}
where the subscript $k$ represents a formal dimension and does not necessarily reflect the dimension of the subspace.
Note that, using 
$f \left( \widetilde{T_p N_p} \right) \subset f \left( \widetilde{T_p S^{n-1}} \right)$,
\eqref{S_pが最小のstratum} and
\eqref{S_pは制限できる},
we obtain
\begin{align}\label{W_iは一点}
W_i
=
 f^{-1}\left(  f \left( \widetilde{T_p N_p} \right) \cap Q_{i} \right)
&=
 f^{-1}\left(  f \left( \widetilde{T_p N_p} \right) \cap f \left( \widetilde{T_p S^{n-1}} \right) \cap Q_{i} \right) \notag
\\
&=
 f^{-1}\left(  f \left( \widetilde{T_p N_p} \right) \cap f \left( \widetilde{T_p S^{n-1}} \right) \cap S_p \right) \notag
\\
&=
 f^{-1}\left(  f \left( \widetilde{T_p N_p} \right) \cap f \left( \widetilde{T_p S_p} \right)  \right)
\notag
\\
&=
\widetilde{T_p N_p} \cap \widetilde{T_p S_p} \notag
\\
&=
f^{-1}(p).
\end{align}
\end{itemize}

\subsubsection{Open neighborhood $U_p$}\label{sss: op nbh U_p}
We next choose an open neighborhood $U_p \subset B^n$ of $p$.  
If necessary, by taking $\mathcal{H}_p$ even closer to $p$, we may assume that $U_p$ satisfies the following conditions:
\begin{itemize}
\item
Let $\mathring{\widetilde{T_p S^{n-1}}}$ denote the interior of $\widetilde{T_p S^{n-1}}$. Then
\begin{align}\label{内部と一致}
U_p \cap S^{n-1} = f \left( \mathring{\widetilde{T_p S^{n-1}}} \right)
\end{align}

\item
$U_p \cong \cone\left( \widetilde{T_p S^{n-1}} \right)$, with $p$ corresponding to the cone vertex.
Moreover, restricting this homeomorphism gives
$U_p \cap S^{n-1} \cong \cone\left( \partial \left( \widetilde{T_p S^{n-1}} \right) \right).$

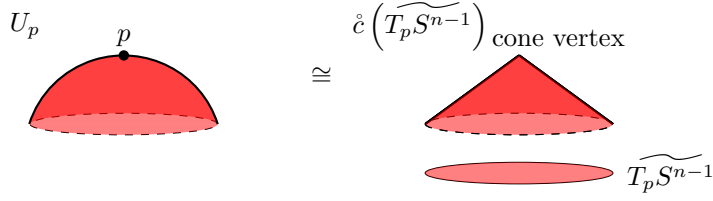
\[
\begin{tikzpicture}
\begin{scope}[scale=1.3]
\draw[dashed] (3.6,2.3) ellipse (0.95 and 0.11);
\fill[red, opacity=0.5]
  (2.65,2.3)
  arc[start angle=162.5, end angle=17.5, radius=1]  
  -- (4.55,2.3)  
  -- cycle;      
  \fill[red, opacity=0.5]
    (3.6,2.3) ++(-0.95,0)                
    arc[start angle=180, end angle=360,  
        x radius=0.95, y radius=0.11]
    -- (4.55,2.3) -- (2.65,2.3) -- cycle;
\fill[red, opacity=0.5] 
    (3.6-0.95,2.3) 
    arc[start angle=180,end angle=0,x radius=0.95,y radius=0.11] 
    -- (4.55,2.3) arc[start angle=17.5,end angle=162.5,radius=1] -- cycle; 
\draw[thick]
  (3.55,2.3)
  ++(1,0)
  arc[start angle=17.5, end angle=162.5, radius=1];

\node at (2.6,3.3) {$U_p$};

\fill (3.6,3) circle (1.5pt);
\node[above] at (3.6,3) {$p$};

\node at (5.6,2.79) {$\cong$};
\end{scope}

\begin{scope}[scale=1.3, xshift=4cm]
\coordinate (Apex) at (3.6,3); 
\coordinate (LeftBase) at (2.65,2.3);
\coordinate (RightBase) at (4.55,2.3);
\draw[thick] (Apex) -- (LeftBase);
\draw[thick] (Apex) -- (RightBase);
\draw[thick] (Apex) -- (LeftBase);
\draw[thick] (Apex) -- (RightBase);
\draw[dashed] (3.6,2.3) ++(-0.95,0)
    arc[start angle=180, end angle=360, x radius=0.95, y radius=0.11];
\fill[red, opacity=0.5] 
  (Apex) -- (RightBase) 
  arc[start angle=0, end angle=180, x radius=0.95, y radius=0.11] 
  -- (LeftBase) -- cycle;
\draw[thick] (Apex) -- (LeftBase);
\draw[thick] (Apex) -- (RightBase);
\draw[dashed] (RightBase) 
  arc[start angle=0, end angle=180, x radius=0.95, y radius=0.11];
\fill[red, opacity=0.5] 
  (Apex) -- (LeftBase) 
  arc[start angle=180, end angle=360, x radius=0.95, y radius=0.11] 
  -- (RightBase) -- cycle;
\draw[dashed] (LeftBase)
  arc[start angle=180, end angle=360, x radius=0.95, y radius=0.11];

\node at (2.6,3.3) {$\cone\left(  \widetilde{T_p S^{n-1}} \right)$};

\node[above] at (4,3) {cone vertex};

\draw (3.6,1.8) ellipse (0.95 and 0.11);
\fill[red, opacity=0.5] (3.6,1.8) ellipse (0.95 and 0.11);
\node[right] at (4.6,1.8) {$\widetilde{T_p S^{n-1}}$};
\end{scope}
\end{tikzpicture}
\]

\end{itemize}
The interior $\mathring{\widetilde{T_p S^{n-1}}}$ is equipped with the subspace filtration induced from $\widetilde{T_p S^{n-1}}$:
\begin{align}\label{filtration Vcirc}
\mathring{\mathfrak{V}}: \mathring{\widetilde{T_p S^{n-1}}}=\mathring{V}_{n-1} \supset \mathring{V}_{n-2}
\supset \cdots \supset
\mathring{V}_k
\supset \cdots \supset
\mathring{V}_i
\supset \emptyset,
\end{align}
where
\[
\mathring{V}_k :=   f^{-1}\left(  f \left( \mathring{\widetilde{T_p S^{n-1}}} \right) \cap Q_{k}. \right)
\]
The open neighborhood $U_p$ is equipped with the subspace filtration induced from $B^n$:
\begin{align}\label{subspace filtration 14.4}
U_p= U_p \cap Q_n \supset U_p \cap Q_{n-1} \supset U_p \cap Q_{n-2} 
\supset \cdots \supset 
U_p \cap Q_{k} 
\supset \cdots \supset 
U_p \cap Q_{i} \supset U_p \cap Q_{i-1}=\emptyset.
\end{align}
Since $U_p \cap S^{n-1} = f \left( \mathring{\widetilde{T_p S^{n-1}}} \right)$, we have
\[
U_p \cap Q_k =f\left(\mathring{V}_k \right), \quad i \le k \le n-1.
\]

\subsubsection{Filtrations on $\cone\left( \widetilde{T_p S^{n-1}} \right)$ and $\cone\left( \widetilde{T_p N_p} \right)$}\label{図と証明の説明}

The proof of Proposition~\ref{ball is a topological stratified pseudomanifold} will be given in the next subsubsection.  
The goal of the proof is to verify Definition~\ref{def of pseudomanifold}-3.  
Here, we aim to show that there exists an $(n-i-1)$-dimensional topological stratified pseudomanifold $L_p$ such that
\[
U_p \cap Q_{i+j+1} \cong \mathbb{R}^i \times \cone((L_p)_j), \quad -1 \le j \le n-i-1.
\]
In this context, we require that $\cone((L_p))_{j+1} = \cone((L_p)_j)$.  
The topological stratified pseudomanifold $L_p$ is constructed from $\widetilde{T_p N_p}$ by equipping it with a filtration different from \eqref{N_pのfiltration}.

For $U_p \cong \cone\left( \widetilde{T_p S^{n-1}} \right)$, since 
$\widetilde{T_p S^{n-1}} \cong B^{n-1}$, 
$\widetilde{T_p N_p} \cong B^{n-i-1}$ and 
$U_p \cap S_p \cong \mathbb{R}^i$, we obtain
\begin{align}\label{U_p の分解}
U_p 
\cong
 \cone\left( \widetilde{T_p S^{n-1}} \right) 
\cong
 (U_p \cap S_p) \times \cone\left( \widetilde{T_p N_p} \right) 
\end{align}
Here, we define the filtrations on 
$\cone\left( \widetilde{T_p S^{n-1}} \right)$ and $\cone\left( \widetilde{T_p N_p} \right)$ 
as those induced from $U_p$.  
The aim of this subsubsection is to explain these filtrations.

\paragraph{Filtration on $\cone\left( \widetilde{T_p S^{n-1}} \right)$.}

From the choice of $U_p$, we have
$U_p \cap S^{n-1} \cong \cone\left( \partial \left( \widetilde{T_p S^{n-1}} \right) \right)$.
Moreover, since 
$\mathring{\widetilde{T_p S^{n-1}}} \cong \cone\left( \partial \left( \widetilde{T_p S^{n-1}} \right) \right)$, 
this homeomorphism induces a filtration on 
$\cone\left( \partial \left( \widetilde{T_p S^{n-1}} \right) \right)$ 
from $\mathring{\mathfrak{V}}$ \eqref{filtration Vcirc}:
\begin{align}\label{cone でない cone}
\cone\left( \partial \left( \widetilde{T_p S^{n-1}} \right) \right)
=
\cone\left(  \partial \left( V_{n-1} \right) \right)
\supset
\cone\left(  \partial \left( V_{n-2} \right) \right)
\supset \cdots \supset
\cone\left(  \partial \left( V_{k} \right) \right)
\supset \cdots \supset
\cone\left(  \partial \left( V_{i} \right) \right)
\supset \emptyset,
\end{align}
where
\[
\partial \left( V_{k} \right):=f^{-1}\left(  f \left( \partial \left( \widetilde{T_p S^{n-1}} \right) \right) \cap Q_{k} \right),
\quad i \le k \le n-1.
\]
Since $\mathring{\widetilde{T_p S^{n-1}}} = f^{-1}(U_p \cap S^{n-1})$ and $f^{-1}$ is a homeomorphism, the filtration $\mathring{\mathfrak{V}}$ \eqref{filtration Vcirc} can be regarded as induced from $U_p$. 
Therefore, the filtration \eqref{cone でない cone} is also induced from $U_p$. 
Consequently, via the homeomorphism $U_p \cong \cone\left( \widetilde{T_p S^{n-1}} \right)$, we obtain a filtration on $\cone\left( \widetilde{T_p S^{n-1}} \right)$:
\begin{align}\label{induced cone でない}
\cone\left(  \widetilde{T_p S^{n-1}} \right)
\supset
\cone\left(  \partial \left( V_{n-1} \right) \right)
\supset
\cone\left(  \partial \left( V_{n-2} \right) \right)
\supset \cdots \supset
\cone\left(  \partial \left( V_{k} \right) \right)
\supset \cdots \supset
\cone\left(  \partial \left( V_{i} \right) \right)
\supset \emptyset.
\end{align}
With this definition of the filtration, we also have
\[
U_p \cap Q_k
\cong
\begin{cases}
\cone\left(  \partial \left( V_{k} \right) \right), \quad &\text{for } i \le k \le n-1 \\
\cone\left(  \widetilde{T_p S^{n-1}} \right), \quad &\text{for } k=n.
\end{cases}
\]
The correspondence of the filtrations is illustrated in Figure~\ref{fig: filtration 1}.  
The filtration induced on $\cone\left( \widetilde{T_p S^{n-1}} \right)$ can be viewed as being covered from above by $\mathring{\widetilde{T_p S^{n-1}}}$.

\begin{figure}[htbp]
    \centering
\begin{tikzpicture}[scale=1.3]

\begin{scope}
\draw[dashed] (3.6,2.3) ellipse (0.95 and 0.11);
\fill[red, opacity=0.5]
  (2.65,2.3)
  arc[start angle=162.5, end angle=17.5, radius=1]  
  -- (4.55,2.3)  
  -- cycle;      
  \fill[red, opacity=0.5]
    (3.6,2.3) ++(-0.95,0)                
    arc[start angle=180, end angle=360,  
        x radius=0.95, y radius=0.11]
    -- (4.55,2.3) -- (2.65,2.3) -- cycle;
\fill[red, opacity=0.5] 
    (3.6-0.95,2.3) 
    arc[start angle=180,end angle=0,x radius=0.95,y radius=0.11] 
    -- (4.55,2.3) arc[start angle=17.5,end angle=162.5,radius=1] -- cycle; 
\draw[thick]
  (3.55,2.3)
  ++(1,0)
  arc[start angle=17.5, end angle=162.5, radius=1];
\end{scope}

\begin{scope}[xshift=4cm]
\coordinate (Apex) at (3.6,3); 
\coordinate (LeftBase) at (2.65,2.3);
\coordinate (RightBase) at (4.55,2.3);
\draw[thick] (Apex) -- (LeftBase);
\draw[thick] (Apex) -- (RightBase);
\draw[thick] (Apex) -- (LeftBase);
\draw[thick] (Apex) -- (RightBase);
\draw[dashed] (3.6,2.3) ++(-0.95,0)
    arc[start angle=180, end angle=360, x radius=0.95, y radius=0.11];
\fill[red, opacity=0.5] 
  (Apex) -- (RightBase) 
  arc[start angle=0, end angle=180, x radius=0.95, y radius=0.11] 
  -- (LeftBase) -- cycle;
\draw[thick] (Apex) -- (LeftBase);
\draw[thick] (Apex) -- (RightBase);
\draw[dashed] (RightBase) 
  arc[start angle=0, end angle=180, x radius=0.95, y radius=0.11];
\fill[red, opacity=0.5] 
  (Apex) -- (LeftBase) 
  arc[start angle=180, end angle=360, x radius=0.95, y radius=0.11] 
  -- (RightBase) -- cycle;
\draw[dashed] (LeftBase)
  arc[start angle=180, end angle=360, x radius=0.95, y radius=0.11];
\end{scope}

\begin{scope}[xshift=4cm, yshift=-3.5cm]
\node [above] at (3.8, 2.5) {$\left(  \widetilde{T_pS^{n-1}} , \mathfrak{V} \, \eqref{filtration V}\right)$};
\end{scope}

\node at (3,3.3) {$\left( U_p , \eqref{subspace filtration 14.4} \right)$};

\node at (7.5,3.3) {$\left(
\cone\left(  \widetilde{T_pS^{n-1}} \right) , \eqref{induced cone でない}
\right)$};

\node at (3.6,1.8) {\rotatebox{90}{$\subset$}};
\node[red] at (3.2,1.3) {$\left( U_p \cap S^{n-1},
f(\mathring{\mathfrak{V})} \right)$};

\node at (3.6,0.8) {\rotatebox{90}{$=$}};
\node at (3,0.3) {$\left( f\left( \mathring{\widetilde{T_pS^{n-1}}} \right) ,
f(\mathring{\mathfrak{V})} \right)$};

\node at (3.6,-0.2) {\rotatebox{90}{$\supset$}};
\node [above] at (3,-1) {$\left( f\left( \widetilde{T_pS^{n-1}} \right) , f(\mathfrak{V}) \right)$};

\node at (7.5,1.8) {\rotatebox{90}{$\subset$}};
\node[red] at (7.8,1.3) {$\left(  \cone\left( \partial \left( \widetilde{T_pS^{n-1}} \right) \right) 
, \eqref{cone でない cone} \right)$};

\node[red] at (7.5,0.8) {\rotatebox{90}{$\cong$}};
\node[red] at (8,0.3) {$\left(  \mathring{\widetilde{T_pS^{n-1}}} , \mathring{\mathfrak{V}}  \, \eqref{filtration Vcirc} \right)$};

\node at (7.5,-0.2) {\rotatebox{90}{$\supset$}};

\node at (5.6,2.7) {$\xrightarrow{\hspace{2.5cm}}$};
\node at (5.6,2.79) {$\cong$};

\node at (5.2,1.3) {$\xrightarrow{\hspace{1.8cm}}$};
\node at (5.2,1.39) {$\cong$};

\node [above] at (5.4,0.1) {$\xrightarrow{\hspace{2.5cm}}$};
\node [above] at (5.4,0.21) {$f^{-1}$};
\node [above] at (4.6,0.21) {$\cong$};

\node [above] at (5.4,-0.9) {$\xrightarrow{\hspace{2.5cm}}$};
\node [above] at (5.4,-0.81) {$f^{-1}$};
\node [above] at (4.6,-0.81) {$\cong$};

\end{tikzpicture}
    \caption{Filtration of $\cone \left( \widetilde{T_p S^{n-1}} \right)$}
    \label{fig: filtration 1}
\end{figure}

\paragraph{Filtration on $\cone\left( \widetilde{ T_p N_p} \right)$.}
Consider the decomposition
\[
\mathring{ \widetilde{ T_p S^{n-1}}}
\cong
\mathring{ \widetilde{ T_p S_p}}
\times
\mathring{ \widetilde{ T_p N_p}}.
\]
Here,
$\mathring{ \widetilde{ T_p S_p}}$
is equipped with the trivial filtration, while
$\mathring{ \widetilde{ T_p N_p}}$ is endowed with the following filtration:
\begin{align}\label{filtration Wcirc}
\mathring{\mathfrak{W}}: \mathring{\widetilde{T_p N_p}}=\mathring{W}_{n-1} \supset \mathring{W}_{n-2}
\supset \cdots \supset
\mathring{W}_k
\supset \cdots \supset
\mathring{W}_i
\supset \emptyset,
\end{align}
where 
\[
\mathring{W}_k :=   f^{-1}\left(  f \left( \mathring{\widetilde{T_p N_p}} \right) \cap Q_{k} \right).
\]
Since $\mathring{\widetilde{T_p N_p}} \cong \cone\left( \partial \left( \widetilde{T_p N_p} \right) \right)$,
in analogy with the filtration \eqref{cone でない cone}, we obtain
\begin{align}\label{cone になる cone}
\cone\left( \partial \left( \widetilde{T_p N_p} \right) \right)
=
\cone\left(  \partial \left( W_{n-1} \right) \right)
\supset
\cone\left(  \partial \left( W_{n-2} \right) \right)
\supset \cdots \supset
\cone\left(  \partial \left( W_{k} \right) \right)
\supset \cdots \supset
\cone\left(  \partial \left( W_{i} \right) \right)
\supset \emptyset,
\end{align}
where
\[
\partial \left( W_{k} \right):=f^{-1}\left(  f \left( \partial \left( \widetilde{T_p N_p} \right) \right) \cap Q_{k} \right),
\quad i \le k \le n-1.
\]
Moreover, by \eqref{S_pが最小のstratum} and \eqref{内部と一致}, we have
\begin{align*}
\mathring{\widetilde{T_p S_p}}
 &\cong 
f\left(  \mathring{\widetilde{T_p S_p}} \right) \\
&=
f\left(  \mathring{\widetilde{T_p S^{n-1}}} \right) \cap S_p \\
&=(U_p \cap S^{n-1}) \cap S_p
=U_p \cap  S_p,
\end{align*}
where the filtration is trivial.
Therefore, the filtration on
$( U_p \cap  S_p) \times \cone\left( \widetilde{T_p N_p} \right)$
induced via the homeomorphism
\[
\cone\left( \widetilde{T_p S^{n-1}} \right) 
\cong
( U_p \cap  S_p) \times \cone\left( \widetilde{T_p N_p} \right)
\]
is given by
\begin{align}\label{right filtration}
( U_p \cap  S_p) \times \cone\left( \widetilde{T_p N_p} \right) 
&\supset
( U_p \cap  S_p) \times \cone\left(  \partial \left( W_{n-1} \right) \right)
\supset
( U_p \cap  S_p) \times \cone\left(  \partial \left( W_{n-2} \right) \right)
\supset \cdots  \notag
\\
&\cdots \supset
( U_p \cap  S_p) \times \cone\left(  \partial \left( W_{k} \right) \right)
\supset \cdots  \notag
\\
&\cdots \supset
( U_p \cap  S_p) \times \cone\left(  \partial \left( W_{i+1} \right) \right)
\supset
( U_p \cap  S_p) \times \cone\left(  \partial \left( W_{i} \right) \right)
\supset \emptyset.
\end{align}
This filtration arises from the decomposition
$
\mathring{ \widetilde{ T_p S^{n-1}}}
\cong
\mathring{ \widetilde{ T_p S_p}}
\times
\mathring{ \widetilde{ T_p N_p}}.
$
Since the filtration $\mathring{\mathfrak{V}}$ \eqref{filtration Vcirc} on
$\mathring{ \widetilde{ T_p S^{n-1}}}$
is induced from $U_p$, the filtration \eqref{right filtration} is also induced from $U_p$.
Therefore, by the definition of the filtration, we have
\begin{align}\label{丁寧な準備}
U_p \cap Q_k
&\cong
\begin{cases}
\cone\left(  \partial \left( V_{k} \right) \right), \quad &\text{for } i \le k \le n-1 \\
\cone\left(  \widetilde{T_p S^{n-1}} \right), \quad &\text{for } k=n
\end{cases}
\notag
\\
&\cong
\begin{cases}
( U_p \cap  S_p) \times \cone\left(  \partial \left( W_{k} \right) \right), \quad &\text{for } i \le k \le n-1 \\
( U_p \cap  S_p) \times \cone\left(  \widetilde{T_p N_p} \right), \quad &\text{for } k=n.
\end{cases}
\end{align}
The correspondence of the filtrations is illustrated in Figure~\ref{fig: filtration}.

\begin{figure}[htbp]
    \centering
\begin{tikzpicture}[scale=1.3]

\begin{scope}
\draw[dashed] (3.6,2.3) ellipse (0.95 and 0.11);
\fill[red, opacity=0.5]
  (2.65,2.3)
  arc[start angle=162.5, end angle=17.5, radius=1]  
  -- (4.55,2.3)  
  -- cycle;      
  \fill[red, opacity=0.5]
    (3.6,2.3) ++(-0.95,0)                
    arc[start angle=180, end angle=360,  
        x radius=0.95, y radius=0.11]
    -- (4.55,2.3) -- (2.65,2.3) -- cycle;
\fill[red, opacity=0.5] 
    (3.6-0.95,2.3) 
    arc[start angle=180,end angle=0,x radius=0.95,y radius=0.11] 
    -- (4.55,2.3) arc[start angle=17.5,end angle=162.5,radius=1] -- cycle; 
\draw[thick]
  (3.55,2.3)
  ++(1,0)
  arc[start angle=17.5, end angle=162.5, radius=1];

\draw[blue, thick] 
    (3.429, 2.172) 
    .. controls (3.5, 2.8) .. 
    (3.6, 3);
\draw[blue, thick, dashed] 
    (3.6, 3) 
    .. controls (3.7, 2.8) .. 
    (3.776, 2.432);
    \draw[blue] (3.429, 2.172) circle (2pt);
    \draw[blue] (3.776, 2.432) circle (2pt);
\end{scope}

\begin{scope}[xshift=3cm]
\coordinate (Apex) at (3.6,3); 
\coordinate (LeftBase) at (2.65,2.3);
\coordinate (RightBase) at (4.55,2.3);
\draw[thick] (Apex) -- (LeftBase);
\draw[thick] (Apex) -- (RightBase);
\draw[thick] (Apex) -- (LeftBase);
\draw[thick] (Apex) -- (RightBase);
\draw[dashed] (3.6,2.3) ++(-0.95,0)
    arc[start angle=180, end angle=360, x radius=0.95, y radius=0.11];
\fill[red, opacity=0.5] 
  (Apex) -- (RightBase) 
  arc[start angle=0, end angle=180, x radius=0.95, y radius=0.11] 
  -- (LeftBase) -- cycle;
\draw[thick] (Apex) -- (LeftBase);
\draw[thick] (Apex) -- (RightBase);
\draw[dashed] (RightBase) 
  arc[start angle=0, end angle=180, x radius=0.95, y radius=0.11];
\fill[red, opacity=0.5] 
  (Apex) -- (LeftBase) 
  arc[start angle=180, end angle=360, x radius=0.95, y radius=0.11] 
  -- (RightBase) -- cycle;
\draw[dashed] (LeftBase)
  arc[start angle=180, end angle=360, x radius=0.95, y radius=0.11];
\end{scope}

\begin{scope}[xshift=6.3cm]
\draw[blue, thick] 
    (3.429, 2.172) 
    .. controls (3.5, 2.8) .. 
    (3.6, 3);
\draw[blue, thick] 
    (3.6, 3) 
    .. controls (3.7, 2.8) .. 
    (3.776, 2.432);
    \draw[blue] (3.429, 2.172) circle (2pt);
    \draw[blue] (3.776, 2.432) circle (2pt);
\end{scope}
\begin{scope}[xshift=9cm]
\coordinate (Apex) at (3.6,3); 
\coordinate (LeftBase) at (2.65,2.3);
\coordinate (RightBase) at (4.55,2.3);
\fill[gray!30] (4.55,2.3) -- (2.65,2.3) -- (3.6,3);
\draw[red, thick] (Apex) -- (LeftBase);
\draw[red,thick] (Apex) -- (RightBase);
\draw[dashed, thick] (4.55,2.3) -- (2.65,2.3);
    \draw[red] (4.55,2.3) circle (2pt);
    \draw[red] (2.65,2.3) circle (2pt);
\end{scope}

\begin{scope}[xshift=3cm, yshift=-3.5cm]
\draw[thick] (3.6,2.3) ellipse (0.95 and 0.11);
\fill[pink] (3.6,2.3) ellipse (0.95 and 0.11);
\begin{scope}
    \clip (3.6,2.3) ellipse (0.95 and 0.11); 
    \draw[red, thick] (5.2,2.3) -- (2.0,2.3);
    \draw[blue, thick] (4,2.6) -- (3.2, 2.0);
\end{scope}
    \fill[red] (4.55,2.3) circle (2pt);
    \fill[red] (2.65,2.3) circle (2pt);
    \fill[blue] (3.429, 2.172) circle (2pt);
    \fill[blue] (3.776, 2.432) circle (2pt);

\node[blue]  at (3.5, 1.8) { $\left(\widetilde{T_pS_p}, \text{trivial} \right)$};

\node[red]  at (1.5,2.3) {$\left(  \widetilde{T_pN_p} , \mathfrak{W} \,\eqref{filtration W}  \right)$};

\node [above] at (3.6, 2.5) {$\left( \widetilde{T_pS^{n-1}} , \mathfrak{V} \right)$};
\end{scope}

\begin{scope}[xshift=9cm, yshift=-3.5cm]
\begin{scope}
    \clip (3.6,2.3) ellipse (0.95 and 0.11); 
    \draw[red, thick] (5.2,2.3) -- (2.0,2.3);
\end{scope}
    \fill[red] (4.55,2.3) circle (2pt);
    \fill[red] (2.65,2.3) circle (2pt);
\end{scope}

\begin{scope}[xshift=6.3cm, yshift=-3.5cm]
\begin{scope}
    \clip (3.6,2.3) ellipse (0.95 and 0.11); 
    \draw[blue, thick] (4,2.6) -- (3.2, 2.0);
\end{scope}
    \fill[blue] (3.429, 2.172) circle (2pt);
    \fill[blue] (3.776, 2.432) circle (2pt);
\end{scope}

\node[blue] at (3.8,3.3) {$U_p \cap S_p$};
\node at (2.6,3.3) {$U_p$};

\node at (5,2.7) {$\xrightarrow{\hspace{1cm}}$};
\node at (5,2.79) {$\cong$};

\node at (8.5,2.7) {$\xrightarrow{\hspace{2cm}}$};
\node at (8.5,2.79) {$\cong$};

\node at (11.5,3.3) 
{$\left(
(U_p \cap S_p) 
\times 
\cone\left(  \widetilde{T_pN_p} \right)
,
\eqref{right filtration}
\right)$};
\node at (11,2.7) {$\times$};

\node[blue] at (9.9,1.8) {\rotatebox{90}{$=$}};
\node[blue] at (9.9,1.3) {$\left(  U_p \cap S_p, \text{trivial}  \right)$};
\node at (11,1.3) {$\times$};
\node at (12.5,1.8) {\rotatebox{90}{$\subset$}};
\node[red] at (12.6,1.3) 
{$\left(   \cone\left( \partial \left( \widetilde{T_pN_p} \right) \right) , \eqref{cone になる cone}    \right)$};

\node at (11,0.3) {$\times$};
\node[blue] at (9.9,0.8) {\rotatebox{90}{$\cong$}};
\node[blue] at (9.9,0.3) {$\left(  \mathring{\widetilde{T_p S_p}}, \text{trivial}  \right)$};
\node[red] at (12.5,0.8) {\rotatebox{90}{$\cong$}};
\node[red]  at (12.5,0.3) {$\left(  \mathring{\widetilde{T_pN_p}} , \mathring{\mathfrak{W}} \, \eqref{filtration Wcirc}   \right)$};

\node at (11,-0.7) {$\times$};
\node[blue] at (9.9,-0.7) {$\left( \widetilde{T_p S_p}, \text{trivial}  \right)$};
\node[blue] at (9.9,-0.2) {\rotatebox{90}{$\supset$}};
\node[red] at (12.5,-0.2) {\rotatebox{90}{$\supset$}};
\node[red]  at (12.5,-0.7) {$\left(  \widetilde{T_pN_p} , \mathfrak{W} \, \eqref{filtration W}   \right)$};

\node at (6.5,3.3) {$\left(
\cone\left(  \widetilde{T_pS^{n-1}} \right) , \eqref{induced cone でない}
\right)$};

\node at (6.5,1.8) {\rotatebox{90}{$\subset$}};
\node at (6.5,1.3) {$\left(  \cone\left( \partial \left( \widetilde{T_pS^{n-1}} \right) \right) 
, \eqref{cone でない cone} \right)$};

\node at (6.5,0.8) {\rotatebox{90}{$\cong$}};
\node at (6.5,0.3) {$\left(  \mathring{\widetilde{T_pS^{n-1}}} , \mathring{\mathfrak{V}}  \right)$};

\node at (6.5,-0.2) {\rotatebox{90}{$\supset$}};

\node at (8.2,0.3) {$\xrightarrow{\hspace{1.7cm}}$};
\node at (7.6,0.39) {$\cong$};
\node at (8.3,0.39) {decomp.};

\node at (8.2,-0.7) {$\xrightarrow{\hspace{1.7cm}}$};
\node at (7.6,-0.61) {$\cong$};
\node at (8.3,-0.61) {decomp.};

\end{tikzpicture}
    \caption{Filtration of $\cone\left(\widetilde{T_p N_p}\right)$}
    \label{fig: filtration}
\end{figure}

\subsubsection{Proof of Proposition~\ref{ball is a topological stratified pseudomanifold}}\label{ballがpseudomanifoldであること}

\begin{proof}
If $n = 0$, then $B^0$ consists of a single point and is hence a topological stratified pseudomanifold.
Assume $n > 0$. We prove the statement by induction on $n$.
For $n = 1$, the ball $B^1 = [-1,1]$ has the following filtration:
\[
[-1,1] \supset \{ \pm 1 \} \supset \emptyset.
\]
Hence, it is a topological stratified pseudomanifold.
We now assume that the proposition holds for all dimensions less than $n$.
Conditions Definition \ref{def of pseudomanifold}-1 and Definition \ref{def of pseudomanifold}-2 are straightforward to verify.

For each $p \in Q_i \setminus Q_{i-1}$ (where $0 \leq i \leq n$), we next construct a triple $(U_p, L_p, \varphi_p)$ satisfying Definition \ref{def of pseudomanifold}-3.
If $i = n$, then we have $p \in \mathring{B^n}$, where $\mathring{B^n}$ denotes the interior of $B^n$.
Take an open neighborhood $U_p$ of $p$ in $B^n$ to be an open ball with sufficiently small radius. Then $U_p$ satisfies the condition Definition \ref{def of pseudomanifold}-3.

We now consider the case $i \neq n$.
In the following, we use the notation introduced in Subsubsection~\ref{接空間からの記号} -- Subsubsection~\ref{図と証明の説明}.
From \eqref{U_p の分解}, we have
\[
U_p \cong (U_p \cap S_p) \times \cone\left( \widetilde{T_p N_p}  \right),
\]
where $U_p$ is equipped with the subspace filtration, and
$(U_p \cap S_p) \times \cone\left( \widetilde{T_p N_p}  \right)$ 
is endowed with the filtration given in \eqref{right filtration}.
By
\eqref{丁寧な準備}, it follows that
\begin{align}\label{上の同型}
U_p \cap Q_k
\cong
\begin{cases}
( U_p \cap  S_p) \times \cone\left(  \partial \left( W_{k} \right) \right), \quad &\text{for } i \le k \le n-1 \\
( U_p \cap  S_p) \times \cone\left(  \widetilde{T_p N_p} \right), \quad &\text{for } k=n.
\end{cases}
\end{align}
The filtration induced on $\cone\left( \widetilde{T_p N_p}  \right)$ is
\begin{align}\label{cone そのもの}
\cone\left( \widetilde{T_p N_p}  \right)
\supset
\cone\left(  \partial \left( W_{n-1} \right) \right)
\supset
\cone\left(  \partial \left( W_{n-2} \right) \right)
\supset \cdots \supset
\cone\left(  \partial \left( W_{k} \right) \right)
\supset \cdots \supset
\cone\left(  \partial \left( W_{i} \right) \right)
\supset \emptyset.
\end{align}
By
\eqref{W_iは一点}, $W_i=f^{-1}(p)$.
Since 
$\partial (W_i) \subset \partial \left(\widetilde{T_p S^{n-1}} \right)$
and
$f^{-1}(p) \notin \partial \left( \widetilde{T_p S^{n-1}} \right)$,
we have
\[
\partial (W_i)=\emptyset.
\]
Hence, $\cone\left( \partial (W_i) \right)$ consists only of the cone vertex.
Therefore, the filtration \eqref{cone そのもの} on the open cone $\cone \left( \widetilde{T_p N_p} \right)$ is induced from the following filtration of $\widetilde{T_p N_p}$ in the usual sense of an open cone filtration (see \eqref{通常のcone filtration}):
\begin{align*}
\mathfrak{N}:
 \widetilde{T_p N_p}  
\supset
  \partial \left( W_{n-1} \right) 
\supset
  \partial \left( W_{n-2} \right) 
\supset \cdots \supset
  \partial \left( W_{k} \right) 
\supset \cdots \supset
  \partial \left( W_{i+1} \right)
\supset \emptyset.
\end{align*}
For $-1 \le j \le n-i-1$, define
\[
\left(\widetilde{T_p N_p}\right)_j :=
\begin{cases}
\partial \left( W_{i+j+1} \right), \quad  &\text{for }  -1 \le j \le n-i-2, \\
 \widetilde{T_p N_p}, \quad  &\text{for } j=n-i-1.
\end{cases}
\]
Then the homeomorphism \eqref{上の同型} can be rewritten as
\[
U_p \cap Q_{i+j+1} \cong
 (U_p \cap S_p)\times 
\cone \left( \left(\widetilde{T_p N_p}\right)_j \right), \quad -1 \le j \le n-i-1.
\]
Since
$U_p \cap S_p \cong \mathbb{R}^i$ and $\widetilde{T_p N_p} \cong B^{n-i-1}$ is compact, Definition~\ref{def of pseudomanifold}-3 is verified.

We finally show that $\left( \widetilde{T_p N_p}, \mathfrak{N} \right)$ is an $(n-i-1)$-dimensional topological stratified pseudomanifold.
It suffices to show that the filtration $\mathfrak{N}$ is induced by a complete fan; then, by the inductive hypothesis, it follows that $\left( \widetilde{T_p N_p}, \mathfrak{N} \right)$ is an $(n-i-1)$-dimensional topological stratified pseudomanifold.
A complete fan on $\mathcal{H}_p N_p$ is defined as follows:
Let $\sigma_p$ be the maximal cone in $\Sigma$ containing $p$.
Consider the fan $\Sigma_{p}$ on $\mathcal{H}_p$ with origin $t(p)$, generated by the rays from $t(p)$ toward the intersections of rays of $\Sigma$ with $\mathcal{H}_p$.
Then $\Sigma_{p}$ is an $(n-1)$-dimensional complete fan.
Since $\mathcal{H}_p N_p$ passes through the origin $t(p)$, restricting $\Sigma_{p}$ to $\mathcal{H}_p N_p$ gives an $(n-i-1)$-dimensional complete fan.
This induces a spherical complex on $\partial\left( \widetilde{T_p N_p} \right)$.
Because $\partial\left( \widetilde{T_p N_p} \right) \subset S^{n-1}$, the obtained spherical complex coincides with the restriction of $C_{\Sigma}$.
Hence, the filtration $\mathfrak{N}$ is induced by a complete fan.
\end{proof}

\subsection{Compact locally standard $T$-manifold}

In this subsection, we prove the following proposition.

\begin{prop}\label{prop: locally standard T-manifold}
A compact locally standard $T$-manifold $M^{2m}$ is a locally standard $T$-pseudomanifold.
\end{prop}

Recall that a {\it locally standard $T$-manifold} is defined as follows:

\begin{definition}[locally standard $T$-manifold {\cite[Chapter 7.4]{BP12}}]\label{def of locally standard T-manifold}
A locally standard $T$-manifold is a smooth, connected, closed, and orientable $2m$-dimensional manifold $M^{2m}$ equipped with a locally standard action of an $m$-dimensional torus $T^m$.
\end{definition}

We assume that all locally standard $T$-manifolds are compact. Since every compact manifold is second-countable, the second countability axiom automatically holds.
We begin by proving the following proposition:

\begin{prop}\label{manifold with corners is a pseudomanifold}
Let $M^{2m}$ be a compact locally standard $T$-manifold, and let $Q^m = M^{2m}/T^m$ be its orbit space. Then $Q^m$ is a second-countable, compact, topological stratified pseudomanifold.
\end{prop}

\begin{proof}
Since $M^{2m}$ is compact, second-countable, and connected, the orbit space $Q^m$ is also compact, second-countable, and connected.
By \cite[Proposition 7.4.13]{BP12}, the orbit space $Q^m$ has the structure of an $m$-dimensional manifold with corners, and furthermore has the structure of a manifold with faces (for manifolds with corners and manifolds with faces, see \cite{Joy12} and \cite{BP12}).
We define a filtration of $Q^m$ by the dimensions of its faces:
\begin{align}\label{filtration: manifold with faces}
\mathfrak{Q}: Q^m=Q_m \supsetneq Q_{m-1} \supset \cdots \supset Q_0 \supset \emptyset,
\end{align}
where each $Q_i$ is the union of $i$-dimensional faces ($0 \leq i \leq m$).
We now verify that this filtration gives $Q^m$ the structure of a topological stratified pseudomanifold (see Definition~\ref{def of pseudomanifold}).

If $m=0$, then $Q^m$ is a single point and hence is a topological stratified pseudomanifold.
Suppose that $m > 0$.
The interior of $Q^m$, i.e., the $m$-dimensional stratum, is an $m$-dimensional manifold. 
By \cite[Proposition 2.7]{Joy12}, each facet of $Q^m$ is an $(m-1)$-dimensional manifold with corners.
Its relative interior, i.e., the $(m-1)$-dimensional stratum, is an $(m-1)$-dimensional manifold.
Thus, by induction on the dimension, Definition \ref{def of pseudomanifold}-1 is satisfied.
The density of $Q_{m} \setminus Q_{m-1}$ follows from the observation that all faces are contained in the boundary of $Q$.
This shows Definition \ref{def of pseudomanifold}-2.
Since $Q^m$ has the structure of a manifold with corners, a neighborhood of a point in the relative interior of an $i$-dimensional face of $Q^m$ can be chosen to be homeomorphic to $\mathbb{R}^{i} \times \mathbb{R}_{\geq 0}^{m-i}$. This is stratified homeomorphic to some neighborhood of the $m$-simplex.
Therefore, since convex polytopes have the structure of topological stratified pseudomanifolds (by Proposition \ref{polytope is pseudomanifold}), Definition \ref{def of pseudomanifold}-3 is satisfied.
\end{proof}

We next give the proof of Proposition~\ref{prop: locally standard T-manifold}.

\begin{proof}[Proof of Proposition \ref{prop: locally standard T-manifold}]
If $n = 0$, then $M^0$ consists of a single point.
Therefore, we have $M^0 \cong T^0$, and hence $M^0$ is a locally standard $T$-pseudomanifold.
Suppose $m > 0$.
Let
\[
\mathfrak{M} : M^{2m} = M_{2m} \supsetneq M_{2(m-1)} \supset \cdots \supset M_0 \supset \emptyset
\]
denote the filtration by orbit dimension.
The orbit projection $\pi: M^{2m} \to Q^m$ induces the filtration \eqref{filtration: manifold with faces}; that is, $Q_i = M_{2i}/T^m$ for each $i$ (see \cite{BP12} and \cite{MP06}). 
We verify that $(M, \mathfrak{M})$ satisfies conditions \eqref{cond-1}, \eqref{cond-2}, and \eqref{cond-3} in Section \ref{Notation}.

For any $x \in M_{2i} \setminus M_{2(i-1)}$, we have $\pi(x) \in Q_i \setminus Q_{i-1}$. 
By Proposition \ref{manifold with corners is a pseudomanifold}, the orbit space $Q^m$ is a compact topological stratified pseudomanifold.
Hence, we can take a small open neighborhood $U_{\pi(x)} \subset Q$ of $\pi(x)$ (see Definition~\ref{def of small open nbh}).
Furthermore, since $Q^m$ is a manifold with corners, we may assume that  
\[
U_{\pi(x)} \cong \mathbb{R}^i \times \mathbb{R}_{\geq 0}^{m-i}.
\]
Because the $T$-action is locally standard, there exists a weakly equivariant homeomorphism
\begin{align}\label{locally standard T-manifold: weak equivariant}
\psi_x : \pi^{-1}(U_{\pi(x)}) \to (\mathbb{C}^{\times})^i \times \mathbb{C}^{m-i},
\end{align}
where $(\mathbb{C}^{\times})^i \times \mathbb{C}^{m-i}$ is equipped with a coordinatewise $T^m$-action via the isomorphism (note that, by local standardness, isotropy subgroups are subtori):
\[
T^n \cong T/T_x \times T_x \cong T^{i} \times T^{m-i}.
\]
That is, $\psi_x$ is chosen so that the following diagram commutes:
\[
\begin{tikzcd}[ampersand replacement=\&]
\pi^{-1}(U_{\pi(x)}) \rar{\pi}  \dar{\psi_x} \& U_{\pi(x)} \dar{\cong} \\
 (\mathbb{C}^{\times})^i \times \mathbb{C}^{m-i}  \rar{/T}\&  \mathbb{R}^i \times \mathbb{R}_{\geq 0}^{m-i} \\
\end{tikzcd}
\]
Since $U_{\pi(x)}$ is small, we have $U_{\pi(x)} \cap Q_i \subset Q_i \setminus Q_{i-1}$.
By Definition~\ref{def of pseudomanifold}-3~(c), it follows that
\[
U_{\pi(x)} \cap Q_i \cong \mathbb{R}^i.
\]
Therefore, 
\[
\pi^{-1}(U_{\pi(x)} \cap Q_i) \cong (\mathbb{C}^{\times})^i,
\]
which shows that $x$ has a neighborhood homeomorphic to $\mathbb{R}^{2i}$.
Hence, $(M^{2m}, \mathfrak{M})$ is a manifold stratified space, verifying condition \eqref{cond-1}.
Since $M^{2m}$ is a locally standard $T$-manifold, the preimage of a point in $Q_i \setminus Q_{i-1}$ is an $i$-dimensional orbit.
Hence, the minimal dimension of the orbit of any point is $l$,
where $l$ is the minimal dimension of a stratum of $Q$.
It follows that condition \eqref{cond-2} is satisfied.
Condition \eqref{cond-3} follows directly from the fact that the action is locally standard (i.e., it acts locally coordinatewise), and hence has free orbits.

We next show that $M_{2m} \setminus M_{2(m-1)}$ is dense in $M^{2m}$.  
Note that the orbit projection $\pi: M^{2m} \to Q^n$ is a surjective open map.
Since $Q_m \setminus Q_{m-1}$ is dense in $Q^m$, it follows that
\[
\pi^{-1}(Q_m \setminus Q_{m-1})  =  M_{2m} \setminus M_{2(m-1)}
\]
is dense in $M^{2m}$. This verifies Definition~\ref{def of T-pseudomanifold}-1.

It remains to verify Definition~\ref{def of T-pseudomanifold}-2.
By Lemma~\ref{lem of sphere}, there exists a weakly equivariant homeomorphism
\[
\theta : (\mathbb{C}^{\times})^i \times \mathbb{C}^{m-i} \to (\mathbb{C}^{\times})^i \times \cone(\mathbb{S}^{2(m-i)-1}),
\]
where $\mathbb{S}^{2(m-i)-1}$ is the unit sphere with the standard $U(1)^{m-i}$-action and $T^m$ acts on $(\mathbb{C}^{\times})^i \times \cone(\mathbb{S}^{2(m-i)-1})$ via the isomorphism
\[
T^m \cong T/T_x \times T_x \cong T^i \times T^{m-i} \cong U(1)^i \times U(1)^{m-i},
\]
with $T^{m-i}$ acting on $\mathbb{S}^{2(m-i)-1}$.
We then put
\[
U_x := \pi^{-1}(U_{\pi(x)}), \quad L_x := \mathbb{S}^{2(m-i)-1},
\]
and define the composition
\[
\begin{tikzcd}[ampersand replacement=\&]
\varphi_x : \pi^{-1}(U_{\pi(x)}) \rar{\psi_x} \& (\mathbb{C}^{\times})^i \times \mathbb{C}^{m-i} \rar{\theta}\& (\mathbb{C}^{\times})^i \times \cone(S^{2(m-i)-1}).
\end{tikzcd}
\]
Then $U_x$ is a $T$-invariant open neighborhood, $L_x$ is a compact locally standard $T$-pseudomanifold (see Proposition~\ref{prop: standard sphere}), and $\varphi_x$ is a weakly equivariant homeomorphism.
Thus, Definition~\ref{def of T-pseudomanifold}-2 is satisfied.
Therefore, all the conditions in the definition of a locally standard $T$-pseudomanifold are fulfilled.
\end{proof}

\begin{rem}
Let $M$ be a locally standard torus manifold with orbit space $Q$ whose second cohomology vanishes (see \cite{MP06}).  
We can define the characteristic data $(Q, \lambda, 0)$ of $M$ and construct its canonical model $M(Q, \lambda, 0)$.  
By \cite[Lemma 4.5]{MP06}, $M$ is equivariantly homeomorphic to $M(Q, \lambda, 0)$.  
Our result (see Lemma~\ref{main thm}) includes this as a special case.
\end{rem}

\section{An example of a locally standard $T$-pseudomanifold over a pyramid}\label{last ex}
In this section, we present a concrete example of a $6$-dimensional locally standard $T$-pseudomanifold over a convex polytope that is not simple.

We consider the Thom space $X$ of the complex line bundle $(S^3 \times S^3) \times_{T^2} \mathbb{C}$ over the Hirzebruch surface $(S^{3}\times S^{3})/T^{2} = \mathbb{C}P(\underline{\mathbb{C}} \oplus \gamma^{\otimes k})$, where $\underline{\mathbb{C}}$ denotes the trivial complex line bundle and $\gamma^{\otimes k}$ is the $k$-th tensor power of the canonical line bundle $\gamma$ over $\mathbb{C}P^1$.
That is, 
\[
X=X_6:=(S^3 \times S^3) \times_{T^2} D^2/(S^3 \times S^3) \times_{T^2} S^1.
\]
Here, the $T^2$-action on $(S^3 \times S^3) \times \mathbb{C}$ is defined as follows: 
\[
(t_1, t_2) \cdot \left( \begin{pmatrix} x_1 \\ x_2 \end{pmatrix}, \begin{pmatrix} y_1 \\ y_2 \end{pmatrix}, z \right) 
= \left( t_1 \begin{pmatrix} x_1 \\ x_2 \end{pmatrix}, t_2 \begin{pmatrix} y_1 \\ t_{1}^{k} y_2 \end{pmatrix}, t_{1}^{a} t_{2}^{b} z \right),
\]
for $(t_1, t_2) \in T^2$ and $\left( \begin{pmatrix} x_1 \\ x_2 \end{pmatrix}, \begin{pmatrix} y_1 \\ y_2 \end{pmatrix}, z \right) \in (S^3 \times S^3) \times \mathbb{C}$, where $k, a, b \in \mathbb{Z}$ are fixed integers. 
We define a $T^3$-action on the Thom space $(S^3 \times S^3) \times_{T^2} D^2/(S^3 \times S^3) \times_{T^2} S^1$ by
\[
(s_1, s_2, s_3) \cdot \left[ \begin{pmatrix} x_1 \\ x_2 \end{pmatrix}, \begin{pmatrix} y_1 \\ y_2 \end{pmatrix}, z \right]
= \left[ \begin{pmatrix} x_1 \\ s_1 x_2 \end{pmatrix}, \begin{pmatrix} y_1 \\ s_2 y_2 \end{pmatrix}, s_3 z \right]
\]
for $(s_1, s_2, s_3) \in T^3$ and $\left[ \begin{pmatrix} x_1 \\ x_2 \end{pmatrix}, \begin{pmatrix} y_1 \\ y_2 \end{pmatrix}, z \right] \in (S^3 \times S^3) \times_{T^2} D^2/(S^3 \times S^3) \times_{T^2} S^1$.

\begin{figure}[htbp]
  \centering
\begin{tikzpicture}[scale=0.9]
  \coordinate (A) at (0, 0);  
  \coordinate (B) at (4, 0);
  \coordinate (C) at (6, 1.2);
  \coordinate (D) at (2, 1.2);
  \coordinate (S) at (3, 4.5);

  \draw[thick] (A)--(B)--(C);
  \draw[dashed, thick] (A)--(D)--(C);
  \draw[thick] (A)--(S);
  \draw[thick] (B)--(S);
  \draw[thick] (C)--(S);
  \draw[dashed, thick] (D)--(S);

  \fill[gray!30, opacity=0.5] (A) -- (B) -- (C) -- (D) -- (A);

  \fill[gray!30, opacity=0.4] (S) -- (A) -- (B) -- (S);
  \fill[gray!30, opacity=0.4] (S) -- (B) -- (C) -- (S);
  \fill[gray!30, opacity=0.4] (S) -- (C) -- (D) -- (S);
  \fill[gray!30, opacity=0.4] (S) -- (D) -- (A) -- (S);

   \draw[fill=gray!30, draw=gray!, rotate around={30.96:(1, 0.6)}]
      (1, 0.6) ellipse [x radius=1.165, y radius=0.3];

    \draw[fill=gray!30, rotate around={-47.73:(4.5, 2.85)}]
      (4.5, 2.85) ellipse [x radius=2.23, y radius=0.3];
    \draw[fill=gray!30, draw=gray!, rotate around={73.16:(2.5, 2.85)}]
      (2.5, 2.85) ellipse [x radius=1.725, y radius=0.3];

    \draw[fill=gray!30, draw=gray!, rotate around={0:(4, 1.2)}]
      (4, 1.2) ellipse [x radius=2, y radius=0.3];

  \draw[fill=gray!30, rotate around={30.96:(5, 0.6)}]
      (5, 0.6) ellipse [x radius=1.165, y radius=0.3];

   \draw[fill=gray!30, rotate around={0:(2, 0)}]
      (2, 0) ellipse [x radius=2, y radius=0.3];

    \draw[fill=gray!30, rotate around={56.31:(1.5, 2.25)}] 
      (1.5, 2.25) ellipse [x radius=2.7, y radius=0.3];
    \draw[fill=gray!30, rotate around={-77.47:(3.5, 2.25)}]
      (3.5, 2.25) ellipse [x radius=2.305, y radius=0.3];

\fill (3, 4.5) circle (2pt);  
\fill (0, 0) circle (2pt);    
\fill (4, 0) circle (2pt);    
\fill (6, 1.2) circle (2pt);  
\fill (2, 1.2) circle (2pt);  

  \node at (A) [below left] {B};
  \node at (B) [below right] {C};
  \node at (C) [above right] {D};
  \node at (D) [above left] {E};
  \node at (S) [above] {A};
\end{tikzpicture}
\caption{Thom space of the complex line bundle over the Hirzebruch surface}
  \label{figure in last ex}
\end{figure}

We enumerate the strata of $X$ in each dimension, below:

\subsubsection*{The fixed points}
{\footnotesize
\begin{align*}
A=\left[\begin{pmatrix} x_1 \\ x_2 \end{pmatrix}, \begin{pmatrix} y_1 \\ y_2 \end{pmatrix}, w (\in S^1)\right],
B=\left[\begin{pmatrix} 1 \\ 0 \end{pmatrix},\begin{pmatrix} 1 \\ 0 \end{pmatrix},0\right], 
C=\left[\begin{pmatrix} 0 \\ 1 \end{pmatrix},\begin{pmatrix} 1 \\ 0 \end{pmatrix},0\right],  
D=\left[\begin{pmatrix} 0 \\ 1 \end{pmatrix},\begin{pmatrix} 0 \\ 1 \end{pmatrix},0\right], 
E=\left[\begin{pmatrix} 1 \\ 0 \end{pmatrix},\begin{pmatrix} 0 \\ 1 \end{pmatrix},0\right].
\end{align*}
}
We denote the set of these fixed points by $X_0$.

\subsubsection*{The $2$-skeleton}
\begin{itemize}
\item
Radial edges in Figure \ref{figure in last ex}:
{\footnotesize
\begin{align*}
&S_{AB}=\left\{  \left[\begin{pmatrix} 1 \\ 0 \end{pmatrix},\begin{pmatrix} 1 \\ 0 \end{pmatrix},z\right] \mid z \in D^2  \right\} \setminus X_0, 
&S_{AC}=\left\{  \left[\begin{pmatrix} 0 \\ 1 \end{pmatrix},\begin{pmatrix} 1 \\ 0 \end{pmatrix},z\right] \mid z \in D^2  \right\} \setminus X_0, \\
&S_{AD}=\left\{  \left[\begin{pmatrix} 0 \\ 1 \end{pmatrix},\begin{pmatrix} 0 \\ 1 \end{pmatrix},z\right] \mid z \in D^2  \right\} \setminus X_0, 
&S_{AE}=\left\{  \left[\begin{pmatrix} 1 \\ 0 \end{pmatrix},\begin{pmatrix} 0 \\ 1 \end{pmatrix},z\right] \mid z \in D^2  \right\} \setminus X_0.
\end{align*}
}
\item
Base edges in Figure \ref{figure in last ex}:
{\footnotesize
\begin{align*}
&S_{BC}=\left\{  \left[\begin{pmatrix} x_1 \\ x_2 \end{pmatrix},\begin{pmatrix} 1 \\ 0 \end{pmatrix},0\right] \mid \begin{pmatrix} x_1 \\ x_2 \end{pmatrix} \in S^3  \right\} \setminus X_0, 
&S_{CD}=\left\{  \left[\begin{pmatrix} 0 \\ 1 \end{pmatrix},\begin{pmatrix} y_1 \\ y_2 \end{pmatrix},0\right] \mid \begin{pmatrix} y_1 \\ y_2 \end{pmatrix} \in S^3  \right\} \setminus X_0, \\
&S_{DE}=\left\{  \left[\begin{pmatrix} x_1 \\ x_2 \end{pmatrix},\begin{pmatrix} 0 \\ 1 \end{pmatrix},0\right] \mid \begin{pmatrix} x_1 \\ x_2 \end{pmatrix} \in S^3  \right\} \setminus X_0, 
&S_{EB}=\left\{  \left[\begin{pmatrix} 1 \\ 0 \end{pmatrix},\begin{pmatrix} y_1 \\ y_2 \end{pmatrix},0\right] \mid \begin{pmatrix} y_1 \\ y_2 \end{pmatrix} \in S^3  \right\} \setminus X_0. 
\end{align*}
}
\end{itemize}

\subsubsection*{The $4$-skeleton}
{\footnotesize
\begin{align*}
&S_{BCDE}=\left\{  \left[\begin{pmatrix} x_1 \\ x_2 \end{pmatrix},\begin{pmatrix} y_1 \\ y_2 \end{pmatrix},0\right] \mid \left(\begin{pmatrix} x_1 \\ x_2 \end{pmatrix},\begin{pmatrix} y_1 \\ y_2 \end{pmatrix}\right) \in S^3 \times S^3  \right\} \setminus X_2,\quad\text{(base face)} \\
&S_{ABC}=\left\{  \left[\begin{pmatrix} x_1 \\ x_2 \end{pmatrix},\begin{pmatrix} 1 \\ 0 \end{pmatrix},z\right] \mid \begin{pmatrix} x_1 \\ x_2 \end{pmatrix} \in S^3, z \in D^2  \right\} \setminus X_2, \,\,
S_{ADE}=\left\{  \left[\begin{pmatrix} x_1 \\ x_2 \end{pmatrix},\begin{pmatrix} 0 \\ 1 \end{pmatrix},z\right] \mid \begin{pmatrix} x_1 \\ x_2 \end{pmatrix} \in S^3, z \in D^2  \right\} \setminus X_2, \\
&S_{ACD}=\left\{  \left[\begin{pmatrix} 0 \\ 1 \end{pmatrix},\begin{pmatrix} y_1 \\ y_2 \end{pmatrix},z\right] \mid \begin{pmatrix} y_1 \\ y_2 \end{pmatrix} \in S^3, z \in D^2  \right\} \setminus X_2, \,\,
S_{AEB}=\left\{  \left[\begin{pmatrix} 1 \\ 0 \end{pmatrix},\begin{pmatrix} y_1 \\ y_2 \end{pmatrix},z\right] \mid \begin{pmatrix} y_1 \\ y_2 \end{pmatrix} \in S^3, z \in D^2  \right\} \setminus X_2.
\end{align*}
}

\subsubsection*{The isotropy subgroups to the corresponding faces}
We next enumerate the isotropy subgroups corresponding to each stratum.
The isotropy subgroups corresponding to the $4$-dimensional strata are:
{\footnotesize
\begin{align*}
S_{BCDE} \text{ (base face)} &\longleftrightarrow T\langle 0,0,1 \rangle, \\
S_{ABC} \text{ (front face)} &\longleftrightarrow T\langle 0,1,0 \rangle, 
&S_{ADE} \text{ (back face)} &\longleftrightarrow T\langle 0,1,b \rangle, \\
S_{ACD} \text{ (right face)} &\longleftrightarrow T\langle 1,k,a \rangle, 
&S_{AEB} \text{ (left face)} &\longleftrightarrow T\langle 1,0,0 \rangle.
\end{align*}
}
For the $2$-dimensional strata (edges):
\begin{itemize}
\item
Base edges in Figure \ref{figure in last ex}:
{\footnotesize
\begin{align*}
S_{BC}\text{ (front edge)} &\longleftrightarrow T\langle (0,0,1),(0,1,0) \rangle = \{ (1,t,s) \mid (s,t) \in T^2 \}, \\
S_{CD}\text{ (right edge)} &\longleftrightarrow T\langle (0,0,1),(1,k,a) \rangle = \{ (t,t^k,s t^a) \mid (s,t) \in T^2 \}, \\
S_{DE}\text{ (back edge)} &\longleftrightarrow T\langle (0,0,1),(0,1,b) \rangle = \{ (1,t,s t^b) \mid (s,t) \in T^2 \}, \\
S_{EB}\text{ (left edge)} &\longleftrightarrow T\langle (0,0,1),(1,0,0) \rangle = \{ (t,1,s) \mid (s,t) \in T^2 \}.
\end{align*}
}
\item
Radial edges in Figure \ref{figure in last ex}:
{\footnotesize
\begin{align*}
S_{AB} &\longleftrightarrow T\langle (0,1,0),(1,0,0) \rangle = \{ (t,s,1) \mid (s,t) \in T^2 \}, \\
S_{AC} &\longleftrightarrow T\langle (0,1,0),(1,k,a) \rangle = \{ (t,st^k,t^a) \mid (s,t) \in T^2 \}, \\
S_{AD} &\longleftrightarrow T\langle (0,1,b),(1,k,a) \rangle = \{ (t,st^k,s^bt^{a}) \mid (s,t) \in T^2 \}, \\
S_{AE} &\longleftrightarrow T\langle (0,1,b),(1,0,0) \rangle = \{ (t,s,s^b) \mid (s,t) \in T^2 \}.
\end{align*}
}
\end{itemize}
Each fixed point corresponds to $T^3$, and the top stratum $X_6 \setminus X_4$ corresponds to the trivial subgroup $\{1\} \subset T^{3}$. Therefore, we obtain the characteristic data $(Q, \lambda, 0)$ of $X$, where $Q$ is the pyramid $ABCDE$ and $\lambda$ is the associated characteristic functor that assigns the above isotropy subgroups to the corresponding faces.
Since $Q$ and its interior are contractible, the Chern class of $X$ is $0$.

\begin{rem}
Note that the weights used to describe the isotropy subgroup associated with each stratum are not uniquely determined in general.
For example, the isotropy subgroup corresponding to $S_{CD}$ can be written as
\[
T\langle (0,0,1),(1,k,a) \rangle,
\] 
but this subgroup is isomorphic to \[T\langle (0,0,1),(1,k,0) \rangle.\]
\end{rem}

\subsubsection*{The link of each point}
Finally, the link at each point can be taken as follows:
{\footnotesize
\begin{itemize}
\item The link $L_A$ of the vertex $A$ is defined by
\[
L_A \cong (S^{3}\times S^{3}) \times_{T^{2}} S^1,
\]
where $T^2$ acts on $(S^{3}\times S^{3}) \times S^1$ by
\[
(t_1, t_2) \cdot \left( \begin{pmatrix} x_1 \\ x_2 \end{pmatrix}, \begin{pmatrix} y_1 \\ y_2 \end{pmatrix}, z \right) 
= \left( t_1 \begin{pmatrix} x_1 \\ x_2 \end{pmatrix}, t_2 \begin{pmatrix} y_1 \\ t_{1}^{k} y_2 \end{pmatrix}, t_{1}^{a} t_{2}^{b} z \right).
\]

\item $L_B \cong L_C \cong L_D \cong L_E \cong S^5$, {with the standard $T^3$-action}.

\item For $x \in X_2 \setminus X_0$, the link $L_{x} \cong S^3$, {with the standard $T^2$-action}.

\item For $x \in X_4 \setminus X_2$, the link $L_{x} \cong S^1$, {with the standard $T^1$-action}.
\end{itemize}
}

\subsubsection*{Conditions for the example to be a toric variety}

We conclude this section by determining when this example is a toric variety (up to a $T$-equivariant homeomorphism).

\begin{rem}
We verify whether there exists a toric variety $X_{\Sigma}$ whose characteristic data is isomorphic to that of $X$, where $\Sigma$ is the fan corresponding to $X_{\Sigma}$.
If such a toric variety $X_{\Sigma}$ exists, then by Theorem~\ref{classification theorem},
$X$ and $X_{\Sigma}$ are $T$-equivariantly homeomorphic.

According to \cite[Chapter~12]{CLS11}, for the characteristic functor associated with a toric variety $X_{\Sigma}$, the isotropy subgroup corresponding to each facet is generated by the primitive generator of the corresponding ray ($1$-dimensional cone) of $\Sigma$.
Therefore, each ray generator must be of one of the following forms (where, for each generator, one of the two possible signs is chosen):
\[
\pm (0,1,0)_{ABC},
\pm (1,k,a)_{ACD},
\pm (0,1,b)_{ADE},
\pm (1,0,0)_{AEB},
\pm (0,0,1)_{BCDE}.
\]
Here, the subscripts indicate the corresponding facets.
We consider a fan $\Sigma$ whose cones are generated by these vectors.
We assume that each maximal cone of $\Sigma$ is generated by the ray generators corresponding to the facets around a vertex of the pyramid.
For example,
the facets around the vertex $A$ are $ABC, ACD, ADE$ and $AEB$.
Hence, the corresponding maximal cone $\sigma_{A}$ is given by
\[
\sigma_{A}:=\pm (0,1,0)_{ABC} \mathbb{R}_{\ge0}  \pm (1,k,a)_{ACD} \mathbb{R}_{\ge0} 
 \pm (0,1,b)_{ADE} \mathbb{R}_{\ge0}  \pm (1,0,0)_{AEB} \mathbb{R}_{\ge0}.
\]
Since $X$ is compact, it suffices to consider compact toric varieties $X_{\Sigma}$.
Because a compact toric variety corresponds to a complete fan, we investigate the conditions under which $\Sigma$ can be constructed as a complete fan.
\begin{itemize}
\item
When $a=0$ or $b=0$, no such fan $\Sigma$ can be constructed.
Indeed, it is easy to see that in this case $\sigma_{A}$ is either not strongly convex or some ray appears that does not occur as a face of any maximal cone.

\item
When $a>0$ and $b>0$, by choosing the signs as
\[
- (0,1,0)_{ABC},
 (1,k,a)_{ACD},
 (0,1,b)_{ADE},
- (1,0,0)_{AEB},
- (0,0,1)_{BCDE}.
\]
we obtain a complete fan $\Sigma$.

\item
When $a>0$ and $b<0$, by choosing the signs as
\[
 (0,1,0)_{ABC},
 (1,k,a)_{ACD},
-(0,1,b)_{ADE},
- (1,0,0)_{AEB},
- (0,0,1)_{BCDE}.
\]
we obtain a complete fan $\Sigma$.

\item
When $a<0$ and $b>0$, by choosing the signs as
\[
- (0,1,0)_{ABC},
 -(1,k,a)_{ACD},
 (0,1,b)_{ADE},
 (1,0,0)_{AEB},
- (0,0,1)_{BCDE}.
\]
we obtain a complete fan $\Sigma$.

\item
When $a<0$ and $b<0$, by choosing the signs as
\[
 (0,1,0)_{ABC},
- (1,k,a)_{ACD},
- (0,1,b)_{ADE},
 (1,0,0)_{AEB},
- (0,0,1)_{BCDE}.
\]
we obtain a complete fan $\Sigma$.
\end{itemize}
It is straightforward to verify that the topological stratified pseudomanifold $(B, \mathfrak{B})$ induced from the complete fan constructed above (see Construction~\ref{filtration construction}) is stratified homeomorphic to the pyramid.
The rays of the fan are chosen so that the characteristic functor of $X_{\Sigma}$ agrees with that of $X$ on facets (that is, the labels on the facets of the pyramid).
In this example, the isotropy subgroup corresponding to a lower-dimensional face of the pyramid is induced from the isotropy subgroups corresponding to the facets containing it.
On the other hand, for $X_{\Sigma}$, the isotropy subgroup of each stratum is determined by the corresponding cone, which is generated by rays. Thus, the isotropy subgroups of $X_{\Sigma}$ are also induced from the rays.
Therefore, the characteristic functors of $X$ and $X_{\Sigma}$ coincide.
Moreover, the Chern classes of both $X$ and $X_{\Sigma}$ are $0$.
Therefore, by Theorem~\ref{classification theorem}, we obtain $X\cong X_{\Sigma}$.
Consequently, we conclude that $X$ is a toric variety if and only if $a\neq0$ and $b\neq0$.
\end{rem}

\appendix
\section{A fact in topology}\label{app: topology}

In this section, we prepare a lemma from topology. For the basics of topology, see \cite{Mun14} and \cite{Wil70}, among others.

\begin{lem}\label{dense coincide}
Let $X$ be a topological space, $Y$ a Hausdorff space, and $Z$ a dense subset of $X$. Suppose that two continuous maps $f, g: X \to Y$ coincide on $Z$. Then $f=g$.
\end{lem}
\begin{proof}
Since $Y$ is Hausdorff, the set $A = \{ x \in X \mid f(x) = g(x) \}$ is closed in $X$ (see \cite[Section 31, Exercises-5]{Mun14}). From the assumption that $f(x) = g(x)$ for $x \in Z$, we have $Z \subset A$. Since $Z$ is dense and $A$ is closed in $X$, it follows that $X \subset A$. Thus, we conclude that 
\[
X = A = \{ x \in X \mid f(x) = g(x) \}.
\]
This implies that $f=g$.
\end{proof}

\section{Some facts in topological stratified pseudomanifold}\label{BB}
In this section, we define topological stratified pseudomanifolds and prepare several supporting lemmas.

\begin{definition}[topological stratified pseudomanifold {{\cite[Chapter 2]{Fri20}}}]\label{def of pseudomanifold}

Let $Q$ be a Hausdorff space.
A {\it topological stratified pseudomanifold $Q$} is defined inductively by dimension as follows:
\begin{itemize}
\item
A {\it $0$-dimensional topological stratified pseudomanifold} is a set of points with the discrete topology;
\item
For $n\ge 0$ and $l \ge 0$, with $l+n\neq 0$, an {\it $(l+n)$-dimensional topological stratified pseudomanifold} $Q$ is equipped with a filtration
\[
\mathfrak{Q} : Q = Q_{l+n} \supsetneq Q_{l+n-1} \supset Q_{l+n-2} \supset \cdots \supset Q_{l+i} \supset \cdots \supset Q_l \supsetneq Q_{l-1}=\emptyset,
\] 
which satisfies the following three conditions for every $0 \leq i \leq n$.
\begin{enumerate}
\item if an $(l+i)$-dimensional strata $Q_{l+i}\setminus Q_{l+i-1}\not=\emptyset$, then each connected component of $Q_{l+i}\setminus Q_{l+i-1}$ is an $(l+i)$-dimensional (topological) manifold;

\item the top dimensional strata $Q_{l+n} \setminus Q_{l+n-1}$ is dense in $Q$;
\item for each $p \in Q_{l+i} \setminus Q_{l+i-1} (\not=\emptyset)$, there exists a triple $(U_p, L_p, \varphi_p)$ consisting of
\begin{enumerate}
\item
an open neighborhood $U_p$ of $p$ in $Q$;
\item
an $(n-i-1)$-dimensional compact topological stratified pseudomanifold $L_p$, possibly the empty set $\emptyset$ ($L_p$ is called a {\it link of $p$});
\item
a homeomorphism $\varphi_p : U_p \to O_p \times \mathring{c}(L_p) $ whose restriction map induces the homeomorphism 
\[
\varphi_p |_{U_p \cap Q_{l+i+j+1}} : U_p \cap Q_{l+i+j+1} \to O_p \times \mathring{c}((L_{p})_j),
\]
for $-1 \leq j \leq n-i-1$, where $O_p$ is a contractible open subset in $\mathbb{R}^{l+i}$.
Here, by the definition of an open cone, $\cone(L_p)$ has a natural filtration by
$
(\mathring{c}(L_{p}))_{j+1}
:=
\mathring{c}((L_{p})_{j})
$
for every $j$ with $-1\le j\le n-i-1$.
\end{enumerate}
\end{enumerate}
\end{itemize}
Note that if $j=-1$, then the homeomorphism $\varphi_{p}|_{U_{p}\cap Q_{l+i}}:U_{p}\cap Q_{l+i}\to O_{p}\times \cone(\emptyset)$ may be regarded as the local coordinate of the $(l+i)$-dimensional manifold $Q_{l+i}\setminus Q_{l+i-1}$.
\end{definition}

\begin{rem}
The "topological stratified pseudomanifold" defined in {\cite{Fri20}} and the "topological pseudomanifold" defined in \cite[Section 1.1]{GM80} and \cite[p.82]{GM83} are different concepts. The latter is called {\it classical} in \cite{Fri20}; that is, we call an $(l+n)$-dimensional topological stratified pseudomanifold {\it classical} if $Q_{l+n-1} = Q_{l+n-2}$.
\end{rem}

\begin{rem}\label{rem of stratified homeo}
If $\varphi_p$ is a stratified homeomorphism, then the restriction conditions in the latter part of Definition~\ref{def of pseudomanifold}-3 (c) are automatically satisfied.
\end{rem}

\begin{prop}[{{\cite[Lemma 2.3.7]{Fri20}}}]\label{frontier condition}
A topological stratified pseudomanifold satisfies that for any two strata $S$ and $S'$, if ${S} \cap \overline{S'} \neq \emptyset$, then ${S} \subset \overline{S'}$.
\end{prop}

We call the condition in Proposition \ref{frontier condition} a {\it frontier condition} of a topological stratified pseudomanifold.

The following lemma asserts that, assuming the topological stratified pseudomanifold is compact, we can take a small open neighborhood around each point.

\begin{lem}\label{lem of small open nbh}
Let $Q$ be a compact topological stratified pseudomanifold and $S$ be a stratum. For each $p \in {S} \subset Q$, we can take a small open neighborhood $U_p$ (see Definition \ref{def of small open nbh}) of $p$ in $Q$.
\end{lem}
\begin{proof}
For any $p \in Q$, let $U$ be an open neighborhood of $p$ in $Q$. If $U$ is small, then we may put $U_p :=U$. Otherwise, by Definition \ref{def of small open nbh}, there exists a stratum $S'$ such that
\[
U \cap {S'} \neq \emptyset 
\quad \text{and}  \quad {S}
\not\subset \overline{S'}.
\]
By the frontier condition (see Proposition \ref{frontier condition}), if ${S} \cap \overline{S'} \neq \emptyset$ then ${S} \subset \overline{S'}$.
Taking the contrapositive, we have
\[
{S} \cap \overline{S'} = \emptyset.
\]
Thus, $p \notin S'$.
Due to {{\cite[Corollary 2.3.16]{Fri20}}}, a topological stratified pseudomanifold is a regular space.
Therefore, there exist open subsets $U_{Q}'$ and $V$ such that $p \in U_{Q}' ,\, \overline{S'} \subset V$ and $U_{Q}' \cap V = \emptyset$. Now we put 
\[
U' := U_{Q}' \cap U.
\]
This is an open neighborhood of $p$, and $U' \cap S' =\emptyset $.
If $U'$ is small, then we may put $U_p := U'$.
Otherwise, the same argument can be applied again. By {{\cite[Lemma 2.3.8-3)]{Fri20}}}, a compact topological stratified pseudomanifold has only finitely many strata. Therefore, this process must stop after a finite number of steps.
\end{proof}

\begin{rem}\label{rem of small open nbh}
Let $Q$ be a compact topological stratified pseudomanifold.
For $p \in Q$, an open neighborhood $U_p$ in Definition \ref{def of pseudomanifold}-3 can be taken as a small open neighborhood, in the following way.
By Lemma \ref{lem of small open nbh}, we can take a small neighborhood $U$ of $p$. 
The {\it subspace filtration} is defined by
\[
\mathfrak{U}: U=U_{l+n} \supset \cdots \supset
U_{l+i}
\supset \cdots \supset
U_l \supsetneq \emptyset,
\]
where $U_{l+i}=Q_{l+i} \cap U$.
According to {{\cite[Lemma 2.4.10]{Fri20}}}, by the subspace filtration, $U$ has the structure of a topological stratified pseudomanifold.
Therefore, there exists an open neighborhood $U_p \,(\subset U)$ of $p$ that satisfies the condition in Definition \ref{def of pseudomanifold}-3.
Since $U$ is small, $U_p \subset U$ is also small.
Furthermore, by Definition \ref{def of pseudomanifold}, $U_p$ is homeomorphic to $O_p \times \cone(L_p^{n-i-1})$. Since both $O_p$ and $\cone(L_p^{n-i-1})$ are contractible, $U_p$ is also contractible. This implies that the small open neighborhood $U_p$ can be chosen to be contractible.
\end{rem}

\section{Locally standard $T$-pseudomanifold over a convex polytope}\label{sec: polytope}
In this section, we first show that a convex polytope admits the structure of a topological stratified pseudomanifold (Subsection \ref{subsec: polytope is pseudomanifold}).
Subsequently, we demonstrate that a characteristic function on a simple polytope (see Definition \ref{def of char function}) determines a characteristic functor (see Subsection~\ref{subsec: characteristic function on simple polytope}).
We begin by establishing terminology related to convex polytopes.

Let $P \subset \mathbb{R}^n$ be an $n$-dimensional convex polytope.
We assume that $P$ has the relative topology inherited from the $n$-dimensional Euclidean space $\mathbb{R}^n$.
For each $0 \leq i \leq n$, let $\mathcal{F}_{i}(P)$ denote the collection of all $i$-dimensional faces of $P$. 
To emphasize the dimension, we often denote $F \in \mathcal{F}_i (P)$ by $F^i$.
For each face $F$, 
we denote the relative interior of $F$ by $\mathring{F}$. 
See {\cite{Zie96}} for the basic facts about convex polytopes.

A convex polytope admits the following natural filtration associated with its faces.

\begin{definition}[polytopal filtration]\label{def of polytopal filtration}
Let $P$ be an $n$-dimensional convex polytope.
The following filtration exists:
\[
\mathfrak{P}: P \supset \bigcup_{F^{n-1} \in \mathcal{F}_{n-1}(P) }F^{n-1} \supset \cdots \supset \bigcup_{F^i \in \mathcal{F}_i(P)}F^i \supset \cdots \supset \bigcup_{v \in \mathcal{F}_0(P)}\{v\} \supsetneq \emptyset.
\]
We refer to this filtration as the {\it polytopal filtration} of $P$.
\end{definition}

\subsection{A convex polytope is a topological stratified pseudomanifold}\label{subsec: polytope is pseudomanifold}
In this subsection, we first prove the following proposition.
\begin{prop}\label{polytope is pseudomanifold}
An $n$-dimensional convex polytope $P$ equipped with the polytopal filtration
is a topological stratified pseudomanifold. 
\end{prop}

\begin{rem}
Proposition \ref{polytope is pseudomanifold} appears to be well-known among experts. However, for the sake of completeness, we include a proof in this paper.
\end{rem}

We put $Q_i = \bigcup_{F^i \in \mathcal{F}_{i}(P)}F^i$, for $0 \leq i \leq n$.
Then, for each $0 \leq i \leq n$,
\[
Q_i \setminus Q_{i-1} =
\bigcup_{F^i \in \mathcal{F}_{i}(P)}F^i 
\setminus
\bigcup_{F^{i-1} \in \mathcal{F}_{i-1}(P)}F^{i-1}
=
\bigsqcup_{F^i \in \mathcal{F}_i(P)} \mathring{F}^i.
\]

To verify the condition in Definition \ref{def of pseudomanifold} for $(P, \mathfrak{P})$,
we first construct an $(n-i-1)$-dimensional polytope $L_{p}^{n-i-1} \subset P$ which will be a link of $p \in \mathring{F}^i \subset P$.

\begin{cons}[link $L_{p}^{n-i-1}$]\label{cons of link in polytope}
For any point $p \in P$, suppose that $p$ lies in the interior of an $i$-dimensional face $F^i$.
We construct a polytope $L_{p}^{n-i-1} \subset P$, which may be regarded as the {\it link} of $p$. For $i=n$, in this case $F^i=P$. We define $L_{p}^{-1}:=\emptyset$.
For $0 \leq i < n$, we construct $L_{p}^{n-i-1}$ by induction on $i$.
\begin{itemize}
\item
For $i=0$, the face $F^0$ is a vertex of $P$. In this case, $F^{0}=\{p\}$. Take a hyperplane $\mathcal{H}_{F^0}$ that separates $F^0$ from the other vertices of $P$. Then, we set
\[
L_{p}^{n-1}:=P \cap \mathcal{H}_{F^0}.
\]
This is also known as a {\it vertex figure} (see {{\cite[Chapter2]{Zie96}}}).
\[
\begin{tikzpicture}
  \coordinate (B) at (0, 0);  
  \coordinate (C) at (4, 0);
  \coordinate (D) at (6, 1.2);
  \coordinate (E) at (2, 1.2);
  \coordinate (A) at (3, 4.5);

  \coordinate (B1) at (0, 2.7);
  \coordinate (C1) at (4, 2.7);
  \coordinate (D1) at (6, 3.9);
  \coordinate (E1) at (2, 3.9);


  \draw[dashed, thick] (B)--(E)--(D);
  \draw[dashed, thick] (E)--(A);
  \draw[thick] (D)--(A);


  \draw[thick] (B)--(C)--(D);
  \draw[thick] (B)--(A);
  \draw[thick] (C)--(A);

  \fill[gray, opacity=0.2] (B1) -- (C1) -- (D1) -- (E1) -- cycle;
  \draw[dotted, thick] (B1) -- (C1) -- (D1) -- (E1) -- cycle;

  \coordinate (P) at (2.25, 3.375);  
  \coordinate (Q) at (3.25, 3.375);  
  \coordinate (R) at (3.75, 3.675);    
  \coordinate (S) at (2.75, 3.675);    

  \draw[red] (P) -- (Q) -- (R) -- (S) -- cycle;
  \fill[red, opacity=0.2] (P) -- (Q) -- (R) -- (S) -- cycle;

  \node at (B) [below left] {B};
  \node at (C) [below right] {C};
  \node at (D) [right] {D};
  \node at (E) [left] {E};
  \node at (A) [above] {A};
\node at (D1) [below, yshift=-2pt] {$\mathcal{H}_{\mathrm{A}}$};
\node[red] at (R) [right] {$L_{\mathrm{A}}^{n-1}$};
\node at (3, -0.7) {vertex figure of A};
\end{tikzpicture}
\]

\item
For $1 \leq i \leq n-1$, we construct $L_{F^i}^{n-i-1}$ as follows. Take a hyperplane $\mathcal{H}_{F^i}$ that separates the vertices of $F^i$ from the other vertices of $P$.
Since the polytope is embedded in $\mathbb{R}^n$, we can consider the normal space $N_p (F^i)$ to $F^i$ at $p$. Note that $\dim N_p (F^i)=n-i.$ We set
\[
L_p^{n-i-1}:= P \cap\mathcal{H}_{F^i} \cap N_p (F^i) .
\]
Then, since $\mathcal{H}_{F^i}$ and $N_p (F^i)$ are intersecting transversally, $L_p^{n-i-1}$ is an $(n-i-1)$-dimensional convex polytope.
\end{itemize}
\end{cons}

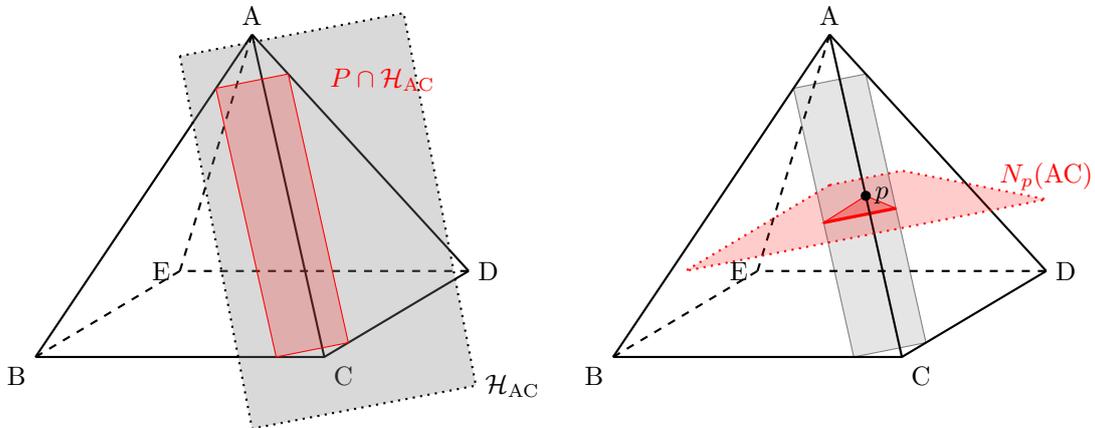
\begin{figure}[htbp]
\centering
\begin{tikzpicture}[scale=0.95]
  \begin{scope}[xshift=0cm]
    \coordinate (B) at (0, 0);  
    \coordinate (C) at (4, 0);
    \coordinate (D) at (6, 1.2);
    \coordinate (E) at (2, 1.2);
    \coordinate (A) at (3, 4.5);
    \draw[thick] (B)--(C)--(D);
    \draw[dashed, thick] (B)--(E)--(D);
    \draw[thick] (B)--(A);
    \draw[thick] (C)--(A);
    \draw[thick] (D)--(A);
    \draw[dashed, thick] (E)--(A);
    \node at (B) [below left] {B};
    \node at (C) [below right] {C};
    \node at (D) [right] {D};
    \node at (E) [left] {E};
    \node at (A) [above] {A};
    \coordinate (P1) at (2, 4.2);
    \coordinate (P2) at (3, -1);
    \coordinate (P3) at (6.1, -0.4);
    \coordinate (P4) at (5.1, 4.8);
    \fill[gray, opacity=0.3] (P1) -- (P2) -- (P3) -- (P4) -- cycle;
    \draw[dotted, thick] (P1) -- (P2) -- (P3) -- (P4) -- cycle;
    \coordinate (PAB) at (2.5, 3.75);
    \coordinate (PCB) at (3.3333, 0);
    \coordinate (PAD) at (3.5, 3.95);
    \coordinate (PCD) at (4.3333, 0.2);
    \draw[red] (PAB) -- (PCB) -- (PCD) -- (PAD) -- cycle;
    \fill[red, opacity=0.2] (PAB) -- (PCB) -- (PCD) -- (PAD) -- cycle;
    \node at (P3) [right] {$\mathcal{H}_{\mathrm{AC}}$};
\node[red] at (R) [right] {$P \cap \mathcal{H}_{\mathrm{AC}}$};
  \end{scope}
  \begin{scope}[xshift=8cm]
    \coordinate (B) at (0, 0);  
    \coordinate (C) at (4, 0);
    \coordinate (D) at (6, 1.2);
    \coordinate (E) at (2, 1.2);
    \coordinate (A) at (3, 4.5);
    \coordinate (p) at (3.5, 2.25);
    \coordinate (Mleft) at (2.9167, 1.875);
    \coordinate (Mright) at (3.9167, 2.075);
    \coordinate (PAB) at (2.5, 3.75);
    \coordinate (PCB) at (3.3333, 0);
    \coordinate (PAD) at (3.5, 3.95);
    \coordinate (PCD) at (4.3333, 0.2);
    \draw[gray] (PAB) -- (PCB) -- (PCD) -- (PAD) -- cycle;
    \fill[gray, opacity=0.2] (PAB) -- (PCB) -- (PCD) -- (PAD) -- cycle;
    \draw[very thick, red] (Mleft) -- (Mright);
    \draw[red] (Mleft) -- (Mright) -- (p) -- cycle;
    \fill[red, opacity=0.3] (Mleft) -- (Mright) -- (p) -- cycle;
    \draw[thick] (B)--(C)--(D);
    \draw[dashed, thick] (B)--(E)--(D);
    \draw[thick] (B)--(A);
    \draw[thick] (C)--(A);
    \draw[thick] (D)--(A);
    \draw[dashed, thick] (E)--(A);
    \node at (B) [below left] {B};
    \node at (C) [below right] {C};
    \node at (D) [right] {D};
    \node at (E) [left] {E};
    \node at (A) [above] {A};
    \coordinate (P1) at (3, 2.4);
    \coordinate (P2) at (1, 1.2);
    \coordinate (P3) at (6, 2.2);
    \coordinate (P4) at (4, 2.6);
    \fill[red, opacity=0.2] (P1) -- (P2) -- (P3) -- (P4) -- cycle;
    \draw[dotted, thick, red] (P1) -- (P2) -- (P3) -- (P4) -- cycle;
    \fill (p) circle (2pt);
    \node at (p) [right] {$p$};
    \node[red] at (P3) [above] {$N_{p}(\mathrm{AC})$};
  \end{scope}
\end{tikzpicture}
\vspace{1pt}
This figure shows the case when $n=3$ and $i=1$. We consider a point $p \in \mathring{\mathrm{AC}}$.
\end{figure}
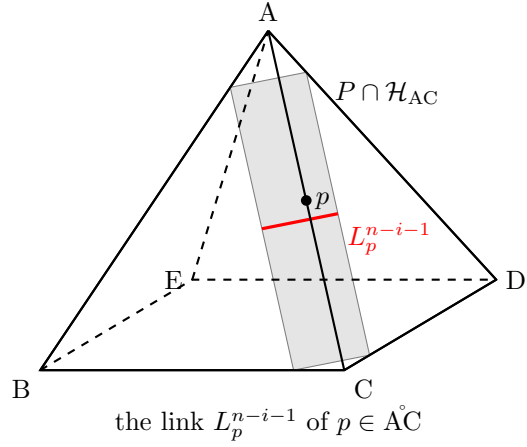
\begin{align*}
\begin{tikzpicture}
  \coordinate (B) at (0, 0);  
  \coordinate (C) at (4, 0);
  \coordinate (D) at (6, 1.2);
  \coordinate (E) at (2, 1.2);
  \coordinate (A) at (3, 4.5);
\coordinate (p) at (3.5, 2.25);
  \coordinate (Mleft) at (2.9167, 1.875);
  \coordinate (Mright) at (3.9167, 2.075);
  \coordinate (PAB) at (2.5, 3.75);
  \coordinate (PCB) at (3.3333, 0);
  \coordinate (PAD) at (3.5, 3.95);
  \coordinate (PCD) at (4.3333, 0.2);
  \draw[thick] (A)--(B)--(C)--(D)--(A);
  \draw[gray] (PAB) -- (PCB) -- (PCD) -- (PAD) -- cycle;
\fill[gray, opacity=0.2] (PAB) -- (PCB) -- (PCD) -- (PAD) -- cycle;
  \draw[very thick, red] (Mleft) -- (Mright);
  \draw[thick] (B)--(C)--(D);
  \draw[dashed, thick] (B)--(E)--(D);
  \draw[thick] (B)--(A);
  \draw[thick] (C)--(A);
  \draw[thick] (D)--(A);
  \draw[dashed, thick] (E)--(A);
  \node at (B) [below left] {B};
  \node at (C) [below right] {C};
  \node at (D) [right] {D};
  \node at (E) [left] {E};
  \node at (A) [above] {A};
\fill (p) circle (2pt);
\node at (R) [right] {$P \cap \mathcal{H}_{\mathrm{AC}}$};
\node at (p) [right] {$p$};
\node[red] at (Mright) [below right] {$L_{p}^{n-i-1}$};
\node at (3, -0.7) {the link $L_{p}^{n-i-1}$ of $p \in \mathring{\mathrm{AC}}$};
\end{tikzpicture}
\end{align*}

To prove Proposition \ref{polytope is pseudomanifold},
we establish the following lemma.

\begin{lem}\label{lem for prop 6.2}
The intersection $P \cap \mathcal{H}_{F^i}^{(+)} \cap N_p (F^i)$ can be viewed as the open cone on $L_p^{n-i-1}$,
where $\mathcal{H}_{F^i}^{(+)}$ denotes the open half-space in $\mathbb{R}^n$ determined by $\mathcal{H}_{F^i}$ that contains the face $F^i$.
\end{lem}

\begin{proof}
Note that $F^i \cap N_p(F^i)=p$.
We put $\mathcal{H}_{F^i}^{[+]}:=\overline{\mathcal{H}_{F^i}^{(+)}}$, i.e., the closed half-space in $\mathbb{R}^n$ determined by $\mathcal{H}_{F^i}$ containing $F^i$. Then it suffices to show the following equality:
\[
\mathrm{conv}(p, L_p^{n-i-1})=P\cap \mathcal{H}_{F^i}^{[+]} \cap N_p(F^i),
\]
where $\mathrm{conv}(p, L_p^{n-i-1})$ denotes the convex hull of $p$ and $L_p^{n-i-1}$.
Since the right-hand side is a convex set, and since both
\[
p\in P\cap \mathcal{H}_{F^i}^{[+]} \cap N_p(F^i) \quad \text{and} \quad L_p^{n-i-1} \subset P\cap \mathcal{H}_{F^i}^{[+]} \cap N_p(F^i),
\]
we obtain
\[
\mathrm{conv}(p, L_p^{n-i-1})\subset P\cap \mathcal{H}_{F^i}^{[+]} \cap N_p(F^i).
\]
Conversely, take a point $q \in P\cap \mathcal{H}_{F^i}^{[+]} \cap N_p(F^i)$.
Suppose for the sake of contradiction that $q \notin \mathrm{conv}(p, L_p^{n-i-1})$.
We represent the convex polytope $P$ using finitely many closed half-spaces as
\[
P = \bigcap_{j \in \mathcal{A}} \mathcal{H}_{j}^{[+]},
\]
where $\mathcal{A}$ is a finite index set.
Since $L_{p}^{n-i-1}$ is a convex polytope with vertices on the boundary of $P$, $\mathrm{conv}(p, L_p^{n-i-1})$ is also a convex polytope.
Let $\mathcal{A}'=\{ j \in \mathcal{A} \mid p \in \mathcal{H}_{j} \}$. Then we can express
\[
\mathrm{conv}(p, L_p^{n-i-1}) = \left( \bigcap_{j \in \mathcal{A}'} \mathcal{H}_{j}^{[+]} \right) \cap \mathcal{H}_{F^i}^{[+]} \cap N_p(F^i),
\]
see the figure below.
Since $q \notin \mathrm{conv}(p, L_p^{n-i-1})$, there exists $k \in \mathcal{A}'$ such that the hyperplane $\mathcal{H}_{k}$ separates $q$ and $L_p^{n-i-1}$, i.e., $q$ and $L_p^{n-i-1}$ lie on opposite sides of $\mathcal{H}_{k}$.
This implies that $q \notin P$, contradicting the assumption that $q \in P$.
Therefore, we must have
\[
\mathrm{conv}(p, L_p^{n-i-1}) \supset P\cap \mathcal{H}_{F^i}^{[+]} \cap N_p(F^i).
\]
\end{proof}

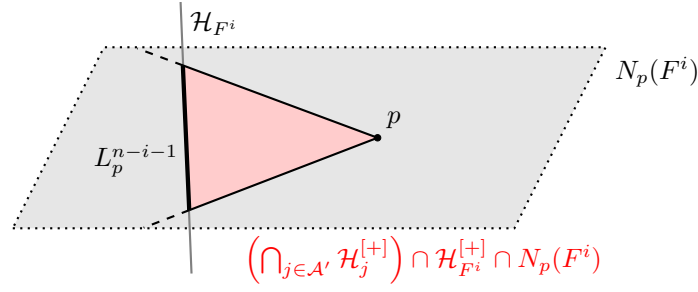
\[
\begin{tikzpicture}[scale=0.6]
  \coordinate (A) at (-1,0);
  \coordinate (B) at (10,0);
  \coordinate (C) at (12,4);
  \coordinate (D) at (1,4);
  \draw[thick, dotted, fill=gray!20] (A) -- (B) -- (C) -- (D) -- cycle;


\draw[gray, thick] (2.92,-1) -- (2.68,5);

\fill[red!20] (2.73,3.6) -- (7,2) -- (2.87,0.4) -- cycle;

  \draw[ultra thick] (2.87,0.4) -- (2.73,3.6);

  \filldraw (7,2) circle (2pt);
  \draw[thick] (2.73,3.6) -- (7,2);
  \draw[thick] (2.87,0.4) -- (7,2);

  \node[above right] at (7,2) {$p$};
  \node at (C) [below right] {$N_p(F^i)$};
  \node at (2.68,5) [below right] {$\mathcal{H}_{F^i}$};

\node[red] at (8,-0.7) {$\left( \bigcap_{j \in \mathcal{A}'} \mathcal{H}_{j}^{[+]} \right) \cap \mathcal{H}_{F^i}^{[+]} \cap N_p(F^i)$};

  \draw[thick, dashed] (2.73,3.6) -- ++(-1,0.3748);  

  \draw[thick, dashed] (2.87,0.4) -- ++(-1,-0.3875); 
\node at (1.7,1.5) {$L_{p}^{n-i-1}$};
\end{tikzpicture}
\]

Now we may prove Proposition \ref{polytope is pseudomanifold}.

\begin{proof}[Proof of Proposition \ref{polytope is pseudomanifold}.]
The conditions in Definition \ref{def of pseudomanifold}-1 and Definition \ref{def of pseudomanifold}-2 are easy to verify as follows.
For $F^i \in \mathcal{F}_{i}(P)$, we have $\mathring{F}^i \cong \mathbb{R}^i$. Thus, each connected component of $Q_i \setminus Q_{i-1}$ is an $i$-dimensional manifold. This shows the first condition.
Since $\mathring{P}$ is dense in $P$, the second condition is satisfied.
We now prove the condition in Definition \ref{def of pseudomanifold}-3 by induction on $n = \dim P$.
For $n=0$, a $0$-dimensional convex polytope is a point. Therefore, this is a topological stratified pseudomanifold.
For $n > 0$, when $p \in \mathring{P}$, we can take an open neighborhood $U_p$ of $p$ in $P$ as an open ball with sufficiently small radius, which satisfies Definition \ref{def of pseudomanifold}-3.
For each $p \in \mathring{F}^i \subsetneq P$, 
we fix a hyperplane $\mathcal{H}_{F^i}$ for $F^i$.
By Construction \ref{cons of link in polytope}, we take the link $L_p^{n-i-1}=P \cap\mathcal{H}_{F^i} \cap N_p (F^i)$ of $p$.
Let $O_p$ be a sufficiently small contractible open neighborhood of $p$ in $\mathring{F}^i$.
By Lemma \ref{lem for prop 6.2}, we may regard
\begin{align}\label{cone in P}
\cone(L_{p}^{n-i-1})=P\cap \mathcal{H}_{F^{i}}^{(+)}\cap N_{p}(F^{i}).
\end{align}
We define
\[
U_p :=
P\cap \mathcal{H}_{F^{i}}^{(+)}\cap \left( \bigcup_{p' \in {O}_p} N_{p'}(F^{i}) \right)
\subset P.
\]
Since $\bigcup_{p' \in {O}_p} N_{p'}(F^{i})$ is an open subset of $\mathbb{R}^n$, the set $U_p$ is open in $P$.
Because 
$\mathcal{H}_{F^{i}}$ is fixed, it follows from (\ref{cone in P}) that
\begin{align}\label{cones p'}
U_p&=
\bigcup_{p' \in {O}_p} \left(
P \cap \mathcal{H}_{F^{i}}^{(+)} \cap N_{p'}(F^{i})
\right)
\notag \\
&=
\bigcup_{p' \in {O}_p} \cone(L_{p'}^{n-i-1}).
\end{align}
Since $\cone(L_{p'}^{n-i-1})$ is obtained by a slight translation of $\cone(L_{p}^{n-i-1})$, it follows that that $\cone(L_{p'}^{n-i-1})$ is homeomorphic to $\cone(L_{p}^{n-i-1})$.
Hence, we have
\[
U_p \cong O_p \times \cone(L_{p}^{n-i-1}).
\]

It remains to show the condition with respect to the restrictions to the skeleton. That is, we will prove that
\[
U_p \cap Q_{i+j+1} \cong O_p \times \cone((L_{p}^{n-i-1})_j).
\]
The polytopal filtration of $L_{p}^{n-i-1}$ is induced by the polytopal filtration of $P$. That is, 
\[
(L_{p}^{n-i-1})_{n-i-1}
=P \cap \mathcal{H}_{F^i} \cap N_{p}(F^i) \supset \cdots \supset 
\bigcap_{F^{i+j+1} \in \mathcal{F}_{i+j+1}(P)}
F^{i+j+1} \cap \mathcal{H}_{F^i} \cap N_{p}(F^i) \supset \cdots.
\]
Since
\[
n-i-1-\mathrm{codim}{F^{i+j+1}}=n-i-1-(n-i-j-1)=j,
\]
we obtain
\begin{align}\label{rest of cone}
\cone
\left(
(L_{p}^{n-i-1})_{j} 
\right) 
&=
\bigcap_{F^{i+j+1} \in \mathcal{F}_{i+j+1}(P)}
F^{i+j+1} \cap \mathcal{H}_{F^i}^{(+)} \cap N_{p}(F^i)
\notag \\
&=
Q_{i+j+1} \cap \left(
P \cap \mathcal{H}_{F^i}^{(+)} \cap N_{p}(F^i) 
\right) \notag \\
&= 
\cone(L_p^{n-i-1}) \cap Q_{i+j+1}.
\end{align}
Therefore, by (\ref{cones p'}) and (\ref{rest of cone}),
we have
\begin{align*}
U_p \cap Q_{i+j+1} 
&=
\bigcup_{p' \in {O}_p} 
\cone(L_{p'}^{n-i-1}) \cap Q_{i+j+1} \\
&=
\bigcup_{p' \in {O}_p} 
\cone
\left(
(L_{p'}^{n-i-1})_{j} 
\right) .
\end{align*}
Since $\cone((L_{p'}^{n-i-1})_{j})$ is obtained by a slight translation of $\cone((L_{p}^{n-i-1})_{j})$, it follows that $\cone((L_{p'}^{n-i-1})_{j})$ is homeomorphic to $\cone((L_{p}^{n-i-1})_{j})$.
Hence, we have
\[
U_p \cap Q_{i+j+1} \cong O_p \times \cone((L_{p}^{n-i-1})_j),
\]
as desired.
\end{proof}

\begin{rem}
Convex polytopes are second-countable and compact. Moreover, since polytopes are contractible, we have $H^2(P) = 0$. Therefore, by Theorem \ref{classification theorem} and Proposition \ref{polytope is pseudomanifold}, we have seen that locally standard $T$-pseudomanifolds over $P$ are classified by characteristic functors on $P$.
\end{rem}

\subsection{Comparison with the characteristic function on a simple polytope}\label{subsec: characteristic function on simple polytope}

The {\it characteristic function} defined in \cite{DJ91} and \cite{PS10} assigns a label (a primitive vector in $\mathbb{Z}^n$) to each facet of an $n$-dimensional simple polytope (also see the unimodular labeling in \cite{KK25} which is the generalization of the characteristic function to arbitrary manifold with corners). A pair $(P, \mu)$ consisting of an $n$-dimensional simple polytope $P$ and a characteristic function $\mu$ is called a characteristic pair.
From a characteristic pair $(P, \mu)$, one can construct the canonical model $X(P, \mu)$.
Quasitoric manifolds introduced in \cite{DJ91} and quasitoric orbifolds studied in \cite{PS10} can be classified up to equivariant homeomorphism by characteristic functions on simple polytopes.
That is, any quasitoric orbifold (manifold) is equivariantly homeomorphic to the canonical model $X(P, \mu)$ constructed from some characteristic pair.
In this section, we show that the characteristic functor defined in Definition \ref{def of char functor} is an extension of the characteristic function (see Construction \ref{from a characteristic function}). Furthermore, as a consequence, we show that locally standard $T$-pseudomanifolds are a generalization of quasitoric orbifolds (manifolds).

\begin{definition}[characteristic function on a simple polytope \cite{PS10}]\label{def of char function}
Let $P$ be an $n$-dimensional simple polytope with facets $\mathcal{F}_{n-1}(P) = \{ F_1, \ldots, F_m \}$.
For each facet $F_i$ of $P$, we assign a primitive vector $\mu_i \in \mathbb{Z}^n$ such that whenever $F_{i_1} \cap \cdots \cap F_{i_k} \neq \emptyset$, the corresponding vectors $\mu_{i_1}, \ldots, \mu_{i_k}$ are linearly independent over $\mathbb{Z}$.
This assignment defines a map
\[
\begin{array}{rccc}
\mu: & \mathcal{F}_{n-1}(P) & \to & \mathbb{Z}^n \\
     & \rotatebox{90}{$\in$} &     & \rotatebox{90}{$\in$} \\
     & F_i & \mapsto & \mu_i
\end{array}
\]
which is called the {\it characteristic function} on $P$.
\end{definition}

We construct a characteristic functor on $P$ from a characteristic function on a simple polytope $P$ as follows.

\begin{cons}\label{from a characteristic function}
Let $P$ be an $n$-dimensional simple polytope and $\mu$ a characteristic function on $P$.
Let $F_j \in \mathcal{F}_{n-1}(P)$ be a facet of an $n$-dimensional simple polytope $P$, and let $\mu(F_j)=\mu_{j}=(\mu_{1j}, \ldots , \mu_{nj}) \in \mathbb{Z}^n$.

For a facet $F_j \in \mathcal{F}_{n-1}(P)$ with $\mu(F_j) = \mu_j = (\mu_{1j}, \ldots, \mu_{nj}) \in \mathbb{Z}^n$,
define the corresponding circle subgroup $T_{F_j}(\cong S^1)$ by
\[
T_{F_j}:= \{ (e^{2 \pi i \mu_{1j}x}, \ldots, e^{2 \pi i \mu_{nj}x}) \in U(1)^n \mid x \in \mathbb{R}  \}.
\]
By Definition \ref{def of char function}, $\mu_j \in \mathbb{Z}^n$ is primitive. Therefore, $T_{F_j}$ is uniquely determined by $\mu_j$ up to sign. This yields the assignment $F_j \mapsto T_{F_j}$.
Now, since $P$ is simple, for any $(n-k)$-dimensional face $F^{n-k} \in \mathcal{F}_{n-k}(P)$, there exist exactly $k$ facets $F_{i_1}, \ldots , F_{i_k}\in \mathcal{F}_{n-1}(P)$ such that $F^{n-k}=F_{i_1} \cap  \ldots \cap F_{i_k}$.
By the definition of the characteristic function, the vectors
$\mu_{i1}, \ldots , \mu_{ik} \in \mathbb{Z}^n$ are linearly independent over $\mathbb{Z}$. Hence, we have
\[
T_{F_{i1}} \times \cdots \times T_{F_{ik}} \cong T^k.
\]
Hence, we define
\[
T_{F^{n-k}}:=T_{F_{i1}} \times \cdots \times T_{F_{ik}}.
\]
Then we obtain a functor
\[
\begin{array}{rccc}
\lambda:&\mathcal{S}(P)^{\mathrm{op}}&\to&\mathcal{T} \\
        & \rotatebox{90}{$\in$}&               & \rotatebox{90}{$\in$} \\
&\mathring{F}^{n-k} &\mapsto&T_{F^{n-k}}
\end{array}
\]
which defines a characteristic functor on $P$.
Therefore, since $P$ is face-acyclic, the characteristic data $(P, \lambda, 0)$ is induced from the characteristic pair $(P, \mu)$.
\end{cons}

According to \cite{PS10}, any quasitoric orbifold is equivariantly homeomorphic to the canonical model $X(P, \lambda):=X(P, \lambda, 0)$ of some characteristic data $(P, \lambda, 0)$, where $P$ is a simple polytope.
Since $X(P, \lambda, 0)$ is a locally standard $T$-pseudomanifold (see Theorem \ref{canonical model is a T-pseudomanifold}), it follows that quasitoric orbifolds are also locally standard $T$-pseudomanifolds.
Moreover, since quasitoric orbifolds generalize quasitoric manifolds introduced in \cite{DJ91}, quasitoric manifolds are also locally standard $T$-pseudomanifolds.

\begin{rem}
The key differences between a characteristic functor and a characteristic function lie in whether the convex polytope must be simple, and whether the labels around each vertex span $\mathbb{Z}^n$.
When the convex polytope is not simple, the linear independence condition in the definition of a characteristic function must be dropped.  
As a result, even when the polytope is simple, a given characteristic functor may not induce a characteristic function.  
Such a situation appears when $k = 0$ in Example~\ref{ex1}, for instance. Furthermore, in this case, the labels around the vertex do not span $\mathbb{Z}^n$.
\end{rem}

\begin{rem}
In \cite{KK25}, locally standard torus actions on smooth manifolds are classified by the data consisting of a smooth manifold with corners, its unimodular labeling, and its Chern class. Using a method similar to the one demonstrated above, one can also define the characteristic functor on a manifold with corners from the unimodular labeling in \cite{KK25}. Hence, our classification also provides the classification of the objects considered in \cite{KK25} up to equivariant homeomorphism. Note that the classification in \cite{KK25} is given up to equivariant diffeomorphism.
\end{rem}


\begin{thebibliography}{20}
\bibitem[Ayz18]{Ayz18} Anton Ayzenberg, {\it "Torus actions of complexity one and their local properties"}. Proceedings of the Steklov Institute of Mathematics (2018), Vol. 302, pp. 16-32.

\bibitem[BH99]{BH99} Martin R. Bridson, Andr\'e Haefliger, "Metric spaces of non-positive curvature". Springer, Volume 319 (1999).

\bibitem[BS25]{BS25} Markus Banagl and Shahryar Ghaed Sharaf, {\it "Link bundles and intersection spaces of complex toric varieties"}. Journal of Singularities volume 28 (2025), 55--103.

\bibitem[BP12]{BP12} Victor Buchstaber, Taras Panov, "Toric Topology". Amer Mathematical Society (2012).

\bibitem[CLS11]{CLS11} David A. Cox, John B. Little, and Henry K. Schenck, "Toric Varieties". Graduate Studies in Mathematics, Vol. 124. American Mathematical Society (2011).

\bibitem[Dav14]{Dav14} Michael W. Davis, {\it "When are two Coxeter orbifolds diffeomorphic?"}. Michigan Math. J. {\bf 63} (2014), 401-421.

\bibitem[DJ91]{DJ91} Michael W. Davis and Tadeusz Januszkiewicz, {\it "Convex polytopes, Coxeter orbifolds and torus actions"}. Duke Math. J. 62 (1991), no. 2, 417-451.

\bibitem[DLS19]{DLS19} Michael W. Davis, Giang Le and Kevin Schreve,
{\it "Action dimensions of some simple complexes of groups"}.
Journal of Topology 12 (2019) 1266-1314.

\bibitem[Fri20]{Fri20} Greg Friedman, "Singular intersection homology". Vol. 33. New Mathematical Monographs. Cambridge University Press, Cambridge (2020).

\bibitem[GM80]{GM80} Mark Goresky and Robert MacPherson, {\it ''Intersection homology theory''}, Topology 19 (1980), no. 2, pp.135--162.

\bibitem[GM83]{GM83} Mark Goresky and Robert MacPherson, {\it "Intersection Homology II"}. In: Invent. Math. 72.1 (1983), pp. 77–129.

\bibitem[Hus94]{Hus94} Dale Husemoller, "Fibre Bundles". Third Edition. Springer Verlag, New York (1994).

\bibitem[Joy12]{Joy12} Dominic D. Joyce, {\it "On manifolds with corners"}. In: Advances in Geometric Analysis. Advanced Lectures in Mathematics (ALM), vol. 21, pp. 225–258. International Press, Somerville, MA (2012).

\bibitem[Jor98]{Jor98} Arno Jordan, “Homology and Cohomology of Toric Varieties”, Konstanzer Schriften in Mathematik und Informatik, Nr. 57 (1998).

\bibitem[KK25]{KK25} Yael Karshon and Shintaro Kuroki, {\it "Classification of locally standard torus actions"}. arXiv:2507.15004 (2025).

\bibitem[LC25]{LC25} Leanprover Community, \textit{Documentation for \texttt{Dense.preimage} in Mathlib4}. Available at \url{https://leanprover-community.github.io/mathlib4_docs/Mathlib/Topology/Maps/Basic.html#Dense.preimage}. Accessed July 2025.

\bibitem[Max19]{Max19} L. G. Maxim, "Intersection Homology and Perverse Sheaves with Applications to Singularities". Graduate Texts in Math. 281, Springer (2019).

\bibitem[MP06]{MP06} Mikiya Masuda and Taras Panov, {\it "On the cohomology of torus manifolds"}. Osaka J. Math. \textbf{43} (2006), no. 3, 711--746.

\bibitem[Mor75]{Mor75} Kiiti Morita, {\it "Čech cohomology and covering dimension for topological spaces"}. Fundamenta Mathematicae 87 (1975), no. 1, 31--52.


\bibitem[Mun14]{Mun14} James Munkres, "Topology". Second Edition. Pearson Education Limited (2014). 

\bibitem[PS10]{PS10} Mainak Poddar, Soumen Sarkar, {\it "On Quasitoric Orbifolds"}. Osaka J. Math. 47 (2010), 1055-1076.

\bibitem[Pop00]{Pop00} Raimund Popper, {\it "Compact Lie groups acting on pseudomanifolds"}. Illinois Journal of Mathematics, Volume 44, Number 1, Spring 2000, 1-19.

\bibitem[RS17]{RS17} Gerd Rudolph and Matthias Schmidt, "Differential Geometry and Mathematical Physics Part II. Fibre Bundles, Topology and Gauge Fields". Springer (2017).

\bibitem[SS21]{SS21}
Soumen Sarkar and Jongbaek Song, {\it "Equivariant cohomological rigidity of certain $T$-manifolds"}. Algebraic and Geometric Topology 21 (2021) 3601-3622.

\bibitem[SS16]{SS16} Soumen Sarkar, Dong Youp Suh, {\it "A new construction of lens spaces"}. Topology and its Applications 240 (2018), 1-20.

\bibitem[Spa66]{Spa66} Edwin H. Spanier, "Algebraic Topology". McGraw-Hill Book Company (1966).

\bibitem[War83]{War83} Frank W. Warner, "Foundations of Differentiable Manifolds and Lie Groups". Springer Verlag, New York (1983).

\bibitem[Wie22]{Wie22} Michael Wiemeler, {\it "Smooth classification of locally standard Tk-manifolds"}, Osaka J. Math. \textbf{59} (2022), no. 3, 549--557.

\bibitem[Wil70]{Wil70} Stephen Willard, "General Topology". Addison-Wesley Publishing Company (1970).

\bibitem[Yos11]{Yos11} Takahiko Yoshida, {\it "Local torus actions modeled on the standard representation"}. Advances in Mathematics, Vol. 227 (2011), no.5, 1914--1955.

\bibitem[Zie96]{Zie96} Günter M. Ziegler, "Lectures on polytopes". Springer-Verlag, GTM 152, (1996).

\end{thebibliography}
\end{document}